\documentclass[11pt,reqno]{amsart}

\usepackage[T1]{fontenc}
\usepackage{palatino}
\usepackage{amsmath,amssymb,amsthm,mathtools}
\usepackage{amscd}
\usepackage[margin=1.1in]{geometry}
\usepackage{float}
\usepackage{hyperref}
\usepackage[nameinlink,capitalize]{cleveref}
\usepackage{tikz}
\usetikzlibrary{intersections}
\usetikzlibrary{decorations.markings}
\usetikzlibrary{calc}

\hypersetup{colorlinks=true,linkcolor=blue,citecolor=blue,urlcolor=blue}

\theoremstyle{plain}
\newtheorem{theorem}{Theorem}[section]
\newtheorem{proposition}[theorem]{Proposition}
\newtheorem{corollary}[theorem]{Corollary}
\newtheorem{lemma}[theorem]{Lemma}

\theoremstyle{definition}
\newtheorem{remark}[theorem]{Remark}
\newtheorem{example}[theorem]{Example}

\newcommand{\Mod}{\operatorname{Mod}}

\newcommand{\ob}{M}
\newcommand{\ncl}[1]{N(#1)}
\newcommand{\Z}{\mathbb Z}

\title[Contact 3--manifolds with infinitely many rational homology ball Stein fillings]
{Contact $3$--manifolds with infinitely many\\ rational homology ball Stein fillings}

\author{Burak Ozbagci}


\usepackage{placeins}
\begin{document}
\begin{abstract}
We construct the first examples of contact $3$--manifolds admitting infinitely
many pairwise non-homeomorphic Stein fillings which are rational homology balls.
Our examples are planar.  They are obtained by a simple modification of Baykur's
recent planar contact $3$--manifold $(Y,\xi)$ with infinitely many pairwise
non-homeomorphic Stein fillings, and the fillings are distinguished by their
fundamental groups.  For one of these rational homology balls we give an explicit
Kirby diagram, with a single $1$--handle and a single $2$--handle attached along a
trefoil, which exhibits its boundary as $(0,-8)$--surgery on a two-component link.
We also identify the smooth topology of Baykur's example:
$Y$ is obtained by $(0,-8)$--surgery on the link $7^2_4$, and it is Seifert
fibered over the Klein bottle with one exceptional fiber.  Finally, we construct
further Stein fillings of $(Y,\xi)$, which realize all the groups
$\mathbb Z/m\rtimes_{-1}\mathbb Z$, $m\ge0$, as their fundamental groups.
\end{abstract}

\maketitle
\section{Introduction}

A contact $3$--manifold is said to be \emph{planar} if it is
supported by a planar open book---an open book whose pages have genus zero. Wendl showed that if a contact
$3$--manifold  is supported by a planar open book $(\Sigma,\phi)$, then every
minimal strong symplectic filling of it, after attaching a trivial symplectic
cobordism to its boundary, admits an allowable Lefschetz fibration over $D^2$
with fiber $\Sigma$ and monodromy $\phi$, inducing the given open book on the
boundary \cite{Wendl}.  As a consequence, the Stein fillings of a planar contact $3$--manifold are governed by the positive factorizations of the monodromy of any planar open book 
supporting the contact $3$--manifold at hand. 

Contact $3$--manifolds with infinitely many pairwise non-homeomorphic Stein
fillings were first constructed by Stipsicz and the author \cite{OS}, more than twenty years ago.  Baykur
\cite{BaykurPlanar} recently constructed planar examples, answering
Problem~4.105 of the K3 problem list \cite{K3}.  His basic example
$(Y,\xi)$ is supported by the open book
\[
(\Sigma_0^6,\phi),\qquad \phi=t_d\,B,\qquad B=t_ut_vt_wt_z,
\]
whose page $\Sigma_0^6$ is a sphere with six holes, and  $t_c$ denotes the
right-handed Dehn twist along a simple closed curve $c$ on the page. The curves 
$u,v,w,z,d \subset \Sigma_0^6$ are depicted below in
\cref{fig:baykur-curves}.  A crucial observation of Baykur is that the curve $\alpha$ of \cref{fig:alpha-beta-x}
satisfies $[B,t_\alpha]=1$.  The Stein fillings $X_n$, for $n\ge0$, are given by  the Lefschetz
fibrations corresponding to  the positive factorizations
$t_d\,t_\alpha^{\,n}B\,t_\alpha^{-n}$ of $\phi$, and Baykur distinguishes them
by their fundamental groups $\mathbb Z/(2n+1)\rtimes_{-1}\mathbb Z$.  Here and
below, for cyclic groups $A$ and $B$, we write $A\rtimes_{-1}B$ for the
semidirect product in which a generator of $B$ acts on $A$ by inversion.

Throughout, for an allowable positive factorization $\mu=t_{c_1}\cdots t_{c_n}$ (with every $c_i$ homologically nontrivial) of a
mapping class of a compact surface $\Sigma$ with boundary, $M(\Sigma, \mu)$ denotes the $3$--manifold given by the open book with page $\Sigma$ and monodromy $\mu$, and 
$X(\Sigma;c_1,\ldots,c_n)$ denotes the total space of the Lefschetz
fibration over the disk with fiber $\Sigma$ and vanishing cycles
$c_1,\ldots,c_n$, in the order of the factors, equipped with the Stein
structure provided by \cite{AO,LP}. Note that $X(\Sigma;c_1,\ldots,c_n)$ is 
a Stein filling of the contact manifold supported by the open book
$(\Sigma,\mu)$.

\smallskip
\noindent\emph{Rational homology ball fillings.}  Our main result is obtained by
varying Baykur's construction.  Recall that $\alpha$ is the round curve enclosing
the holes $\delta_1,\delta_2$ (see \cref{fig:alpha-beta-x}).  For a homologically nontrivial simple closed curve $c$ on the page,
let $(Y_c,\xi_c)$ be the planar contact $3$--manifold supported by the open book
$(\Sigma_0^6,t_cB)$, and let
\[
X_n(c)=X(\Sigma_0^6;c,\,t_\alpha^{\,n}(u),\,t_\alpha^{\,n}(v),\,t_\alpha^{\,n}(w),\,t_\alpha^{\,n}(z)),\qquad n\ge0 ,
\]
which are Stein fillings of $(Y_c,\xi_c)$, since $t_\alpha$ commutes with $B$.
For $c=d$ these are Baykur's $(Y,\xi)$ and $X_n$.  The fundamental groups of the fillings $X_n(c)$ can be computed
directly from a word representing $c$, and we do this for two natural families
of curves.  Let $c_{12}$ and $c_{345}$ be the round curves enclosing the holes
$\delta_1,\delta_2$ and $\delta_3,\delta_4,\delta_5$, respectively, and recall
that the vanishing cycle $u$ is the round curve enclosing $\delta_2,\delta_3$.
For a positive integer $m$ we set
\[
d_m=t_u^m(c_{12}),\qquad d'_m=t_u^m(c_{345}).
\]
The simplest members are $d_1=t_u(c_{12})$,  and
$d'_1=t_u(c_{345})$ (see \cref{fig:di}).  We write $(Y_m,\xi_m)$ and $(Y'_m,\xi'_m)$ for the
contact $3$--manifolds $(Y_c,\xi_c)$ with $c=d_m$ and $c=d'_m$, and $X_{m,n}$ and
$X'_{m,n}$ for their Stein fillings $X_n(c)$.

\begin{theorem}\label[theorem]{thm:ball}
Let $m$ be a positive integer.  Then the following hold.
\begin{enumerate}
\item The planar contact $3$--manifolds $(Y_m,\xi_m)$ and $(Y'_m,\xi'_m)$ are
rational homology spheres with
\[
H_1(Y_m;\Z)\cong\Z/4\oplus\Z/4,\qquad
H_1(Y'_m;\Z)\cong\Z/2\oplus\Z/2\oplus\Z/4 .
\]
In particular, no $Y_m$ is homeomorphic to any $Y'_{m'}$.
\item The Stein fillings $X_{m,n}$ of $(Y_m,\xi_m)$, $n\ge0$, are rational
homology balls with
\[
\pi_1(X_{m,n})\cong\Z/(4mn+1)\rtimes_{-1}\Z/4,\qquad H_1(X_{m,n};\Z)\cong\Z/4 .
\]
\item The Stein fillings $X'_{m,n}$ of $(Y'_m,\xi'_m)$, $n\ge0$, are rational
homology balls with
\[
\pi_1(X'_{m,n})\cong\langle r,b\mid brb^{-1}=r^{-1},\ b^2=r^{4mn}\rangle,\qquad
H_1(X'_{m,n};\Z)\cong\Z/2\oplus\Z/2 .
\]
For $n\ge1$ the group $\pi_1(X'_{m,n})$ is the dicyclic group of order $16mn$,
and for $n=0$ it is the infinite dihedral group $\Z\rtimes_{-1}\Z/2$.
\item The Stein fillings of each of these contact $3$--manifolds are pairwise
non-homeomorphic.
\end{enumerate}
\end{theorem}

The fillings in each family have the same homology, so they are distinguished
only by their fundamental groups.  The orders of the first homology groups fit
together as they must: $\lvert H_1(Y_m;\Z)\rvert=\lvert H_1(Y'_m;\Z)\rvert=16$
is the square of $\lvert H_1(X_{m,n};\Z)\rvert=\lvert H_1(X'_{m,n};\Z)\rvert=4$, in
accordance with the fact that $\lvert H_1(\partial W;\Z)\rvert=\lvert
H_1(W;\Z)\rvert^2$ for every rational homology ball $W$ with $H_2(W;\Z)=0$, for
example a Stein one.

Symplectic rational homology ball fillings of contact $3$--manifolds have been
studied extensively, see for example \cite{BhupalStipsicz,EOT,EOTlegs}.  For
instance, a closed, oriented $3$--manifold with finite fundamental group carries
at most six contact structures, up to isotopy, which are symplectically fillable
by rational homology balls \cite{EOTlegs}.  

To our knowledge,
$(Y_m,\xi_m)$ and $(Y'_m,\xi'_m)$ are the first contact
$3$--manifolds known to have infinitely many pairwise non-homeomorphic {\em rational
homology ball symplectic} fillings.  

Unlike Baykur's manifold $Y$, which is Seifert fibered over the Klein
bottle with $b_1(Y)=1$, the manifolds $Y_m$ and $Y'_m$ are rational homology
spheres with infinite fundamental groups.  Indeed, a Stein filling $X$ has a
handle decomposition with one $0$--handle and only $1$-- and $2$--handles.  Turned
upside down, $X$ is obtained from $\partial X\times[0,1]$ by attaching $2$--,
$3$-- and $4$--handles, so the inclusion induces a surjection
$\pi_1(\partial X)\to\pi_1(X)$.  Hence $\pi_1(Y_m)$ surjects onto the groups
$\pi_1(X_{m,n})$, $n\ge0$, whose orders $4(4mn+1)$ are unbounded, and
$\pi_1(Y'_m)$ surjects onto the infinite dihedral group $\pi_1(X'_{m,0})$.  More
generally, the same argument shows that a contact $3$--manifold with finite
fundamental group cannot have Stein fillings whose fundamental groups have
unbounded orders.  Finally, for one of these rational homology balls, $X_{1,0}$, we give an
explicit Kirby diagram, with a single $1$--handle and a single $2$--handle,
attached with framing $-8$ along a trefoil which runs four times over the
$1$--handle (\cref{fig:X0d1-trefoil}).

\smallskip
\noindent\emph{Baykur's example.}  Baykur does not identify the underlying
$3$--manifold $Y$.  We identify it, using an intermediate manifold: the
Lefschetz fibration $X_B=X(\Sigma_0^6;u,v,w,z)$ obtained by deleting the vanishing
cycle $d$, whose boundary $Y_B=M(\Sigma_0^6,B)$ is supported by the open book
$(\Sigma_0^6,B)$.  Baykur's fillings are
$X_n=X(\Sigma_0^6;d,\,t_\alpha^{\,n}(u),\,t_\alpha^{\,n}(v),\,t_\alpha^{\,n}(w),\,t_\alpha^{\,n}(z))$,
and since $t_\alpha^n$ carries $u,v,w,z$ to $t_\alpha^n(u),\ldots,t_\alpha^n(z)$,
every $X_n$ is obtained from a copy of $X_B$ by attaching a single $2$--handle (see \cref{fig:X0-DTK} for $n=0$).
In the following, $K$ denotes the Klein bottle, $DT^*K$ its disk cotangent
bundle, and $ST^*K$ its unit cotangent bundle.
\begin{theorem}\label[theorem]{thm:main}
\leavevmode
\begin{enumerate}
\item The manifold $Y$ is obtained by $(0,-8)$--surgery on the two-component
link $7^2_4$, with coefficient $0$ on the unknotted component and $-8$ on the
left-handed trefoil component.  It is Seifert fibered over the Klein bottle
with one exceptional fiber of type $(2,1)$.
\item $X_B$ is diffeomorphic to $DT^*K$, and $Y_B$ is
orientation-preservingly diffeomorphic to $ST^*K$.  Consequently, every Stein filling $X_n$ in Baykur's family is diffeomorphic to $DT^*K$ with a single $2$--handle
attached, and $Y$ is obtained from $ST^*K$ by an integral surgery along a knot.
\end{enumerate}
\end{theorem}

We make a few comments on the statement and its proof.

\smallskip
\noindent\emph{The smooth topology of $Y$.}  Part~(1) is proved by Kirby
calculus.  The surgery description comes from the diagram of Baykur's
Lefschetz fibration, which reduces to a two-component surgery diagram.  For
the Seifert fibration we use instead one of the new fillings constructed
below, the Stein filling $W_0$ with $\pi_1(W_0)\cong\pi_1(K)$. Handle slides
show that $W_0$ is diffeomorphic to $DT^*K$ with a single $2$--handle
attached along a circle fiber of $ST^*K\to K$, with framing $-2$ relative to
the framing given by the bundle (\cref{prop:W0}).  Hence $Y$ is obtained from
$ST^*K$ by Dehn surgery along a regular fiber, and the Seifert fibration of
$ST^*K$ extends over the surgery torus with one exceptional fiber of
multiplicity two.

\smallskip
\noindent\emph{The disk cotangent bundle.}  Part~(2) is proved by Kirby
calculus (\cref{thm:klein-main}).  

\smallskip
\noindent\emph{New fillings.}  Besides $t_\alpha$, two half twists, each
interchanging a pair of holes, also commute with $B$, by Baykur's generalized lantern
relation, and conjugating $B$ by them while keeping $t_d$ fixed produces
positive factorizations of $\phi$ of a new homological type.  For $m\ge0$,
the group $\mathbb Z/m\rtimes_{-1}\mathbb Z$ is the Klein bottle group
$\pi_1(K)$ for $m=0$, and it is $\mathbb Z$ for $m=1$.
Baykur's fillings realize the odd values of $m$.

\begin{theorem}\label[theorem]{thm:newfillings}
For every integer $m\ge0$ the contact $3$--manifold $(Y,\xi)$ admits a Stein filling $W_m$ with
\[
\pi_1(W_m)\cong\mathbb Z/m\rtimes_{-1}\mathbb Z,\qquad \chi(W_m)=1,\qquad
H_1(W_m;\mathbb Z)\cong\begin{cases}\mathbb Z & m\ \text{odd},\\ \mathbb Z\oplus\mathbb Z/2 & m\ \text{even}.\end{cases}
\]
The fillings $W_m$ are pairwise non-homeomorphic.  For $m$ odd, $W_m$ is Baykur's filling $X_{(m-1)/2}$.
\end{theorem}

Moreover, every Stein filling of $(Y,\xi)$ has first homology $\Z$ or
$\Z\oplus\Z/2$ (\cref{prop:allfillings}), so the fillings $W_m$ realize both
possibilities.

The new fillings are again diffeomorphic to $DT^*K$ with a single $2$--handle
attached: each of them is obtained by attaching a $2$--handle along $d$ to a
Lefschetz fibration $X(\Sigma_0^6;f(u),f(v),f(w),f(z))$, where $f$ is a diffeomorphism
of the page commuting with $B$, and $f$ identifies this Lefschetz fibration with
$X_B$.
For $m=0$ the attaching curve is trivial in $\pi_1(DT^*K)$, exactly as the
circle fiber of $ST^*K\to K$ is, and indeed the filling $W_0$ is
diffeomorphic to $DT^*K$ with a $2$--handle attached along a circle fiber
(\cref{prop:W0}).

\smallskip
\noindent\emph{Outline.}  In \cref{sec:baykur} we prove
\cref{prop:baykur-surgery}, the surgery description of $Y$, and state
\cref{prop:baykur-seifert}, the Seifert fibration of $Y$.  Together they give
\cref{thm:main}(1).  In \cref{sec:klein} we prove, by Kirby calculus, that
$X_B$ is diffeomorphic to $DT^*K$, which gives \cref{thm:main}(2).  In
\cref{sec:newfillings} we construct the fillings $W_m$ and prove
\cref{thm:newfillings}, and in \cref{subsec:W0} we show that $W_0$ is
diffeomorphic to $DT^*K$ with a $2$--handle attached along a circle fiber, which
proves \cref{prop:baykur-seifert}.  In \cref{sec:vary} we prove \cref{thm:ball},
and we describe the filling $X_{1,0}$ by explicit Kirby diagrams.  Finally, in
\cref{sec:computer} we collect the computations which we carried out by computer
and which are not used in any proof.

\medskip
\noindent\emph{Acknowledgments.}  The author thanks Marc Kegel for pointing out
that the identification of Baykur's manifold as a Seifert fibered space can be
confirmed with SnapPy and Regina (\cref{subsec:computer-Y}). 

\medskip
\noindent\emph{Declaration on the use of AI.}  In the preparation of this paper
the author used LLM tools for writing the scripts for
the computer checks in \cref{sec:computer} and in the ancillary files, for
drawing the figures in TikZ from the author's Kirby diagrams, and for editing the
text to improve the presentation.  The author takes full responsibility for the
content of the paper.

\section{The smooth topology of Baykur's manifold}\label{sec:baykur}

Baykur \cite{BaykurPlanar} recently constructed a planar contact $3$--manifold
$(Y,\xi)$ supported by the open book with page the six-holed sphere 
$\Sigma_0^6$ and monodromy $\phi=t_d t_u t_v t_w t_z \in \Mod(\Sigma_0^6,\partial\Sigma_0^6)$, 
and proved that $(Y,\xi)$ has infinitely many pairwise non-homeomorphic
Stein fillings.  The page $\Sigma_0^6$ is drawn as a disk with outer
boundary $\delta_0$ and five holes $\delta_1,\ldots,\delta_5$ ordered from
left to right, and the five curves $u,v,w,z,d$ are those of
\cite[Figure~2]{BaykurPlanar}, reproduced in \cref{fig:baykur-curves}.
We identify the underlying $3$--manifold $Y$,
first by an explicit surgery diagram and then as a Seifert fibered space.
The proof of the second identification uses one of the Stein fillings of
\cref{sec:newfillings} and is given in \cref{subsec:W0}.
These identifications do not appear in \cite{BaykurPlanar}.

\begin{figure}[ht]
\centering
\makebox[\textwidth][c]{%
\begin{tikzpicture}[scale=0.3]
\begin{scope}[shift={(0.0,0.0)}]
\draw[thick,rounded corners=10pt] (0,-2.75) rectangle (15.600000000000001,2.75);
\draw[thick] (2.60,0.00) circle (0.35);
\draw[thick] (5.20,0.00) circle (0.35);
\draw[thick] (7.80,0.00) circle (0.35);
\draw[thick] (10.40,0.00) circle (0.35);
\draw[thick] (13.00,0.00) circle (0.35);
\draw[orange!85!black,thick] plot[smooth cycle,tension=0.7] coordinates{(5.14,-0.69) (5.05,-0.68) (4.97,-0.66) (4.90,-0.63) (4.83,-0.59) (4.77,-0.55) (4.72,-0.50) (4.67,-0.45) (4.62,-0.39) (4.59,-0.33) (4.55,-0.26) (4.53,-0.19) (4.51,-0.10) (4.50,-0.01) (4.51,0.08) (4.52,0.16) (4.55,0.24) (4.58,0.31) (4.61,0.38) (4.65,0.44) (4.70,0.49) (4.76,0.54) (4.81,0.58) (4.88,0.62) (4.95,0.65) (5.03,0.68) (5.12,0.69) (5.21,0.70) (5.31,0.70) (5.42,0.70) (5.52,0.70) (5.63,0.70) (5.73,0.70) (5.83,0.70) (5.94,0.70) (6.04,0.70) (6.14,0.70) (6.25,0.70) (6.35,0.70) (6.46,0.70) (6.56,0.70) (6.66,0.70) (6.77,0.70) (6.87,0.70) (6.97,0.70) (7.08,0.70) (7.18,0.70) (7.29,0.70) (7.39,0.70) (7.49,0.70) (7.60,0.70) (7.70,0.70) (7.80,0.70) (7.90,0.69) (7.98,0.67) (8.06,0.65) (8.13,0.61) (8.19,0.58) (8.25,0.53) (8.30,0.48) (8.35,0.43) (8.39,0.37) (8.43,0.30) (8.46,0.23) (8.48,0.15) (8.49,0.06) (8.50,-0.03) (8.49,-0.13) (8.47,-0.21) (8.44,-0.28) (8.40,-0.35) (8.36,-0.41) (8.32,-0.47) (8.27,-0.52) (8.21,-0.56) (8.15,-0.60) (8.08,-0.64) (8.01,-0.67) (7.92,-0.69) (7.83,-0.70) (7.73,-0.70) (7.63,-0.70) (7.52,-0.70) (7.42,-0.70) (7.32,-0.70) (7.21,-0.70) (7.11,-0.70) (7.01,-0.70) (6.90,-0.70) (6.80,-0.70) (6.69,-0.70) (6.59,-0.70) (6.49,-0.70) (6.38,-0.70) (6.28,-0.70) (6.18,-0.70) (6.07,-0.70) (5.97,-0.70) (5.86,-0.70) (5.76,-0.70) (5.66,-0.70) (5.55,-0.70) (5.45,-0.70) (5.34,-0.70) (5.24,-0.70) (5.15,-0.70)};
\node[orange!85!black] at (6.50,1.55) {$u$};
\end{scope}
\begin{scope}[shift={(16.8,0.0)}]
\draw[thick,rounded corners=10pt] (0,-2.75) rectangle (15.600000000000001,2.75);
\draw[thick] (2.60,0.00) circle (0.35);
\draw[thick] (5.20,0.00) circle (0.35);
\draw[thick] (7.80,0.00) circle (0.35);
\draw[thick] (10.40,0.00) circle (0.35);
\draw[thick] (13.00,0.00) circle (0.35);
\draw[violet,thick] plot[smooth cycle,tension=0.7] coordinates{(5.17,-2.35) (4.94,-2.30) (4.77,-2.20) (4.63,-2.06) (4.54,-1.88) (4.50,-1.65) (4.50,-1.38) (4.50,-1.11) (4.50,-0.84) (4.50,-0.57) (4.50,-0.30) (4.50,-0.03) (4.53,0.21) (4.62,0.39) (4.76,0.54) (4.93,0.64) (5.15,0.69) (5.40,0.67) (5.58,0.58) (5.73,0.45) (5.84,0.28) (5.89,0.07) (5.90,-0.20) (5.90,-0.47) (5.90,-0.74) (5.97,-0.94) (6.24,-0.95) (6.52,-0.95) (6.80,-0.95) (7.08,-0.95) (7.36,-0.95) (7.64,-0.95) (7.92,-0.95) (8.20,-0.95) (8.48,-0.95) (8.76,-0.95) (9.04,-0.95) (9.32,-0.95) (9.60,-0.95) (9.70,-0.78) (9.70,-0.51) (9.70,-0.24) (9.70,0.03) (9.75,0.26) (9.85,0.43) (9.99,0.57) (10.17,0.66) (10.41,0.70) (10.69,0.70) (10.97,0.70) (11.25,0.70) (11.53,0.70) (11.81,0.70) (12.09,0.70) (12.37,0.70) (12.65,0.70) (12.93,0.70) (13.18,0.67) (13.38,0.59) (13.52,0.46) (13.63,0.29) (13.69,0.08) (13.68,-0.17) (13.60,-0.36) (13.47,-0.51) (13.31,-0.62) (13.10,-0.69) (12.83,-0.70) (12.55,-0.70) (12.27,-0.70) (11.99,-0.70) (11.71,-0.70) (11.43,-0.70) (11.16,-0.71) (11.10,-0.93) (11.10,-1.20) (11.10,-1.47) (11.09,-1.73) (11.03,-1.94) (10.93,-2.11) (10.78,-2.24) (10.59,-2.32) (10.34,-2.35) (10.06,-2.35) (9.78,-2.35) (9.50,-2.35) (9.22,-2.35) (8.94,-2.35) (8.66,-2.35) (8.38,-2.35) (8.10,-2.35) (7.82,-2.35) (7.54,-2.35) (7.25,-2.35) (6.97,-2.35) (6.69,-2.35) (6.41,-2.35) (6.13,-2.35) (5.85,-2.35) (5.57,-2.35) (5.29,-2.35)};
\node[violet] at (11.70,1.55) {$v$};
\end{scope}
\begin{scope}[shift={(33.6,0.0)}]
\draw[thick,rounded corners=10pt] (0,-2.75) rectangle (15.600000000000001,2.75);
\draw[thick] (2.60,0.00) circle (0.35);
\draw[thick] (5.20,0.00) circle (0.35);
\draw[thick] (7.80,0.00) circle (0.35);
\draw[thick] (10.40,0.00) circle (0.35);
\draw[thick] (13.00,0.00) circle (0.35);
\draw[green!55!black,thick] plot[smooth cycle,tension=0.7] coordinates{(2.53,-0.69) (2.25,-0.60) (2.05,-0.43) (1.92,-0.17) (1.90,0.18) (1.90,0.55) (1.90,0.92) (1.90,1.29) (1.90,1.66) (1.97,1.96) (2.14,2.17) (2.38,2.31) (2.70,2.35) (3.00,2.35) (3.32,2.35) (3.62,2.35) (3.94,2.35) (4.24,2.35) (4.55,2.35) (4.87,2.35) (5.17,2.35) (5.48,2.35) (5.79,2.35) (6.10,2.35) (6.41,2.35) (6.72,2.35) (7.03,2.35) (7.34,2.35) (7.65,2.35) (7.96,2.35) (8.27,2.35) (8.58,2.35) (8.89,2.35) (9.20,2.35) (9.51,2.35) (9.82,2.35) (10.13,2.35) (10.44,2.35) (10.75,2.35) (11.06,2.35) (11.37,2.35) (11.68,2.35) (11.99,2.35) (12.30,2.35) (12.61,2.35) (12.92,2.35) (13.24,2.31) (13.47,2.16) (13.63,1.95) (13.70,1.64) (13.70,1.27) (13.70,0.90) (13.70,0.53) (13.70,0.16) (13.67,-0.19) (13.54,-0.44) (13.34,-0.61) (13.04,-0.70) (12.73,-0.64) (12.50,-0.49) (12.35,-0.26) (12.30,0.06) (12.30,0.43) (12.30,0.80) (12.07,0.95) (11.69,0.95) (11.30,0.95) (10.92,0.95) (10.54,0.95) (10.15,0.95) (9.77,0.95) (9.38,0.95) (9.00,0.95) (8.62,0.95) (8.50,0.70) (8.50,0.33) (8.50,-0.04) (8.41,-0.34) (8.24,-0.54) (7.99,-0.67) (7.66,-0.68) (7.39,-0.57) (7.21,-0.37) (7.11,-0.10) (7.10,0.27) (7.10,0.64) (7.04,0.94) (6.66,0.95) (6.28,0.95) (5.89,0.95) (5.51,0.95) (5.13,0.95) (4.74,0.95) (4.36,0.95) (3.97,0.95) (3.59,0.95) (3.30,0.85) (3.30,0.49) (3.30,0.12) (3.26,-0.22) (3.13,-0.46) (2.92,-0.62) (2.61,-0.70)};
\node[green!55!black] at (7.80,-1.50) {$w$};
\end{scope}
\begin{scope}[shift={(8.4,-6.3)}]
\draw[thick,rounded corners=10pt] (0,-2.75) rectangle (15.600000000000001,2.75);
\draw[thick] (2.60,0.00) circle (0.35);
\draw[thick] (5.20,0.00) circle (0.35);
\draw[thick] (7.80,0.00) circle (0.35);
\draw[thick] (10.40,0.00) circle (0.35);
\draw[thick] (13.00,0.00) circle (0.35);
\draw[blue!70!black,thick] plot[smooth cycle,tension=0.7] coordinates{(6.48,-2.35) (6.21,-2.28) (6.01,-2.14) (5.87,-1.95) (5.80,-1.68) (5.80,-1.36) (5.80,-1.03) (5.80,-0.70) (5.80,-0.37) (5.80,-0.04) (5.80,0.30) (5.80,0.63) (5.76,0.92) (5.45,0.95) (5.11,0.95) (4.76,0.95) (4.42,0.95) (4.08,0.95) (3.74,0.95) (3.40,0.95) (3.30,0.71) (3.30,0.38) (3.30,0.05) (3.26,-0.24) (3.13,-0.45) (2.95,-0.60) (2.70,-0.69) (2.40,-0.67) (2.18,-0.55) (2.01,-0.38) (1.92,-0.14) (1.90,0.18) (1.90,0.51) (1.90,0.84) (1.90,1.17) (1.90,1.50) (1.92,1.81) (2.02,2.04) (2.19,2.21) (2.42,2.32) (2.73,2.35) (3.07,2.35) (3.41,2.35) (3.76,2.35) (4.10,2.35) (4.44,2.35) (4.78,2.35) (5.13,2.35) (5.47,2.35) (5.81,2.35) (6.15,2.35) (6.49,2.35) (6.77,2.29) (6.98,2.16) (7.12,1.97) (7.19,1.70) (7.20,1.38) (7.20,1.05) (7.20,0.72) (7.20,0.39) (7.20,0.06) (7.20,-0.28) (7.20,-0.61) (7.22,-0.91) (7.52,-0.95) (7.87,-0.95) (8.21,-0.95) (8.55,-0.95) (8.90,-0.95) (9.24,-0.95) (9.58,-0.95) (9.70,-0.74) (9.70,-0.41) (9.70,-0.08) (9.74,0.22) (9.86,0.44) (10.03,0.59) (10.28,0.69) (10.58,0.67) (10.81,0.57) (10.98,0.39) (11.08,0.16) (11.10,-0.15) (11.10,-0.48) (11.10,-0.81) (11.10,-1.14) (11.10,-1.47) (11.08,-1.78) (10.99,-2.02) (10.83,-2.20) (10.61,-2.32) (10.31,-2.35) (9.96,-2.35) (9.62,-2.35) (9.28,-2.35) (8.94,-2.35) (8.59,-2.35) (8.25,-2.35) (7.91,-2.35) (7.57,-2.35) (7.22,-2.35) (6.88,-2.35) (6.54,-2.35)};
\node[blue!70!black] at (10.40,1.50) {$z$};
\end{scope}
\begin{scope}[shift={(25.2,-6.3)}]
\draw[thick,rounded corners=10pt] (0,-2.75) rectangle (15.600000000000001,2.75);
\draw[thick] (2.60,0.00) circle (0.35);
\draw[thick] (5.20,0.00) circle (0.35);
\draw[thick] (7.80,0.00) circle (0.35);
\draw[thick] (10.40,0.00) circle (0.35);
\draw[thick] (13.00,0.00) circle (0.35);
\draw[teal!80!black,thick] plot[smooth cycle,tension=0.7] coordinates{(2.57,-2.35) (2.39,-2.31) (2.23,-2.24) (2.11,-2.14) (2.01,-2.02) (1.94,-1.87) (1.90,-1.68) (1.90,-1.46) (1.90,-1.24) (1.90,-1.02) (1.90,-0.80) (1.90,-0.58) (1.90,-0.36) (1.90,-0.14) (1.91,0.08) (1.95,0.26) (2.03,0.40) (2.14,0.52) (2.27,0.61) (2.43,0.68) (2.63,0.70) (2.82,0.66) (2.98,0.59) (3.10,0.49) (3.20,0.36) (3.27,0.21) (3.30,0.02) (3.30,-0.20) (3.30,-0.42) (3.30,-0.64) (3.30,-0.85) (3.43,-0.95) (3.66,-0.95) (3.89,-0.95) (4.12,-0.95) (4.35,-0.95) (4.58,-0.95) (4.80,-0.95) (5.03,-0.95) (5.26,-0.95) (5.49,-0.95) (5.72,-0.95) (5.95,-0.95) (6.18,-0.95) (6.40,-0.95) (6.63,-0.95) (6.86,-0.95) (7.06,-0.93) (7.10,-0.75) (7.10,-0.53) (7.10,-0.31) (7.10,-0.09) (7.11,0.13) (7.17,0.29) (7.25,0.43) (7.36,0.55) (7.50,0.63) (7.68,0.69) (7.88,0.69) (8.06,0.65) (8.21,0.57) (8.32,0.46) (8.42,0.33) (8.48,0.17) (8.50,-0.04) (8.50,-0.26) (8.50,-0.48) (8.50,-0.70) (8.50,-0.92) (8.50,-1.14) (8.50,-1.36) (8.50,-1.58) (8.48,-1.79) (8.43,-1.95) (8.34,-2.09) (8.23,-2.20) (8.09,-2.29) (7.91,-2.34) (7.70,-2.35) (7.47,-2.35) (7.24,-2.35) (7.02,-2.35) (6.79,-2.35) (6.56,-2.35) (6.33,-2.35) (6.10,-2.35) (5.87,-2.35) (5.65,-2.35) (5.42,-2.35) (5.19,-2.35) (4.96,-2.35) (4.73,-2.35) (4.50,-2.35) (4.28,-2.35) (4.05,-2.35) (3.82,-2.35) (3.59,-2.35) (3.36,-2.35) (3.13,-2.35) (2.91,-2.35) (2.68,-2.35)};
\node[teal!80!black] at (7.80,1.50) {$d$};
\end{scope}
\end{tikzpicture}}
\caption{Baykur's curves $u,v,w,z,d$ on the page $\Sigma_0^6$.}
\label{fig:baykur-curves}
\end{figure}
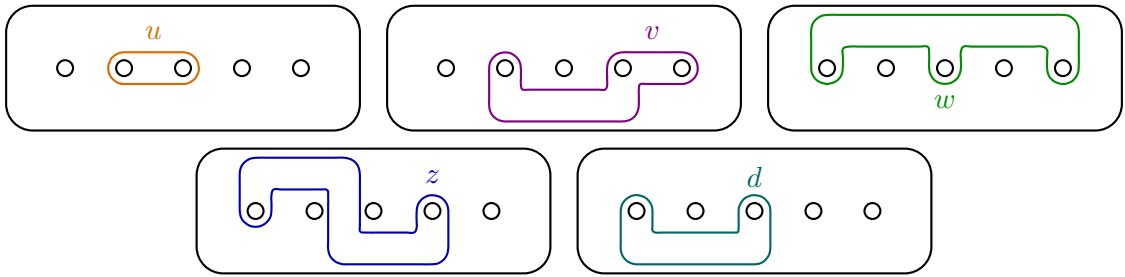

\begin{proposition}\label[proposition]{prop:baykur-surgery}
Baykur's manifold $Y$ is the result of surgery on the two-component link
$L=7^2_4$ shown on the right of \cref{fig:final-kirby}, with coefficient $0$ on the unknotted
component $U$ and coefficient $-8$ on the trefoil component $J$, which is
left-handed.
\end{proposition}

\begin{proof}
Note that \Cref{fig:baykur-kirby} is the Kirby diagram of Baykur's Stein filling
$X_0$ of $(Y, \xi)$ described by  the Lefschetz fibration via the factorization
$\phi=t_dt_ut_vt_wt_z$: five $1$--handles dual to the inner holes 
$\delta_1,\ldots,\delta_5$ in $\Sigma_0^6$ and five $2$--handles attached along the
vanishing cycles with framing $-1$ relative to the page.  Its boundary is
$Y$.

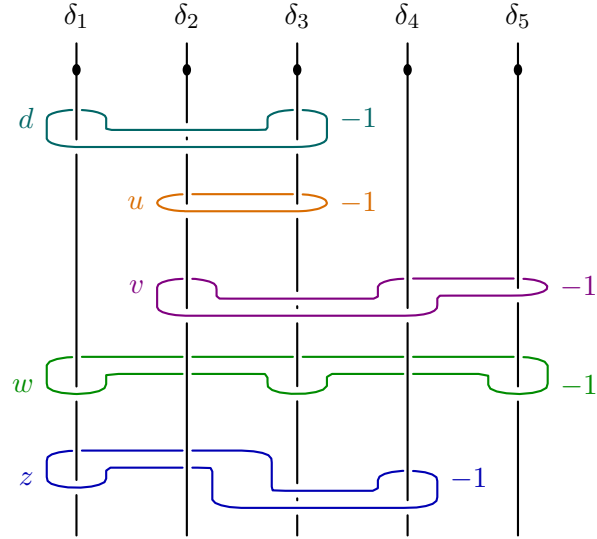
\begin{figure}[ht]
\centering
\begin{tikzpicture}[xscale=0.72,line cap=round, line join=round]
\draw[teal!80!black,line width=0.8pt] plot coordinates{(2.127,6.062) (2.200,6.058) (2.324,6.047) (2.418,6.031) (2.496,6.010) (2.551,5.986) (2.574,5.955) (2.574,5.920) (2.574,5.884) (2.574,5.849) (2.574,5.815) (2.675,5.799) (2.855,5.799) (3.034,5.799) (3.214,5.799) (3.393,5.799) (3.572,5.799) (3.744,5.799) (3.923,5.799) (4.103,5.799) (4.282,5.799) (4.462,5.799) (4.641,5.799) (4.820,5.799) (4.992,5.799) (5.171,5.799) (5.351,5.799) (5.507,5.802) (5.538,5.831) (5.538,5.867) (5.538,5.902) (5.538,5.937) (5.546,5.973) (5.593,5.999) (5.655,6.021) (5.741,6.041) (5.850,6.054) (5.990,6.063) (6.005,6.063)};
\draw[teal!80!black,line width=0.8pt] plot coordinates{(6.184,6.062) (6.287,6.057) (6.404,6.044) (6.490,6.026) (6.568,6.005) (6.614,5.979) (6.630,5.946) (6.630,5.910) (6.630,5.875) (6.630,5.839) (6.630,5.804) (6.630,5.768) (6.630,5.733) (6.630,5.697) (6.614,5.663) (6.575,5.638) (6.505,5.615) (6.419,5.597) (6.310,5.583) (6.170,5.575) (6.006,5.573) (5.827,5.573) (5.647,5.573) (5.476,5.573) (5.296,5.573) (5.117,5.573) (4.937,5.573) (4.758,5.573) (4.579,5.573) (4.407,5.573) (4.228,5.573) (4.048,5.573) (3.869,5.573) (3.689,5.573) (3.510,5.573) (3.338,5.573) (3.159,5.573) (2.980,5.573) (2.800,5.573) (2.621,5.573) (2.441,5.573) (2.270,5.573) (2.090,5.573) (2.005,5.573) (1.864,5.580) (1.739,5.591) (1.646,5.607) (1.568,5.626) (1.513,5.651) (1.482,5.681) (1.482,5.717) (1.482,5.752) (1.482,5.788) (1.482,5.823) (1.482,5.859) (1.482,5.894) (1.482,5.929) (1.490,5.965) (1.521,5.994) (1.583,6.016) (1.669,6.036) (1.771,6.050) (1.895,6.062) (1.949,6.063)};
\draw[orange!85!black,line width=0.8pt] plot coordinates{(4.157,4.949) (4.228,4.949) (4.306,4.949) (4.391,4.949) (4.469,4.949) (4.547,4.949) (4.633,4.949) (4.711,4.949) (4.789,4.949) (4.875,4.949) (4.953,4.949) (5.039,4.949) (5.117,4.949) (5.195,4.949) (5.281,4.949) (5.359,4.949) (5.437,4.949) (5.522,4.949) (5.600,4.949) (5.686,4.949) (5.764,4.949) (5.842,4.949) (5.928,4.949) (6.005,4.949)};
\draw[orange!85!black,line width=0.8pt] plot coordinates{(6.184,4.946) (6.224,4.944) (6.287,4.941) (6.341,4.934) (6.388,4.929) (6.435,4.921) (6.474,4.913) (6.513,4.905) (6.544,4.896) (6.575,4.884) (6.599,4.873) (6.614,4.860) (6.622,4.846) (6.630,4.831) (6.622,4.815) (6.607,4.802) (6.583,4.791) (6.552,4.780) (6.521,4.770) (6.490,4.760) (6.451,4.752) (6.404,4.746) (6.357,4.739) (6.302,4.733) (6.248,4.728) (6.178,4.725) (6.107,4.723) (6.029,4.723) (5.951,4.723) (5.866,4.723) (5.788,4.723) (5.710,4.723) (5.624,4.723) (5.546,4.723) (5.468,4.723) (5.382,4.723) (5.304,4.723) (5.218,4.723) (5.140,4.723) (5.062,4.723) (4.976,4.723) (4.898,4.723) (4.820,4.723) (4.735,4.723) (4.657,4.723) (4.571,4.723) (4.493,4.723) (4.415,4.723) (4.329,4.723) (4.251,4.723) (4.165,4.723) (4.087,4.723) (4.017,4.723) (4.009,4.725) (3.939,4.726) (3.877,4.730) (3.822,4.734) (3.767,4.741) (3.721,4.747) (3.682,4.755) (3.643,4.763) (3.604,4.773) (3.580,4.783) (3.549,4.794) (3.533,4.805) (3.518,4.820) (3.510,4.834) (3.518,4.849) (3.526,4.862) (3.549,4.875) (3.572,4.886) (3.596,4.897) (3.627,4.907) (3.666,4.915) (3.713,4.923) (3.752,4.929) (3.806,4.936) (3.861,4.941) (3.923,4.946) (3.977,4.947)};
\draw[violet,line width=0.8pt] plot coordinates{(4.156,3.829) (4.212,3.828) (4.352,3.813) (4.469,3.793) (4.555,3.765) (4.594,3.731) (4.602,3.688) (4.602,3.644) (4.602,3.601) (4.657,3.568) (4.867,3.567) (5.086,3.567) (5.304,3.567) (5.522,3.567) (5.741,3.567) (5.959,3.567) (6.178,3.567) (6.396,3.567) (6.614,3.567) (6.833,3.567) (7.051,3.567) (7.270,3.567) (7.488,3.567) (7.566,3.594) (7.566,3.638) (7.566,3.681) (7.566,3.725) (7.605,3.762) (7.683,3.789) (7.792,3.812) (7.933,3.826) (8.034,3.830)};
\draw[violet,line width=0.8pt] plot coordinates{(8.213,3.833) (8.338,3.833) (8.557,3.833) (8.775,3.833) (8.993,3.833) (9.212,3.833) (9.430,3.833) (9.649,3.833) (9.867,3.833) (10.061,3.833)};
\draw[violet,line width=0.8pt] plot coordinates{(10.240,3.829) (10.280,3.828) (10.436,3.815) (10.546,3.794) (10.631,3.767) (10.678,3.733) (10.670,3.693) (10.608,3.662) (10.507,3.638) (10.382,3.620) (10.218,3.609) (10.007,3.607) (9.789,3.607) (9.571,3.607) (9.352,3.607) (9.134,3.607) (8.915,3.607) (8.705,3.606) (8.658,3.570) (8.658,3.527) (8.658,3.483) (8.650,3.441) (8.603,3.407) (8.525,3.380) (8.408,3.359) (8.260,3.346) (8.065,3.341) (7.847,3.341) (7.628,3.341) (7.410,3.341) (7.192,3.341) (6.973,3.341) (6.755,3.341) (6.536,3.341) (6.318,3.341) (6.100,3.341) (5.881,3.341) (5.655,3.341) (5.437,3.341) (5.218,3.341) (5.000,3.341) (4.781,3.341) (4.563,3.341) (4.345,3.341) (4.126,3.341) (4.033,3.341) (3.853,3.349) (3.721,3.365) (3.611,3.388) (3.541,3.417) (3.510,3.454) (3.510,3.498) (3.510,3.541) (3.510,3.585) (3.510,3.628) (3.510,3.672) (3.510,3.715) (3.533,3.754) (3.604,3.783) (3.713,3.807) (3.845,3.823) (3.978,3.829)};
\draw[green!55!black,line width=0.8pt] plot coordinates{(2.128,2.797) (2.340,2.797) (2.590,2.797) (2.824,2.797) (3.073,2.797) (3.307,2.797) (3.549,2.797) (3.799,2.797) (3.977,2.797)};
\draw[green!55!black,line width=0.8pt] plot coordinates{(4.157,2.797) (4.274,2.797) (4.516,2.797) (4.758,2.797) (5.000,2.797) (5.242,2.797) (5.483,2.797) (5.725,2.797) (5.967,2.797) (6.005,2.797)};
\draw[green!55!black,line width=0.8pt] plot coordinates{(6.185,2.797) (6.209,2.797) (6.451,2.797) (6.692,2.797) (6.934,2.797) (7.176,2.797) (7.418,2.797) (7.660,2.797) (7.901,2.797) (8.033,2.797)};
\draw[green!55!black,line width=0.8pt] plot coordinates{(8.213,2.797) (8.385,2.797) (8.627,2.797) (8.869,2.797) (9.110,2.797) (9.352,2.797) (9.594,2.797) (9.836,2.797) (10.061,2.797)};
\draw[green!55!black,line width=0.8pt] plot coordinates{(10.240,2.793) (10.327,2.790) (10.507,2.766) (10.631,2.732) (10.686,2.682) (10.686,2.623) (10.686,2.563) (10.686,2.503) (10.686,2.444) (10.663,2.387) (10.561,2.347) (10.405,2.320) (10.171,2.305) (9.929,2.315) (9.750,2.339) (9.633,2.376) (9.594,2.428) (9.594,2.487) (9.594,2.547) (9.415,2.571) (9.118,2.571) (8.814,2.571) (8.518,2.571) (8.221,2.571) (8.213,2.571)};
\draw[green!55!black,line width=0.8pt] plot coordinates{(8.033,2.571) (7.917,2.571) (7.621,2.571) (7.316,2.571) (7.020,2.571) (6.724,2.571) (6.630,2.531) (6.630,2.471) (6.630,2.412) (6.560,2.363) (6.427,2.331) (6.232,2.310) (5.975,2.308) (5.764,2.326) (5.624,2.358) (5.546,2.402) (5.538,2.462) (5.538,2.521) (5.491,2.570) (5.195,2.571) (4.898,2.571) (4.594,2.571) (4.298,2.571) (4.157,2.571)};
\draw[green!55!black,line width=0.8pt] plot coordinates{(3.977,2.571) (3.697,2.571) (3.401,2.571) (3.097,2.571) (2.800,2.571) (2.574,2.555) (2.574,2.497) (2.574,2.437) (2.543,2.383) (2.441,2.344) (2.278,2.318) (2.036,2.305) (1.973,2.307) (1.755,2.321) (1.599,2.349) (1.498,2.391) (1.482,2.447) (1.482,2.507) (1.482,2.566) (1.482,2.626) (1.482,2.686) (1.537,2.734) (1.669,2.768) (1.856,2.790) (1.949,2.793)};
\draw[blue!70!black,line width=0.8pt] plot coordinates{(3.977,1.331) (3.713,1.331) (3.448,1.331) (3.182,1.331) (2.917,1.331) (2.652,1.331) (2.574,1.292) (2.574,1.239) (2.574,1.186) (2.543,1.139) (2.441,1.105) (2.301,1.081) (2.106,1.067) (1.872,1.070) (1.700,1.089) (1.568,1.117) (1.498,1.155) (1.482,1.207) (1.482,1.260) (1.482,1.313) (1.482,1.367) (1.482,1.420) (1.498,1.470) (1.576,1.507) (1.708,1.534) (1.888,1.552) (1.949,1.553)};
\draw[blue!70!black,line width=0.8pt] plot coordinates{(2.128,1.557) (2.129,1.557) (2.395,1.557) (2.660,1.557) (2.933,1.557) (3.198,1.557) (3.463,1.557) (3.728,1.557) (3.977,1.557)};
\draw[blue!70!black,line width=0.8pt] plot coordinates{(4.157,1.557) (4.267,1.557) (4.532,1.557) (4.797,1.557) (5.062,1.557) (5.281,1.547) (5.444,1.526) (5.554,1.496) (5.608,1.452) (5.616,1.400) (5.616,1.347) (5.616,1.294) (5.616,1.241) (5.616,1.188) (5.616,1.133) (5.616,1.080) (5.632,1.031) (5.866,1.025) (6.139,1.025) (6.404,1.025) (6.669,1.025) (6.942,1.025) (7.207,1.025) (7.472,1.025) (7.566,1.059) (7.566,1.112) (7.566,1.165) (7.597,1.213) (7.691,1.249) (7.823,1.273) (8.018,1.289) (8.033,1.289)};
\draw[blue!70!black,line width=0.8pt] plot coordinates{(8.212,1.287) (8.252,1.286) (8.432,1.270) (8.564,1.241) (8.642,1.204) (8.658,1.154) (8.658,1.101) (8.658,1.047) (8.658,0.994) (8.658,0.941) (8.642,0.891) (8.572,0.852) (8.447,0.823) (8.276,0.804) (8.042,0.799) (7.769,0.799) (7.504,0.799) (7.238,0.799) (6.973,0.799) (6.700,0.799) (6.435,0.799) (6.170,0.799) (5.905,0.799) (5.632,0.799) (5.366,0.799) (5.101,0.799) (5.054,0.799) (4.844,0.810) (4.688,0.833) (4.579,0.864) (4.524,0.907) (4.524,0.959) (4.524,1.012) (4.524,1.065) (4.524,1.118) (4.524,1.172) (4.524,1.226) (4.524,1.280) (4.493,1.326) (4.251,1.331) (4.157,1.331)};
\node[teal!80!black,left] at (1.432,5.952) {$d$};
\node[teal!80!black,right] at (6.680,5.952) {$-1$};
\node[orange!85!black,left] at (3.460,4.836) {$u$};
\node[orange!85!black,right] at (6.680,4.836) {$-1$};
\node[violet,left] at (3.460,3.720) {$v$};
\node[violet,right] at (10.728,3.720) {$-1$};
\node[green!55!black,left] at (1.432,2.418) {$w$};
\node[green!55!black,right] at (10.736,2.418) {$-1$};
\node[blue!70!black,left] at (1.432,1.178) {$z$};
\node[blue!70!black,right] at (8.708,1.178) {$-1$};
\draw[black,line width=0.8pt] (2.028,0.434) -- (2.028,0.968);
\draw[black,line width=0.8pt] (2.028,1.168) -- (2.028,2.205);
\draw[black,line width=0.8pt] (2.028,2.405) -- (2.028,5.473);
\draw[black,line width=0.8pt] (2.028,5.673) -- (2.028,6.944);
\fill (2.028,6.594) circle (2.2pt);
\node[above] at (2.028,7.006) {$\delta_{1}$};
\draw[black,line width=0.8pt] (4.056,0.434) -- (4.056,3.241);
\draw[black,line width=0.8pt] (4.056,3.441) -- (4.056,4.623);
\draw[black,line width=0.8pt] (4.056,4.823) -- (4.056,5.473);
\draw[black,line width=0.8pt] (4.056,5.673) -- (4.056,5.699);
\draw[black,line width=0.8pt] (4.056,5.899) -- (4.056,6.944);
\fill (4.056,6.594) circle (2.2pt);
\node[above] at (4.056,7.006) {$\delta_{2}$};
\draw[black,line width=0.8pt] (6.084,0.434) -- (6.084,0.699);
\draw[black,line width=0.8pt] (6.084,0.899) -- (6.084,0.925);
\draw[black,line width=0.8pt] (6.084,1.125) -- (6.084,2.209);
\draw[black,line width=0.8pt] (6.084,2.409) -- (6.084,3.241);
\draw[black,line width=0.8pt] (6.084,3.441) -- (6.084,3.467);
\draw[black,line width=0.8pt] (6.084,3.667) -- (6.084,4.623);
\draw[black,line width=0.8pt] (6.084,4.823) -- (6.084,5.474);
\draw[black,line width=0.8pt] (6.084,5.674) -- (6.084,6.944);
\fill (6.084,6.594) circle (2.2pt);
\node[above] at (6.084,7.006) {$\delta_{3}$};
\draw[black,line width=0.8pt] (8.112,0.434) -- (8.112,0.701);
\draw[black,line width=0.8pt] (8.112,0.901) -- (8.112,3.243);
\draw[black,line width=0.8pt] (8.112,3.443) -- (8.112,6.944);
\fill (8.112,6.594) circle (2.2pt);
\node[above] at (8.112,7.006) {$\delta_{4}$};
\draw[black,line width=0.8pt] (10.140,0.434) -- (10.140,2.206);
\draw[black,line width=0.8pt] (10.140,2.406) -- (10.140,3.508);
\draw[black,line width=0.8pt] (10.140,3.708) -- (10.140,6.944);
\fill (10.140,6.594) circle (2.2pt);
\node[above] at (10.140,7.006) {$\delta_{5}$};
\end{tikzpicture}
\caption{The Kirby diagram of Baykur's Stein filling
$X_0$, with $\partial X_0=Y$.  The five vertical
arcs are the dotted circles of the $1$--handles dual to
$\delta_1,\ldots,\delta_5$ (drawn through infinity), and $u,v,w,z,d$ are
the vanishing cycles.}
\label{fig:baykur-kirby}
\end{figure}

We reduce this diagram by Kirby calculus, writing $h'$ for the $2$--handle
obtained from $h$ by a slide.  The curves $v,w$ are the only ones enclosing
$\delta_5$: slide $h_w$ over $h_v$ to obtain $h_w'$, which no longer runs
over the $1$--handle dual to $\delta_5$, and cancel $h_v$ against that
$1$--handle.  Now $h_w'$ and $h_z$ are the only $2$--handles running over
the $1$--handle dual to $\delta_4$: slide $h_w'$ over $h_z$ to obtain
$h_w''$ and cancel $h_z$ against that $1$--handle.  Finally $h_w''$, $h_u$
and $h_d$ run over the $1$--handle dual to $\delta_3$: slide $h_w''$ and
$h_d$ over $h_u$ to obtain $h_w'''$ and $h_d'$, and cancel $h_u$ against
that $1$--handle. What remains is
$\natural^2(S^1\times D^3)$, with $1$--handles dual to $\delta_1,\delta_2$,
together with two $2$--handles: $h_w'''$ of framing
$-4$, and $h_d'$ of framing $-2$, which runs once over each
surviving $1$--handle, see \cref{fig:X0-DTK}.

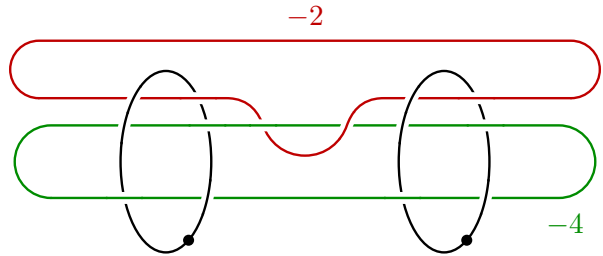
\begin{figure}[ht]
\centering
\begin{tikzpicture}[scale=0.8, line cap=round, line join=round]
\def\cc{2.3}
\draw[green!55!black,line width=0.9pt]
  (-4.2,-0.6) -- (4.2,-0.6) arc[start angle=-90,end angle=90,radius=0.6]
  -- (-4.2,0.6) arc[start angle=90,end angle=270,radius=0.6];
\draw[black,line width=0.9pt] (\cc,0) ellipse[x radius=0.75,y radius=1.5];
\draw[black,line width=0.9pt] (-\cc,0) ellipse[x radius=0.75,y radius=1.5];
\fill ({\cc+0.75*cos(-60)},{1.5*sin(-60)}) circle (2.6pt);
\fill ({-\cc+0.75*cos(-60)},{1.5*sin(-60)}) circle (2.6pt);
\draw[red!75!black,line width=0.9pt]
  (-4.4,1.05) -- (-1.3,1.05)
  .. controls (-0.9,1.05) and (-0.76,0.78) .. (-0.693,0.6)
  .. controls (-0.6,0.33) and (-0.35,0.1) .. (0,0.1)
  .. controls (0.35,0.1) and (0.6,0.33) .. (0.693,0.6)
  .. controls (0.76,0.78) and (0.9,1.05) .. (1.3,1.05) -- (4.4,1.05);
\draw[red!75!black,line width=0.9pt]
  (4.4,1.05) arc[start angle=-90,end angle=90,radius=0.475]
  -- (-4.4,2.0) arc[start angle=90,end angle=270,radius=0.475];
\foreach \x/\y in {-2.987/-0.6, -1.613/0.6, 1.613/-0.6, 2.987/0.6}{
  \draw[white,line width=4.5pt,line cap=butt] (\x-0.28,\y) -- (\x+0.28,\y);
  \draw[green!55!black,line width=0.9pt] (\x-0.3,\y) -- (\x+0.3,\y);}
\draw[white,line width=4.5pt,line cap=butt] (-0.693-0.2,0.6) -- (-0.693+0.2,0.6);
\draw[green!55!black,line width=0.9pt] (-0.693-0.22,0.6) -- (-0.693+0.22,0.6);
\foreach \cx/\t in {-2.3/156.42, -2.3/-23.58, 2.3/156.42, 2.3/-23.58, -2.3/135.57, 2.3/135.57}{
  \draw[white,line width=4.5pt,line cap=butt,domain=\t-11:\t+11,samples=12]
    plot ({\cx+0.75*cos(\x)},{1.5*sin(\x)});
  \draw[black,line width=0.9pt,domain=\t-12:\t+12,samples=12]
    plot ({\cx+0.75*cos(\x)},{1.5*sin(\x)});}
\foreach \x in {-1.764, 2.836}{
  \draw[white,line width=4.5pt,line cap=butt] (\x-0.28,1.05) -- (\x+0.28,1.05);
  \draw[red!75!black,line width=0.9pt] (\x-0.3,1.05) -- (\x+0.3,1.05);}
\begin{scope}
\clip (0.693,0.6) circle (0.24);
\draw[white,line width=4.5pt]
  (0,0.1) .. controls (0.35,0.1) and (0.6,0.33) .. (0.693,0.6)
  .. controls (0.76,0.78) and (0.9,1.05) .. (1.3,1.05);
\end{scope}
\begin{scope}
\clip (0.693,0.6) circle (0.26);
\draw[red!75!black,line width=0.9pt]
  (0,0.1) .. controls (0.35,0.1) and (0.6,0.33) .. (0.693,0.6)
  .. controls (0.76,0.78) and (0.9,1.05) .. (1.3,1.05);
\end{scope}
\node[green!55!black] at (4.3,-1.05) {$-4$};
\node[red!75!black] at (0,2.4) {$-2$};
\end{tikzpicture}
\caption{Baykur's Stein filling $X_0$ after the first handle moves in the proof
of \cref{prop:baykur-surgery}: the diagram of $DT^*K$ of \cref{fig:klein-kirby}
together with the $2$--handle $h_d'$ of framing $-2$, which runs once over each
$1$--handle, in the same direction, and links the $2$--handle of $DT^*K$ once.
Compare \cref{fig:X0d1}.}
\label{fig:X0-DTK}
\end{figure}

The $2$--handle $h_w'''$ runs twice, in the same direction, over the
$1$--handle dual to $\delta_1$, and $h_d'$ runs over it once.  Slide
$h_w'''$ twice over $h_d'$ to obtain $h_w''''$, which no longer runs over the
$1$--handles dual to $\delta_1$ and $\delta_2$. Its framing is $-8$.
  Now cancel $h_d'$ against the $1$--handle dual to
$\delta_1$.  The result is $S^1\times D^3$ with a single $2$--handle attached
along a null-homologous knot $J'$ of framing $-8$. Hence $Y$ is
$(0,-8)$--surgery on the two-component link formed by the dotted circle and
$J'$.  Tracing the attaching circle through these moves and simplifying
the resulting diagram by a framed isotopy gives the Kirby diagram on
the left of \cref{fig:final-kirby}. The sequence of handle moves, drawn by hand
and showing the slides, the framings and the final isotopy, is included as an
ancillary file with the arXiv version of this paper.  Replacing the
dotted circle by a $0$--framed unknot gives the surgery diagram of $Y$ depicted on the right of \cref{fig:final-kirby}. 
The link is $7^2_4$ in Rolfsen's table \cite{Rolfsen}, i.e., L7a3 in the
Thistlethwaite notation used by LinkInfo \cite{LinkInfo}. The trefoil
component $J$ on the right of \cref{fig:final-kirby} is left-handed. 
\end{proof}

\begin{figure}[ht]
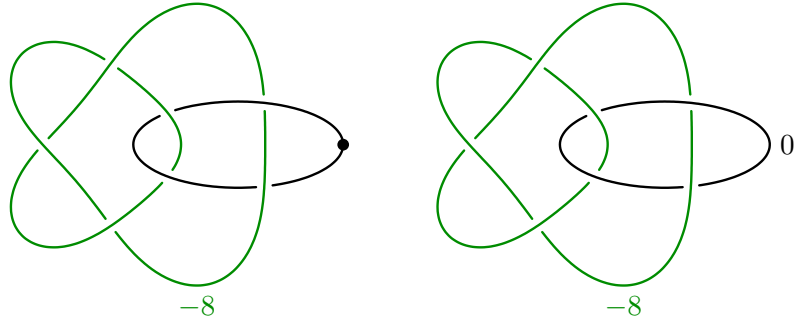

\centering

\caption{Left: A Kirby diagram of Baykur's Stein filling $X_0$, obtained from
\cref{fig:baykur-kirby} by the handle slides and cancellations in the proof
of \cref{prop:baykur-surgery}.  Right: The corresponding surgery diagram of
Baykur's manifold $Y$: the link $L=7^2_4$, with coefficient $0$ on the
unknotted component $U$ and $-8$ on the trefoil component $J$.}
\label{fig:final-kirby}
\end{figure}

\begin{proposition}\label[proposition]{prop:baykur-seifert}
Baykur's manifold $Y$ is orientation-preservingly diffeomorphic to the
$3$--manifold obtained from $ST^*K$ by Dehn surgery along a regular fiber of
$ST^*K\to K$, with coefficient $-2$ relative to the bundle framing.  In
particular, $Y$ is Seifert fibered over the Klein bottle with one exceptional
fiber of multiplicity $2$. Up to orientation, it is the Seifert fibered space
with one exceptional fiber of type $(2,1)$ and Seifert Euler parameter $b=0$,
in the convention of \cite[Proposition~4.3]{Wilkes}.
\end{proposition}

The proof is given in \cref{subsec:W0}, after the new filling $W_0$ has been
constructed: Kirby calculus on the Lefschetz fibration $W_0$ exhibits $Y$ as
surgery along a regular fiber of $ST^*K\to K$ (\cref{prop:W0}).

Note that \cref{prop:baykur-surgery,prop:baykur-seifert} together prove \cref{thm:main}(1). Both propositions are also confirmed by computer, independently of the handle
calculus, see \cref{subsec:computer-Y}.

\section{The filling \texorpdfstring{$X_B$}{X\_B}}\label{sec:klein}

Baykur's mapping class $\phi=t_dB$ is obtained from
\[
B=t_ut_vt_wt_z\in\Mod(\Sigma_0^6,\partial\Sigma_0^6)
\]
by adding a single Dehn twist.  We write
\[
Y_B=\ob(\Sigma_0^6,B),\qquad X_B=X(\Sigma_0^6;u,v,w,z),
\]
where $X(\Sigma_0^6;u,v,w,z)$ is the Lefschetz fibration with vanishing cycles
$u,v,w,z$, as in the introduction.

\begin{proposition}\label[proposition]{thm:klein-main}
$X_B$ is diffeomorphic to $DT^*K$.  Consequently $Y_B$ is
orientation-preservingly diffeomorphic to $ST^*K$.  In particular $Y_B$ is
the circle bundle over the Klein bottle with Euler number zero.
\end{proposition}

\subsection{Kirby calculus}\label{subsec:kirby}

We prove \cref{thm:klein-main} by Kirby calculus. The curves $u,v,w,z$ are shown in \cref{fig:baykur-curves}.  The
Lefschetz fibration $X_B$ is
$\Sigma_0^6\times D^2=\natural^5(S^1\times D^3)$, with $1$--handles dual to
the holes $\delta_1,\ldots,\delta_5$, together with four $2$--handles
$h_u,h_v,h_w,h_z$ attached along the page curves with framing $-1$ relative
to the page (see \cref{fig:baykur-kirby}, with $d$ omitted). 

\begin{figure}[htbp]
\centering
\begin{tikzpicture}[scale=0.8, line cap=round, line join=round]
\def\cc{2.3}
\draw[green!55!black,line width=0.9pt]
  (-4.2,-0.6) -- (4.2,-0.6) arc[start angle=-90,end angle=90,radius=0.6]
  -- (-4.2,0.6) arc[start angle=90,end angle=270,radius=0.6];
\draw[black,line width=0.9pt] (\cc,0) ellipse[x radius=0.75,y radius=1.5];
\draw[black,line width=0.9pt] (-\cc,0) ellipse[x radius=0.75,y radius=1.5];
\fill ({\cc+0.75*cos(60)},{1.5*sin(60)}) circle (2.6pt);
\fill ({-\cc+0.75*cos(60)},{1.5*sin(60)}) circle (2.6pt);
\foreach \x/\y in {-2.987/-0.6, -1.613/0.6, 1.613/-0.6, 2.987/0.6}{
  \draw[white,line width=4.5pt,line cap=butt] (\x-0.28,\y) -- (\x+0.28,\y);
  \draw[green!55!black,line width=0.9pt] (\x-0.3,\y) -- (\x+0.3,\y);}
\foreach \cx/\t in {-2.3/156.42, -2.3/-23.58, 2.3/156.42, 2.3/-23.58}{
  \draw[white,line width=4.5pt,line cap=butt,domain=\t-11:\t+11,samples=12]
    plot ({\cx+0.75*cos(\x)},{1.5*sin(\x)});
  \draw[black,line width=0.9pt,domain=\t-12:\t+12,samples=12]
    plot ({\cx+0.75*cos(\x)},{1.5*sin(\x)});}
\node[green!55!black] at (4.3,-1.05) {$-4$};
\end{tikzpicture}
\caption{The diagram obtained from the Kirby diagram of $X_B$
(\cref{fig:baykur-kirby} with $d$ omitted), which is also the Kirby diagram of the
disk bundle over the Klein bottle, as in
\cite[Figure~21(b), $k=2$]{ObaOzbagci}. Compare the exercise on disk
bundles over nonorientable surfaces in \cite[\S4.6]{GS}.}
\label{fig:klein-kirby}
\end{figure}
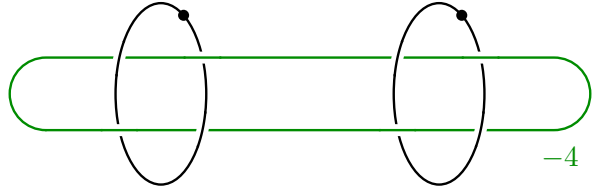

\begin{proof}[Proof of \cref{thm:klein-main}]
The curves $v$ and $w$ are the only ones enclosing $\delta_5$.  Slide $h_w$
over $h_v$ so that the resulting $2$--handle $h_w'$ no longer runs over the
$1$--handle dual to $\delta_5$, and cancel $h_v$ against that $1$--handle.
Now $h_w'$ and $h_z$ are the only $2$--handles running over the $1$--handle
dual to $\delta_4$. Slide $h_w'$ over $h_z$ to obtain $h_w''$ and cancel
$h_z$ against that $1$--handle.  Finally $h_w''$ and $h_u$ are the only
$2$--handles running over the $1$--handle dual to $\delta_3$. Slide $h_w''$
over $h_u$ to obtain $h_w'''$ and cancel $h_u$ against that $1$--handle.
What remains is $\natural^2(S^1\times D^3)$, with $1$--handles dual to
$\delta_1,\delta_2$, and the single $2$--handle $h_w'''$.

The framing of $h_w'''$ is $-4$, following the three slides
($-1\to-2\to-3\to-4$).  An isotopy carries the resulting diagram to
\cref{fig:klein-kirby}. The sequence of handle moves, drawn by hand, is
included as an ancillary file with the arXiv version of this paper.  \cref{fig:klein-kirby} is the $k=2$ diagram of
\cite[\S3.4, Figure~21(b)]{ObaOzbagci}: in that convention, the diagram with
$k$ crosscap $1$--handles and $2$--handle coefficient $-k-2$ is the disk
bundle over the nonorientable surface of genus $k$ with Euler number $k-2$,
in the sense of Gompf--Stipsicz \cite[\S4.6]{GS}.  Thus the coefficient $-4$
at $k=2$ gives Euler number zero, and $X_B\cong DT^*K$.

The manifold $DT^*K$ admits an orientation-reversing self-diffeomorphism:
since $\chi(K)=0$,
$K$ carries a nowhere-vanishing vector field, and reflecting each fiber of
$T^*K$ across the line it determines, via a Riemannian metric, is the identity on $K$ and a reflection
on the fibers.  Hence the
diffeomorphism $X_B\to DT^*K$ may be taken orientation preserving.  Its
restriction to the boundary is then an orientation-preserving diffeomorphism
$Y_B\to ST^*K$.  Finally $ST^*K$ is the circle bundle over $K$ with
Euler number zero, in the sense of \cite[\S4.6]{GS}.  Since $K$ is
nonorientable, this is not a principal circle bundle: the fiber orientation
of the orientable total space reverses along orientation-reversing loops of
$K$, and the Euler number is defined with twisted coefficients.
\end{proof}

\section{Stein fillings with even torsion}\label{sec:newfillings}

In this section we exploit two further symmetries of the monodromy $B$, half twists interchanging two of the holes, to produce Stein fillings of $(Y,\xi)$ that are not in Baykur's list.  Throughout, $\Sigma=\Sigma_0^6$, $\Gamma=\Mod(\Sigma,\partial\Sigma)$, products are written in functional order, and $ht_ch^{-1}=t_{h(c)}$ for $h\in\Gamma$.  Recall that for $m\ge0$, $\mathbb Z/m\rtimes_{-1}\mathbb Z$ denotes the semidirect product in which the generator of $\mathbb Z$ acts on $\mathbb Z/m$ by inversion. For $m=0$ this is the Klein bottle group $\pi_1(K)=\langle r,b\mid brb^{-1}=r^{-1}\rangle$, and for $m=1$ it is $\mathbb Z$.  Baykur's fillings $X_n$ have $\pi_1(X_n)\cong\mathbb Z/(2n{+}1)\rtimes_{-1}\mathbb Z$ \cite[Theorem~1]{BaykurPlanar}.  We prove \cref{thm:newfillings}, stated in the introduction.

\subsection{Half twists commuting with the monodromy}\label{subsec:commute}

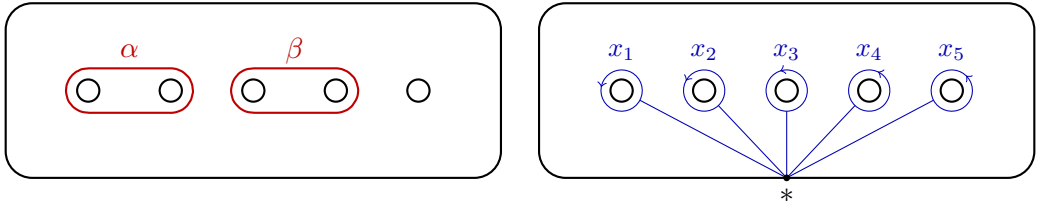
\begin{figure}[ht]
\centering
\makebox[\textwidth][c]{%
\begin{tikzpicture}[scale=0.42]
\draw[thick,rounded corners=10pt] (0,-2.75) rectangle (15.600000000000001,2.75);
\draw[thick] (2.60,0.00) circle (0.35);
\draw[thick] (5.20,0.00) circle (0.35);
\draw[thick] (7.80,0.00) circle (0.35);
\draw[thick] (10.40,0.00) circle (0.35);
\draw[thick] (13.00,0.00) circle (0.35);
\draw[red!75!black,thick] plot[smooth cycle,tension=0.7] coordinates{(2.53,-0.69) (2.45,-0.68) (2.37,-0.66) (2.30,-0.63) (2.23,-0.59) (2.17,-0.55) (2.11,-0.50) (2.07,-0.45) (2.02,-0.39) (1.98,-0.33) (1.95,-0.26) (1.93,-0.18) (1.91,-0.10) (1.90,-0.00) (1.91,0.09) (1.92,0.17) (1.95,0.25) (1.98,0.32) (2.02,0.39) (2.06,0.44) (2.11,0.50) (2.16,0.54) (2.22,0.59) (2.29,0.62) (2.36,0.66) (2.44,0.68) (2.53,0.69) (2.63,0.70) (2.73,0.70) (2.83,0.70) (2.94,0.70) (3.04,0.70) (3.14,0.70) (3.25,0.70) (3.35,0.70) (3.46,0.70) (3.56,0.70) (3.66,0.70) (3.77,0.70) (3.87,0.70) (3.97,0.70) (4.08,0.70) (4.18,0.70) (4.29,0.70) (4.39,0.70) (4.49,0.70) (4.60,0.70) (4.70,0.70) (4.80,0.70) (4.91,0.70) (5.01,0.70) (5.12,0.70) (5.22,0.70) (5.31,0.69) (5.40,0.67) (5.47,0.64) (5.54,0.61) (5.60,0.57) (5.66,0.52) (5.71,0.47) (5.76,0.42) (5.80,0.36) (5.83,0.29) (5.86,0.22) (5.88,0.14) (5.89,0.05) (5.90,-0.04) (5.88,-0.13) (5.86,-0.21) (5.84,-0.28) (5.80,-0.35) (5.76,-0.41) (5.72,-0.47) (5.67,-0.52) (5.61,-0.57) (5.55,-0.61) (5.48,-0.64) (5.40,-0.67) (5.32,-0.69) (5.23,-0.70) (5.13,-0.70) (5.02,-0.70) (4.92,-0.70) (4.82,-0.70) (4.71,-0.70) (4.61,-0.70) (4.50,-0.70) (4.40,-0.70) (4.30,-0.70) (4.19,-0.70) (4.09,-0.70) (3.98,-0.70) (3.88,-0.70) (3.78,-0.70) (3.67,-0.70) (3.57,-0.70) (3.47,-0.70) (3.36,-0.70) (3.26,-0.70) (3.15,-0.70) (3.05,-0.70) (2.95,-0.70) (2.84,-0.70) (2.74,-0.70) (2.64,-0.70) (2.54,-0.69)};
\draw[red!75!black,thick] plot[smooth cycle,tension=0.7] coordinates{(7.73,-0.69) (7.65,-0.68) (7.57,-0.66) (7.50,-0.63) (7.43,-0.59) (7.37,-0.55) (7.32,-0.50) (7.27,-0.45) (7.22,-0.39) (7.18,-0.33) (7.15,-0.26) (7.13,-0.18) (7.11,-0.10) (7.10,-0.00) (7.11,0.09) (7.13,0.18) (7.15,0.25) (7.18,0.32) (7.22,0.39) (7.26,0.44) (7.31,0.50) (7.36,0.55) (7.42,0.59) (7.49,0.63) (7.56,0.66) (7.64,0.68) (7.73,0.69) (7.83,0.70) (7.93,0.70) (8.03,0.70) (8.14,0.70) (8.24,0.70) (8.35,0.70) (8.45,0.70) (8.55,0.70) (8.66,0.70) (8.76,0.70) (8.86,0.70) (8.97,0.70) (9.07,0.70) (9.18,0.70) (9.28,0.70) (9.38,0.70) (9.49,0.70) (9.59,0.70) (9.69,0.70) (9.80,0.70) (9.90,0.70) (10.01,0.70) (10.11,0.70) (10.21,0.70) (10.32,0.70) (10.42,0.70) (10.51,0.69) (10.60,0.67) (10.67,0.64) (10.74,0.61) (10.80,0.57) (10.86,0.52) (10.91,0.47) (10.96,0.42) (11.00,0.36) (11.03,0.29) (11.06,0.22) (11.08,0.14) (11.10,0.05) (11.10,-0.04) (11.09,-0.13) (11.07,-0.21) (11.04,-0.28) (11.00,-0.35) (10.96,-0.41) (10.92,-0.47) (10.87,-0.52) (10.81,-0.57) (10.75,-0.61) (10.68,-0.64) (10.60,-0.67) (10.52,-0.69) (10.43,-0.70) (10.33,-0.70) (10.22,-0.70) (10.12,-0.70) (10.02,-0.70) (9.91,-0.70) (9.81,-0.70) (9.71,-0.70) (9.60,-0.70) (9.50,-0.70) (9.39,-0.70) (9.29,-0.70) (9.19,-0.70) (9.08,-0.70) (8.98,-0.70) (8.87,-0.70) (8.77,-0.70) (8.67,-0.70) (8.56,-0.70) (8.46,-0.70) (8.36,-0.70) (8.25,-0.70) (8.15,-0.70) (8.04,-0.70) (7.94,-0.70) (7.84,-0.70) (7.74,-0.69)};
\node[red!75!black] at (3.90,1.25) {$\alpha$};
\node[red!75!black] at (9.10,1.25) {$\beta$};
\begin{scope}[shift={(16.8,0)}]
\draw[thick,rounded corners=10pt] (0,-2.75) rectangle (15.600000000000001,2.75);
\draw[thick] (2.60,0.00) circle (0.35);
\draw[blue!70!black] (7.80,-2.75) -- (3.175,-0.304);
\draw[blue!70!black,decoration={markings,mark=at position 0.55 with {\arrow{>}}},postaction={decorate}] (3.175,-0.304) arc (-27.87:332.13:0.65);
\node[blue!70!black] at (2.60,1.20) {\small$x_1$};
\draw[thick] (5.20,0.00) circle (0.35);
\draw[blue!70!black] (7.80,-2.75) -- (5.647,-0.472);
\draw[blue!70!black,decoration={markings,mark=at position 0.55 with {\arrow{>}}},postaction={decorate}] (5.647,-0.472) arc (-46.61:313.39:0.65);
\node[blue!70!black] at (5.20,1.20) {\small$x_2$};
\draw[thick] (7.80,0.00) circle (0.35);
\draw[blue!70!black] (7.80,-2.75) -- (7.800,-0.650);
\draw[blue!70!black,decoration={markings,mark=at position 0.55 with {\arrow{>}}},postaction={decorate}] (7.800,-0.650) arc (-90.00:270.00:0.65);
\node[blue!70!black] at (7.80,1.20) {\small$x_3$};
\draw[thick] (10.40,0.00) circle (0.35);
\draw[blue!70!black] (7.80,-2.75) -- (9.953,-0.472);
\draw[blue!70!black,decoration={markings,mark=at position 0.55 with {\arrow{>}}},postaction={decorate}] (9.953,-0.472) arc (-133.39:226.61:0.65);
\node[blue!70!black] at (10.40,1.20) {\small$x_4$};
\draw[thick] (13.00,0.00) circle (0.35);
\draw[blue!70!black] (7.80,-2.75) -- (12.425,-0.304);
\draw[blue!70!black,decoration={markings,mark=at position 0.55 with {\arrow{>}}},postaction={decorate}] (12.425,-0.304) arc (-152.13:207.87:0.65);
\node[blue!70!black] at (13.00,1.20) {\small$x_5$};
\fill (7.80,-2.75) circle (3pt); \node[below] at (7.80,-2.75) {$\ast$};
\end{scope}
\end{tikzpicture}}
\caption{Left: the curves $\alpha$ and $\beta$.  Right: the free generators
$x_1,\dots,x_5$ of $\pi_1(\Sigma_0^6,\ast)$, with $\ast\in\delta_0$.  The holes are $\delta_1,\ldots,\delta_5$
from left to right.}
\label{fig:alpha-beta-x}
\end{figure}

Fix a basepoint $\ast\in\delta_0$ and the free generators $x_1,\dots,x_5$ of
$\pi_1(\Sigma_0^6,\ast)$ of \cite[Figure~2(h)]{BaykurPlanar}, shown in
\cref{fig:alpha-beta-x}.  In these generators Baykur's curves are
$u=x_3x_2$, $v=x_5x_4x_2$, $w=x_5x_4x_3x_2x_1x_2^{-1}x_4^{-1}$,
$z=x_4x_2x_1x_2^{-1}$ and $d=x_3x_1$, up to conjugation \cite{BaykurPlanar}.
Each of these curves is the boundary of a disk containing some of the holes, and
a word representing such a curve has exponent sum $1$ in $x_k$ if the curve
encloses $\delta_k$ and $0$ otherwise.  The right-handed Dehn twist along the
boundary of a disk containing consecutive holes $\delta_i,\ldots,\delta_j$
conjugates $x_i,\ldots,x_j$ by the boundary word $x_j\cdots x_i$ and fixes the
other generators.

We fix the half-twist convention on $\pi_1(\Sigma_0^6,\ast)$, reading path
products from left to right. In this
convention the boundary of the two-hole disk containing $\delta_i$ and
$\delta_{i+1}$ is represented by $\gamma_i=x_{i+1}x_i$. For $1\le i\le4$, let $\sigma_i$ denote the clockwise half twist
interchanging $\delta_i$ and $\delta_{i+1}$, with the page oriented as in
\cref{fig:alpha-beta-x}. Its action is
\begin{equation}\label{eq:artin}
\sigma_i\colon\quad x_i\mapsto x_i^{-1}x_{i+1}x_i,\qquad
x_{i+1}\mapsto x_i,\qquad x_j\mapsto x_j\ (j\ne i,i+1).
\end{equation}
Its inverse acts by
\begin{equation}\label{eq:artininv}
\sigma_i^{-1}\colon\quad x_i\mapsto x_{i+1},\qquad
x_{i+1}\mapsto x_{i+1}x_ix_{i+1}^{-1},\qquad
x_j\mapsto x_j\ (j\ne i,i+1).
\end{equation}
These formulas fix $\gamma_i$ and satisfy
$\sigma_i^2(x_j)=\gamma_i^{-1}x_j\gamma_i$ for $j=i,i+1$. Thus they are compatible
with the based boundary word, and $\sigma_i^2$ acts as the left-handed Dehn
twist along the boundary of the disk containing $\delta_i,\delta_{i+1}$.

Let $\alpha$ and $\beta$ be the curves of \cref{fig:alpha-beta-x}.  Baykur's generalized lantern relation \cite[Lemma~3]{BaykurPlanar} reads
\begin{equation}\label{eq:lantern}
B\,t_\alpha t_\beta=\Delta,\qquad \Delta:=t_{\delta_0}\,t_{\delta_1}^2t_{\delta_2}^2t_{\delta_3}^2t_{\delta_4}^2\,t_{\delta_5}.
\end{equation}
To define the group in which the half twists live, fix parametrizations
$\theta_i\colon S^1\to\delta_i$ of the inner boundary components,
$i=1,\dots,5$.  Let $\widehat\Gamma$ be the group of isotopy classes of
orientation-preserving diffeomorphisms $h$ of $\Sigma$ that restrict to the
identity on $\delta_0$ and satisfy $h\circ\theta_i=\theta_{\pi(i)}$ for some
permutation $\pi$ of $\{1,\dots,5\}$, where isotopies are through
diffeomorphisms of the same kind with the same permutation. In particular,
isotopies fix every boundary component pointwise.  Restricting to $\pi=\mathrm{id}$
gives $\Gamma=\Mod(\Sigma,\partial\Sigma)$, which is a normal subgroup of
$\widehat\Gamma$, and the twists along the inner boundary components remain
nontrivial in $\widehat\Gamma$.  The half twists $\sigma_1$ and $\sigma_3$ of
\eqref{eq:artin} may be taken to be supported in the disks $D_\alpha$ and
$D_\beta$ bounded by $\alpha$ and $\beta$ and to satisfy this compatibility
with the parametrizations. We regard them as elements of $\widehat\Gamma$.
A different compatible choice changes such an element by twists along the
inner boundary components. These fix every closed curve up to isotopy and act
trivially on $\pi_1(\Sigma,\ast)$, so the choice does not affect anything
below.

\begin{proposition}\label[proposition]{prop:commute}
$\sigma_1B\sigma_1^{-1}=B$ and $\sigma_3B\sigma_3^{-1}=B$ in $\Gamma$.
\end{proposition}

\begin{proof}
Let $\sigma\in\{\sigma_1,\sigma_3\}$.  Since $\Delta$ is a product of twists along boundary-parallel curves, it is central in $\Gamma$, and \eqref{eq:lantern} gives
\begin{equation}\label{eq:Bformula}
B=\Delta\,t_\beta^{-1}t_\alpha^{-1}.
\end{equation}
Since $\sigma$ is orientation preserving, $\sigma t_c\sigma^{-1}=t_{\sigma(c)}$ for every simple closed curve $c$.  The support of $\sigma$ lies in the interior of $D_\alpha$ or of $D_\beta$. In either case $\sigma$ fixes both $\alpha$ and $\beta$ up to isotopy (one of them bounds the supporting disk, the other is disjoint from it).  Hence $\sigma t_\alpha\sigma^{-1}=t_\alpha$ and $\sigma t_\beta\sigma^{-1}=t_\beta$.  Finally $\sigma$ maps each $\delta_i$ to a $\delta_j$, interchanging either $\delta_1,\delta_2$ or $\delta_3,\delta_4$ and fixing the others, and the two interchanged boundary components appear in $\Delta$ with the same exponent. Hence $\sigma\Delta\sigma^{-1}=\Delta$.  Conjugating \eqref{eq:Bformula} by $\sigma$ therefore gives $\sigma B\sigma^{-1}=\Delta t_\beta^{-1}t_\alpha^{-1}=B$.
\end{proof}

\begin{corollary}\label[corollary]{cor:factorizations}
For every $f\in\widehat\Gamma$ commuting with $B$, for instance every $f$ in the subgroup generated by $\sigma_1,\sigma_3,t_\alpha$ (\cref{prop:commute}),
\begin{equation}\label{eq:Pg}
P_f:=t_d\;t_{f(u)}\,t_{f(v)}\,t_{f(w)}\,t_{f(z)}
\end{equation}
is a positive factorization of $\phi=t_dB$ in $\Gamma$ into Dehn twists along homotopically nontrivial, non-boundary-parallel curves.
\end{corollary}

\begin{proof}
Note that $t_{f(u)}t_{f(v)}t_{f(w)}t_{f(z)}=fBf^{-1}=B$ by \cref{prop:commute} and \cite[(8)]{BaykurPlanar}. The curves $f(c)$ are images of curves that enclose at least two holes under a diffeomorphism permuting the holes, so they are neither null-homotopic nor boundary-parallel.
\end{proof}

\begin{remark}\label[remark]{rem:extended}
The element $f$ need not lie in $\Gamma$, but this is harmless: it is only used to produce the curves $f(u),f(v),f(w),f(z)$, and the factorization \eqref{eq:Pg} is a word in $\Gamma$.  The identity $t_{f(u)}t_{f(v)}t_{f(w)}t_{f(z)}=B$ holds in $\Gamma$: for any orientation-preserving diffeomorphism $f$ of $\Sigma$ mapping $\partial\Sigma$ to itself, $f\,t_c\,f^{-1}=t_{f(c)}$ at the level of diffeomorphisms, and conjugation by $f$ carries isotopies rel $\partial\Sigma$ to isotopies rel $\partial\Sigma$, so the relation $B=\Delta t_\beta^{-1}t_\alpha^{-1}$ in $\Gamma$ is carried to $fBf^{-1}=(f\Delta f^{-1})\,t_{f(\beta)}^{-1}t_{f(\alpha)}^{-1}$ in $\Gamma$.  For the elements $f$ considered here, which preserve $\alpha$ and $\beta$ and permute only boundary components with equal exponents in $\Delta$, we have $f\Delta f^{-1}=\Delta$, and this is the argument of \cref{prop:commute}.  In the fundamental group computations below, $f$ is supported in $D_\alpha\cup D_\beta$ and hence fixes the basepoint on $\delta_0$.
\end{remark}

By \cite{AO,LP}, the Lefschetz fibration over the disk with fiber $\Sigma$ prescribed by $P_f$ has a Stein structure, and its boundary, with the induced contact structure, is the contact manifold supported by the open book $(\Sigma,\phi)$, i.e., $(Y,\xi)$.  We denote this Stein filling by $X_f$.  Its Euler characteristic is $\chi(\Sigma)+5=1$.

\subsection{Fundamental groups}\label{subsec:pi1new}

Let $x_1,\dots,x_5$ be the free generators of $\pi_1(\Sigma,\ast)=F_5$ shown in \cref{fig:alpha-beta-x}, with $\ast\in\delta_0$.  Let $Q=F_5/\ncl{u,v,w,z}$, where $N(S)$ denotes the normal closure of $S$.  Baykur
computes \cite[(12)--(16)]{BaykurPlanar} that $Q$ is generated by the images $a$
and $b$ of $x_1$ and $x_2$ in $Q$, with the single relation $a^2=b^2$.  This is the
fundamental group of the Klein bottle, and below we change generators to obtain
its standard presentation.  For the reader's convenience we recall the
computation.  The relations $u=x_3x_2=1$, $z=x_4x_2x_1x_2^{-1}=1$ and
$v=x_5x_4x_2=1$ give, in $Q$,
\begin{equation}\label{eq:subst}
x_3=b^{-1},\qquad x_4=x_2x_1^{-1}x_2^{-1}=ba^{-1}b^{-1},\qquad
x_5=x_2^{-1}x_4^{-1}=ab^{-1}.
\end{equation}
Substituting these into $w=x_5x_4x_3x_2x_1x_2^{-1}x_4^{-1}$, and using
$x_5x_4=(ab^{-1})(ba^{-1}b^{-1})=b^{-1}$, we get
$w=(b^{-1})\,b^{-1}ba\,b^{-1}(bab^{-1})=b^{-1}a^2b^{-1}$, and the relation $w=1$
becomes $a^2=b^2$.  This shows that $Q=\langle a,b\mid a^2=b^2\rangle$.  We now
replace the generators $a,b$ by $r=ab^{-1}$ and $b$, with $a=rb$.  In these
generators the relation $a^2=b^2$ becomes $rbrb=b^2$, which is equivalent to
$brb^{-1}=r^{-1}$.  Hence $Q$ has the presentation
$\langle r,b\mid brb^{-1}=r^{-1}\rangle$, which we use from now on.  Moreover
$b^2$ is central, and $x_2x_1=ba=brb=r^{-1}b^2$.  In summary,
\begin{equation}\label{eq:Q}
Q\cong\langle r,b\mid brb^{-1}=r^{-1}\rangle,\qquad b^2=a^2\ \text{central},\qquad x_2x_1=r^{-1}b^2.
\end{equation}
We record the identities, valid in $Q$,
\begin{equation}\label{eq:ids}
b^{-1}rb=r^{-1},\qquad br^{-1}=rb,\qquad b^{-1}r^kb=r^{-k}\ (k\in\Z).
\end{equation}

We recall how to compute the fundamental group of a Lefschetz fibration from its
vanishing cycles.  Let $X\to D^2$ be a Lefschetz fibration with regular fiber a
surface $\Sigma$ with nonempty boundary, and with vanishing cycles
$c_1,\ldots,c_k\subset\Sigma$.  Then $X$ is obtained from $\Sigma\times D^2$ by
attaching $2$--handles along copies of $c_1,\ldots,c_k$ in fibers over points of
$\partial D^2$, with framing $-1$ relative to the fiber framing
\cite[\S8.2]{GS}.  Since $\Sigma\times D^2$ deformation retracts onto
$\Sigma$, and attaching a $2$--handle along a curve kills the class of that
curve, the Seifert--van Kampen theorem gives
\[
\pi_1(X)\cong\pi_1(\Sigma,\ast)/\ncl{c_1,\ldots,c_k},
\]
where each $c_i$ is regarded as an element of $\pi_1(\Sigma,\ast)$ by joining it
to the base point by an arc.  A different choice of arc changes $c_i$ by
conjugation, and replacing $c_i$ by its inverse does not change the normal
closure either, so the right-hand side is well defined.  This vanishing-cycle
presentation gives, for the filling $X_f$ of \eqref{eq:Pg},
\begin{equation}\label{eq:pi1Xg}
\pi_1(X_f)\cong F_5/\ncl{d,\,f(u),\,f(v),\,f(w),\,f(z)},
\end{equation}
where the curves are regarded as elements of $F_5$ (well defined up to conjugation and inversion, which does not affect normal closures), and where $f$ acts on $F_5$ as the automorphism induced by the diffeomorphism $f$, which fixes $\ast\in\delta_0$.

\begin{lemma}\label[lemma]{lem:pi1}
For $f$ as in \cref{cor:factorizations}, $\pi_1(X_f)\cong Q/\ncl{\,\overline{f^{-1}(d)}\,}$, where the bar denotes the image in $Q$.
\end{lemma}

\begin{proof}
Apply the automorphism $f^{-1}$ of $F_5$ to \eqref{eq:pi1Xg}: it maps
$\ncl{d,f(u),f(v),f(w),f(z)}$ isomorphically onto $\ncl{f^{-1}(d),u,v,w,z}$,
and
\[
F_5/\ncl{f^{-1}(d),u,v,w,z}=\big(F_5/\ncl{u,v,w,z}\big)/\ncl{\overline{f^{-1}(d)}}.
\]
\end{proof}

\begin{lemma}\label[lemma]{lem:images}
With $d=x_3x_1$ and the convention \eqref{eq:artin}--\eqref{eq:artininv},
for every $n\ge0$ the image of $t_\alpha^{-n}\sigma_1^{-1}(d)$ in $Q$ is
$r^{-2n}$.
\end{lemma}

\begin{proof}
By \eqref{eq:artininv}, $\sigma_1^{-1}(d)=x_3x_2$.  The curve $\alpha$ is the round curve enclosing $\delta_1$ and $\delta_2$, so
the right-handed twist $t_\alpha$ conjugates $x_1$ and $x_2$ by the boundary word
$x_2x_1$ of the disk containing these two holes and fixes $x_3,x_4,x_5$.
Since $\alpha=x_2x_1$ is fixed by $t_\alpha$, it follows by induction that
$t_\alpha^{-n}(x_i)=\alpha^{-n}x_i\alpha^n$ for $i=1,2$ and every $n\ge0$, and $t_\alpha$ fixes $x_3$
\cite[(17)--(18)]{BaykurPlanar}.  Hence $t_\alpha^{-n}\sigma_1^{-1}(d)=x_3\alpha^{-n}x_2\alpha^n$.  In $Q$ we have
$\overline{\alpha}=r^{-1}b^2$ with $b^2$ central, so conjugation by $\overline{\alpha}^{\,-n}$ is
conjugation by $r^n$, and by \eqref{eq:subst} and \eqref{eq:ids}
\[
x_3\alpha^{-n}x_2\alpha^n\longmapsto b^{-1}r^nbr^{-n}=r^{-n}r^{-n}=r^{-2n}.\qedhere
\]
\end{proof}

\begin{lemma}\label[lemma]{lem:groups}
For $m\ge0$, $Q/\ncl{r^m}\cong\langle r,b\mid r^m,\ brb^{-1}=r^{-1}\rangle=\Z/m\rtimes_{-1}\Z$.  Its abelianization is $\Z$ if $m$ is odd and $\Z\oplus\Z/2$ if $m$ is even (including $m=0$).  The groups $\Z/m\rtimes_{-1}\Z$, $m\ge0$, are pairwise non-isomorphic.
\end{lemma}

\begin{proof}
By \eqref{eq:Q}, $Q=\langle r,b\mid brb^{-1}=r^{-1}\rangle$, and dividing by the
normal closure of $r^m$ adds the single relator $r^m$, which gives the stated
presentation.  This is the standard presentation of the semidirect product
$\Z/m\rtimes_{-1}\Z$, with $r$ generating the normal factor $\Z/m$ (which is
$\Z$ if $m=0$) and $b$ generating the factor $\Z$, acting on it by
inversion.  Abelianizing imposes $2r=0$ and $mr=0$, so $r$ has order $\gcd(2,m)$, giving the stated abelianizations.  Every element of finite order lies in $\langle r\rangle$ (its image in $\Z$ must vanish), so for $m\ge2$ the maximal order of a torsion element is $m$, and these groups are pairwise non-isomorphic.  The remaining two groups are torsion-free: for $m=1$ the group is $\Z$, while for $m=0$ it is the non-abelian Klein bottle group $\langle r,b\mid brb^{-1}=r^{-1}\rangle$. They are distinguished from each other, and from the groups with $m\ge2$, by torsion and by their abelianizations $\Z$ and $\Z\oplus\Z/2$.
\end{proof}

\subsection{Proof of \texorpdfstring{\cref{thm:newfillings}}{Theorem 1.3}}\label{subsec:proof-newfillings}

For $n\ge0$ let
\[
f_n=\sigma_1\,t_\alpha^{\,n}.
\]
By \cref{prop:commute} and Baykur's $[B,t_\alpha]=1$, $f_n$ commutes with $B$,
so by \cref{cor:factorizations} $P_{f_n}$ is a positive factorization of
$\phi$ and $X_{f_n}$ is a Stein filling of $(Y,\xi)$ with $\chi(X_{f_n})=1$.
Since $f_n^{-1}=t_\alpha^{-n}\sigma_1^{-1}$, \cref{lem:images} shows that the
image of $f_n^{-1}(d)$ in $Q$ is $r^{-2n}$, so by
\cref{lem:pi1,lem:groups}
\[
\pi_1(X_{f_n})\cong Q/\ncl{r^{2n}}\cong\Z/2n\rtimes_{-1}\Z,
\]
with abelianization $\Z\oplus\Z/2$.  For $m$ odd let $W_m:=X_{(m-1)/2}$ be
Baykur's filling, with $\pi_1\cong\Z/m\rtimes_{-1}\Z$ by
\cite[(19)]{BaykurPlanar}, and for $m=2n$ even let $W_m:=X_{f_n}$.  In both
cases $H_1(W_m)$ has rank one and $W_m$ has a handle decomposition with
handles of index $\le2$, so $H_2(W_m)$ is free of rank $\chi-1+b_1=1$.
Pairwise non-homeomorphy follows from \cref{lem:groups}. \qed

\begin{example}\label[example]{prop:m0}
Take $f_0=\sigma_1$ in \cref{cor:factorizations}, the case $m=0$ of the proof
above.  Recall that $\sigma_1$ is the clockwise half twist interchanging the
holes $\delta_1$ and $\delta_2$.  Applying $\sigma_1$ to Baykur's curves
$u,v,w,z$ of \cref{fig:baykur-curves}, we get the curves $u'=d$, $v'$, $w'$ and
$z'$ of \cref{fig:newfact}.  In the free generators $x_1,\ldots,x_5$ of $F_5$
shown in \cref{fig:alpha-beta-x}, these curves are represented by
\[
d=x_3x_1,\qquad v'=x_5x_4x_1,\qquad w'=x_5x_4x_3x_2x_4^{-1},\qquad z'=x_4x_2 .
\]
Since $\sigma_1$ commutes with $B$, we have $B=\sigma_1B\sigma_1^{-1}=t_d\,t_{v'}t_{w'}t_{z'}$, and
\[
\phi=t_dB=t_d\,t_d\,t_{v'}\,t_{w'}\,t_{z'}.
\]
This is the factorization of the filling $W_0$ with
$\pi_1(W_0)\cong\pi_1(K)$.  This can also be seen directly: in
$F_5/\ncl{d,v',w',z'}$ the relations $d=1$ and $z'=1$ give $x_3=x_1^{-1}$ and
$x_4=x_2^{-1}$, then $v'=1$ gives $x_5=x_1^{-1}x_2$, and $w'=1$ becomes
$x_2x_1^{-2}x_2=1$, i.e., $x_1^2=x_2^2$.  Hence the group is
$\langle x_1,x_2\mid x_1^2=x_2^2\rangle=\pi_1(K)$.
\end{example}

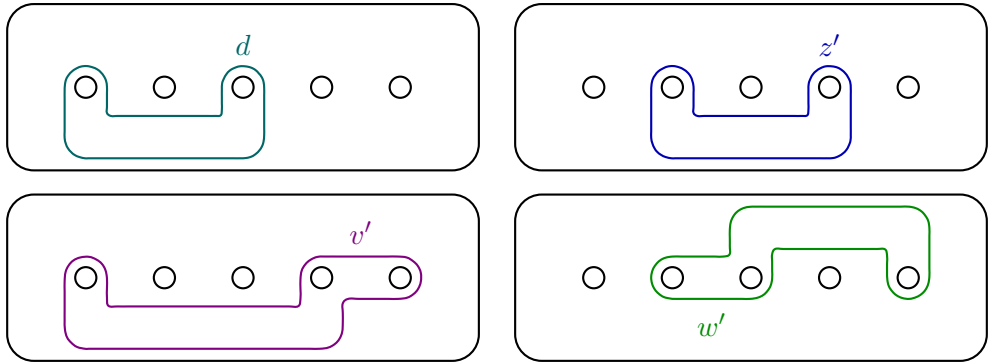
\begin{figure}[ht]
\centering
\begin{tikzpicture}[scale=0.4]
\begin{scope}[shift={(0.0,0.0)}]
\draw[thick,rounded corners=10pt] (0,-2.75) rectangle (15.600000000000001,2.75);
\draw[thick] (2.60,0.00) circle (0.35);
\draw[thick] (5.20,0.00) circle (0.35);
\draw[thick] (7.80,0.00) circle (0.35);
\draw[thick] (10.40,0.00) circle (0.35);
\draw[thick] (13.00,0.00) circle (0.35);
\draw[teal!80!black,thick] plot[smooth cycle,tension=0.7] coordinates{(2.57,-2.35) (2.39,-2.31) (2.23,-2.24) (2.11,-2.14) (2.01,-2.02) (1.94,-1.87) (1.90,-1.68) (1.90,-1.46) (1.90,-1.24) (1.90,-1.02) (1.90,-0.80) (1.90,-0.58) (1.90,-0.36) (1.90,-0.14) (1.91,0.08) (1.95,0.26) (2.03,0.40) (2.14,0.52) (2.27,0.61) (2.43,0.68) (2.63,0.70) (2.82,0.66) (2.98,0.59) (3.10,0.49) (3.20,0.36) (3.27,0.21) (3.30,0.02) (3.30,-0.20) (3.30,-0.42) (3.30,-0.64) (3.30,-0.85) (3.43,-0.95) (3.66,-0.95) (3.89,-0.95) (4.12,-0.95) (4.35,-0.95) (4.58,-0.95) (4.80,-0.95) (5.03,-0.95) (5.26,-0.95) (5.49,-0.95) (5.72,-0.95) (5.95,-0.95) (6.18,-0.95) (6.40,-0.95) (6.63,-0.95) (6.86,-0.95) (7.06,-0.93) (7.10,-0.75) (7.10,-0.53) (7.10,-0.31) (7.10,-0.09) (7.11,0.13) (7.17,0.29) (7.25,0.43) (7.36,0.55) (7.50,0.63) (7.68,0.69) (7.88,0.69) (8.06,0.65) (8.21,0.57) (8.32,0.46) (8.42,0.33) (8.48,0.17) (8.50,-0.04) (8.50,-0.26) (8.50,-0.48) (8.50,-0.70) (8.50,-0.92) (8.50,-1.14) (8.50,-1.36) (8.50,-1.58) (8.48,-1.79) (8.43,-1.95) (8.34,-2.09) (8.23,-2.20) (8.09,-2.29) (7.91,-2.34) (7.70,-2.35) (7.47,-2.35) (7.24,-2.35) (7.02,-2.35) (6.79,-2.35) (6.56,-2.35) (6.33,-2.35) (6.10,-2.35) (5.87,-2.35) (5.65,-2.35) (5.42,-2.35) (5.19,-2.35) (4.96,-2.35) (4.73,-2.35) (4.50,-2.35) (4.28,-2.35) (4.05,-2.35) (3.82,-2.35) (3.59,-2.35) (3.36,-2.35) (3.13,-2.35) (2.91,-2.35) (2.68,-2.35)};
\node[teal!80!black] at (7.80,1.40) {$d$};
\end{scope}
\begin{scope}[shift={(16.8,0.0)}]
\draw[thick,rounded corners=10pt] (0,-2.75) rectangle (15.600000000000001,2.75);
\draw[thick] (2.60,0.00) circle (0.35);
\draw[thick] (5.20,0.00) circle (0.35);
\draw[thick] (7.80,0.00) circle (0.35);
\draw[thick] (10.40,0.00) circle (0.35);
\draw[thick] (13.00,0.00) circle (0.35);
\draw[blue!70!black,thick] plot[smooth cycle,tension=0.7] coordinates{(5.17,-2.35) (4.98,-2.31) (4.83,-2.24) (4.70,-2.14) (4.61,-2.02) (4.54,-1.86) (4.50,-1.68) (4.50,-1.46) (4.50,-1.24) (4.50,-1.02) (4.50,-0.80) (4.50,-0.58) (4.50,-0.36) (4.50,-0.14) (4.51,0.08) (4.55,0.26) (4.63,0.40) (4.73,0.52) (4.87,0.61) (5.03,0.68) (5.23,0.70) (5.42,0.66) (5.57,0.59) (5.70,0.49) (5.80,0.36) (5.86,0.21) (5.90,0.02) (5.90,-0.20) (5.90,-0.42) (5.90,-0.64) (5.90,-0.85) (6.03,-0.95) (6.26,-0.95) (6.49,-0.95) (6.72,-0.95) (6.94,-0.95) (7.17,-0.95) (7.40,-0.95) (7.63,-0.95) (7.86,-0.95) (8.09,-0.95) (8.31,-0.95) (8.54,-0.95) (8.77,-0.95) (9.00,-0.95) (9.23,-0.95) (9.46,-0.95) (9.66,-0.93) (9.70,-0.75) (9.70,-0.53) (9.70,-0.31) (9.70,-0.09) (9.71,0.12) (9.77,0.29) (9.85,0.43) (9.96,0.54) (10.10,0.63) (10.27,0.69) (10.48,0.69) (10.66,0.65) (10.80,0.57) (10.92,0.46) (11.01,0.33) (11.08,0.17) (11.10,-0.03) (11.10,-0.25) (11.10,-0.47) (11.10,-0.69) (11.10,-0.91) (11.10,-1.13) (11.10,-1.35) (11.10,-1.57) (11.09,-1.77) (11.03,-1.94) (10.95,-2.08) (10.84,-2.19) (10.70,-2.28) (10.53,-2.34) (10.32,-2.35) (10.09,-2.35) (9.86,-2.35) (9.63,-2.35) (9.40,-2.35) (9.18,-2.35) (8.95,-2.35) (8.72,-2.35) (8.49,-2.35) (8.26,-2.35) (8.03,-2.35) (7.81,-2.35) (7.58,-2.35) (7.35,-2.35) (7.12,-2.35) (6.89,-2.35) (6.66,-2.35) (6.43,-2.35) (6.21,-2.35) (5.98,-2.35) (5.75,-2.35) (5.52,-2.35) (5.29,-2.35)};
\node[blue!70!black] at (10.40,1.40) {$z'$};
\end{scope}
\begin{scope}[shift={(0.0,-6.3)}]
\draw[thick,rounded corners=10pt] (0,-2.75) rectangle (15.600000000000001,2.75);
\draw[thick] (2.60,0.00) circle (0.35);
\draw[thick] (5.20,0.00) circle (0.35);
\draw[thick] (7.80,0.00) circle (0.35);
\draw[thick] (10.40,0.00) circle (0.35);
\draw[thick] (13.00,0.00) circle (0.35);
\draw[violet,thick] plot[smooth cycle,tension=0.7] coordinates{(2.57,-2.35) (2.31,-2.29) (2.12,-2.15) (1.98,-1.97) (1.91,-1.71) (1.90,-1.40) (1.90,-1.08) (1.90,-0.76) (1.90,-0.44) (1.90,-0.12) (1.93,0.18) (2.03,0.40) (2.19,0.57) (2.42,0.67) (2.71,0.69) (2.95,0.60) (3.13,0.46) (3.25,0.25) (3.30,-0.02) (3.30,-0.34) (3.30,-0.66) (3.35,-0.93) (3.66,-0.95) (4.00,-0.95) (4.33,-0.95) (4.66,-0.95) (4.99,-0.95) (5.32,-0.95) (5.66,-0.95) (5.99,-0.95) (6.32,-0.95) (6.65,-0.95) (6.99,-0.95) (7.32,-0.95) (7.65,-0.95) (7.98,-0.95) (8.31,-0.95) (8.65,-0.95) (8.98,-0.95) (9.31,-0.95) (9.63,-0.94) (9.70,-0.69) (9.70,-0.37) (9.70,-0.05) (9.74,0.23) (9.86,0.44) (10.03,0.59) (10.27,0.69) (10.59,0.70) (10.92,0.70) (11.25,0.70) (11.58,0.70) (11.92,0.70) (12.25,0.70) (12.58,0.70) (12.91,0.70) (13.21,0.67) (13.42,0.55) (13.58,0.39) (13.68,0.16) (13.68,-0.13) (13.59,-0.37) (13.44,-0.54) (13.23,-0.66) (12.94,-0.70) (12.61,-0.70) (12.28,-0.70) (11.95,-0.70) (11.62,-0.70) (11.28,-0.70) (11.10,-0.85) (11.10,-1.17) (11.10,-1.49) (11.08,-1.79) (10.99,-2.02) (10.84,-2.19) (10.62,-2.31) (10.34,-2.35) (10.01,-2.35) (9.67,-2.35) (9.34,-2.35) (9.01,-2.35) (8.68,-2.35) (8.35,-2.35) (8.01,-2.35) (7.68,-2.35) (7.35,-2.35) (7.02,-2.35) (6.68,-2.35) (6.35,-2.35) (6.02,-2.35) (5.69,-2.35) (5.36,-2.35) (5.02,-2.35) (4.69,-2.35) (4.36,-2.35) (4.03,-2.35) (3.69,-2.35) (3.36,-2.35) (3.03,-2.35) (2.70,-2.35)};
\node[violet] at (11.70,1.50) {$v'$};
\end{scope}
\begin{scope}[shift={(16.8,-6.3)}]
\draw[thick,rounded corners=10pt] (0,-2.75) rectangle (15.600000000000001,2.75);
\draw[thick] (2.60,0.00) circle (0.35);
\draw[thick] (5.20,0.00) circle (0.35);
\draw[thick] (7.80,0.00) circle (0.35);
\draw[thick] (10.40,0.00) circle (0.35);
\draw[thick] (13.00,0.00) circle (0.35);
\draw[green!55!black,thick] plot[smooth cycle,tension=0.7] coordinates{(5.14,-0.69) (4.92,-0.64) (4.75,-0.53) (4.62,-0.39) (4.53,-0.20) (4.50,0.03) (4.55,0.26) (4.65,0.43) (4.79,0.57) (4.97,0.66) (5.21,0.70) (5.49,0.70) (5.77,0.70) (6.05,0.70) (6.33,0.70) (6.61,0.70) (6.89,0.70) (7.09,0.77) (7.10,1.04) (7.10,1.31) (7.10,1.58) (7.13,1.83) (7.21,2.02) (7.33,2.17) (7.50,2.28) (7.71,2.34) (7.98,2.35) (8.26,2.35) (8.54,2.35) (8.82,2.35) (9.10,2.35) (9.38,2.35) (9.66,2.35) (9.94,2.35) (10.22,2.35) (10.50,2.35) (10.78,2.35) (11.06,2.35) (11.35,2.35) (11.63,2.35) (11.91,2.35) (12.19,2.35) (12.47,2.35) (12.75,2.35) (13.02,2.35) (13.25,2.30) (13.43,2.20) (13.56,2.06) (13.66,1.88) (13.70,1.65) (13.70,1.38) (13.70,1.11) (13.70,0.84) (13.70,0.57) (13.70,0.30) (13.70,0.03) (13.66,-0.21) (13.57,-0.40) (13.44,-0.54) (13.27,-0.64) (13.04,-0.70) (12.80,-0.67) (12.61,-0.58) (12.47,-0.45) (12.36,-0.28) (12.31,-0.07) (12.30,0.20) (12.30,0.47) (12.30,0.74) (12.23,0.94) (11.96,0.95) (11.68,0.95) (11.40,0.95) (11.12,0.95) (10.84,0.95) (10.56,0.95) (10.28,0.95) (10.00,0.95) (9.72,0.95) (9.44,0.95) (9.15,0.95) (8.87,0.95) (8.60,0.95) (8.50,0.78) (8.50,0.51) (8.50,0.24) (8.50,-0.03) (8.45,-0.26) (8.35,-0.43) (8.21,-0.57) (8.02,-0.66) (7.78,-0.70) (7.50,-0.70) (7.22,-0.70) (6.94,-0.70) (6.66,-0.70) (6.38,-0.70) (6.10,-0.70) (5.82,-0.70) (5.54,-0.70) (5.26,-0.70)};
\node[green!55!black] at (6.50,-1.55) {$w'$};
\end{scope}
\end{tikzpicture}
\caption{The vanishing cycles $d$ (twice), $v'$, $w'$, $z'$ of the factorization $\phi=t_d^2\,t_{v'}t_{w'}t_{z'}$ of \cref{prop:m0}, whose Stein filling has fundamental group $\pi_1(K)$.}
\label{fig:newfact}
\end{figure}

\subsection{The filling \texorpdfstring{$W_0$}{W\_0} and the Seifert fibration of \texorpdfstring{$Y$}{Y}}\label{subsec:W0}

The filling $W_0=X(\Sigma_0^6;d,d,v',w',z')$ of \cref{prop:m0} has a
particularly simple description, which exhibits the Seifert fibration of $Y$.
Recall from the proof of \cref{thm:klein-main} that \cref{fig:klein-kirby}
is the standard Kirby diagram of the disk bundle $DT^*K$ over the Klein bottle
with Euler number zero \cite[\S4.6]{GS}.  In this diagram the $2$--handle is
the part $D\times D^2$ of the bundle over a $2$--cell $D$ of $K$.  Hence its
cocore is a fiber disk $\{q\}\times D^2$, and its belt circle, i.e., a
meridian of its attaching circle, is a fiber of the circle bundle
$\partial DT^*K=ST^*K\to K$.  Moreover, the framing $0$ of this meridian in the
diagram is the framing induced by a disk in the boundary of the $0$--handle
which meets the attaching circle once.  This is the framing induced by the
cocore, which is the framing of the fiber given by the bundle, since a nearby
fiber bounds a disjoint fiber disk.

\begin{figure}[htbp]
\centering
\begin{tikzpicture}[xscale=0.72,line cap=round, line join=round]
\draw[teal!80!black,line width=0.8pt] plot coordinates{(2.127,6.062) (2.200,6.058) (2.324,6.047) (2.418,6.031) (2.496,6.010) (2.551,5.986) (2.574,5.955) (2.574,5.920) (2.574,5.884) (2.574,5.849) (2.574,5.815) (2.675,5.799) (2.855,5.799) (3.034,5.799) (3.214,5.799) (3.393,5.799) (3.572,5.799) (3.744,5.799) (3.923,5.799) (4.103,5.799) (4.282,5.799) (4.462,5.799) (4.641,5.799) (4.820,5.799) (4.992,5.799) (5.171,5.799) (5.351,5.799) (5.507,5.802) (5.538,5.831) (5.538,5.867) (5.538,5.902) (5.538,5.937) (5.546,5.973) (5.593,5.999) (5.655,6.021) (5.741,6.041) (5.850,6.054) (5.990,6.063) (6.005,6.063)};
\draw[teal!80!black,line width=0.8pt] plot coordinates{(6.184,6.062) (6.287,6.057) (6.404,6.044) (6.490,6.026) (6.568,6.005) (6.614,5.979) (6.630,5.946) (6.630,5.910) (6.630,5.875) (6.630,5.839) (6.630,5.804) (6.630,5.768) (6.630,5.733) (6.630,5.697) (6.614,5.663) (6.575,5.638) (6.505,5.615) (6.419,5.597) (6.310,5.583) (6.170,5.575) (6.006,5.573) (5.827,5.573) (5.647,5.573) (5.476,5.573) (5.296,5.573) (5.117,5.573) (4.937,5.573) (4.758,5.573) (4.579,5.573) (4.407,5.573) (4.228,5.573) (4.048,5.573) (3.869,5.573) (3.689,5.573) (3.510,5.573) (3.338,5.573) (3.159,5.573) (2.980,5.573) (2.800,5.573) (2.621,5.573) (2.441,5.573) (2.270,5.573) (2.090,5.573) (2.005,5.573) (1.864,5.580) (1.739,5.591) (1.646,5.607) (1.568,5.626) (1.513,5.651) (1.482,5.681) (1.482,5.717) (1.482,5.752) (1.482,5.788) (1.482,5.823) (1.482,5.859) (1.482,5.894) (1.482,5.929) (1.490,5.965) (1.521,5.994) (1.583,6.016) (1.669,6.036) (1.771,6.050) (1.895,6.062) (1.949,6.063)};
\draw[orange!85!black,line width=0.8pt] plot coordinates{(2.127,4.946) (2.200,4.942) (2.324,4.931) (2.418,4.915) (2.496,4.894) (2.551,4.870) (2.574,4.839) (2.574,4.804) (2.574,4.768) (2.574,4.733) (2.574,4.699) (2.675,4.683) (2.855,4.683) (3.034,4.683) (3.214,4.683) (3.393,4.683) (3.572,4.683) (3.744,4.683) (3.923,4.683) (4.103,4.683) (4.282,4.683) (4.462,4.683) (4.641,4.683) (4.820,4.683) (4.992,4.683) (5.171,4.683) (5.351,4.683) (5.507,4.686) (5.538,4.715) (5.538,4.751) (5.538,4.786) (5.538,4.821) (5.546,4.857) (5.593,4.883) (5.655,4.905) (5.741,4.925) (5.850,4.938) (5.990,4.947) (6.005,4.947)};
\draw[orange!85!black,line width=0.8pt] plot coordinates{(6.184,4.946) (6.287,4.941) (6.404,4.928) (6.490,4.910) (6.568,4.889) (6.614,4.863) (6.630,4.830) (6.630,4.794) (6.630,4.759) (6.630,4.723) (6.630,4.688) (6.630,4.652) (6.630,4.617) (6.630,4.581) (6.614,4.547) (6.575,4.522) (6.505,4.499) (6.419,4.481) (6.310,4.467) (6.170,4.459) (6.006,4.457) (5.827,4.457) (5.647,4.457) (5.476,4.457) (5.296,4.457) (5.117,4.457) (4.937,4.457) (4.758,4.457) (4.579,4.457) (4.407,4.457) (4.228,4.457) (4.048,4.457) (3.869,4.457) (3.689,4.457) (3.510,4.457) (3.338,4.457) (3.159,4.457) (2.980,4.457) (2.800,4.457) (2.621,4.457) (2.441,4.457) (2.270,4.457) (2.090,4.457) (2.005,4.457) (1.864,4.464) (1.739,4.475) (1.646,4.491) (1.568,4.510) (1.513,4.535) (1.482,4.565) (1.482,4.601) (1.482,4.636) (1.482,4.672) (1.482,4.707) (1.482,4.743) (1.482,4.778) (1.482,4.813) (1.490,4.849) (1.521,4.878) (1.583,4.900) (1.669,4.920) (1.771,4.934) (1.895,4.946) (1.949,4.947)};
\draw[violet,line width=0.8pt] plot coordinates{(2.128,3.830) (2.301,3.817) (2.441,3.794) (2.535,3.760) (2.574,3.717) (2.574,3.665) (2.574,3.614) (2.613,3.570) (2.855,3.567) (3.120,3.567) (3.377,3.567) (3.635,3.567) (3.892,3.567) (4.150,3.567) (4.415,3.567) (4.672,3.567) (4.930,3.567) (5.187,3.567) (5.452,3.567) (5.710,3.567) (5.967,3.567) (6.224,3.567) (6.482,3.567) (6.747,3.567) (7.004,3.567) (7.262,3.567) (7.511,3.568) (7.566,3.609) (7.566,3.660) (7.566,3.712) (7.597,3.757) (7.691,3.791) (7.823,3.815) (8.011,3.831) (8.033,3.831)};
\draw[violet,line width=0.8pt] plot coordinates{(8.213,3.833) (8.260,3.833) (8.518,3.833) (8.775,3.833) (9.032,3.833) (9.298,3.833) (9.555,3.833) (9.812,3.833) (10.061,3.833)};
\draw[violet,line width=0.8pt] plot coordinates{(10.240,3.829) (10.304,3.828) (10.468,3.809) (10.592,3.783) (10.670,3.746) (10.670,3.699) (10.600,3.660) (10.483,3.633) (10.319,3.614) (10.093,3.607) (9.836,3.607) (9.578,3.607) (9.321,3.607) (9.064,3.607) (8.798,3.607) (8.658,3.583) (8.658,3.531) (8.658,3.480) (8.642,3.431) (8.572,3.394) (8.455,3.367) (8.284,3.348) (8.065,3.341) (7.808,3.341) (7.543,3.341) (7.285,3.341) (7.028,3.341) (6.770,3.341) (6.513,3.341) (6.248,3.341) (5.990,3.341) (5.733,3.341) (5.476,3.341) (5.210,3.341) (4.953,3.341) (4.696,3.341) (4.438,3.341) (4.181,3.341) (3.916,3.341) (3.658,3.341) (3.401,3.341) (3.143,3.341) (2.878,3.341) (2.621,3.341) (2.363,3.341) (2.106,3.341) (2.005,3.341) (1.802,3.351) (1.654,3.373) (1.544,3.402) (1.490,3.444) (1.482,3.494) (1.482,3.546) (1.482,3.597) (1.482,3.649) (1.482,3.701) (1.505,3.749) (1.583,3.784) (1.708,3.812) (1.888,3.828) (1.949,3.829)};
\draw[green!55!black,line width=0.8pt] plot coordinates{(4.157,2.531) (4.282,2.531) (4.501,2.531) (4.719,2.531) (4.937,2.531) (5.156,2.531) (5.374,2.531) (5.530,2.542) (5.538,2.586) (5.538,2.629) (5.538,2.673) (5.561,2.713) (5.624,2.744) (5.717,2.768) (5.850,2.786) (6.005,2.795)};
\draw[green!55!black,line width=0.8pt] plot coordinates{(6.185,2.797) (6.224,2.797) (6.443,2.797) (6.661,2.797) (6.880,2.797) (7.098,2.797) (7.316,2.797) (7.535,2.797) (7.753,2.797) (7.972,2.797) (8.033,2.797)};
\draw[green!55!black,line width=0.8pt] plot coordinates{(8.213,2.797) (8.408,2.797) (8.627,2.797) (8.853,2.797) (9.071,2.797) (9.290,2.797) (9.508,2.797) (9.727,2.797) (9.945,2.797) (10.061,2.797)};
\draw[green!55!black,line width=0.8pt] plot coordinates{(10.239,2.793) (10.335,2.789) (10.475,2.773) (10.577,2.750) (10.655,2.721) (10.686,2.684) (10.686,2.640) (10.686,2.597) (10.686,2.553) (10.686,2.510) (10.686,2.466) (10.686,2.423) (10.655,2.384) (10.585,2.354) (10.483,2.331) (10.351,2.315) (10.171,2.305) (9.984,2.310) (9.836,2.325) (9.727,2.345) (9.641,2.373) (9.602,2.407) (9.594,2.450) (9.594,2.494) (9.594,2.537) (9.539,2.570) (9.329,2.571) (9.110,2.571) (8.892,2.571) (8.674,2.571) (8.455,2.571) (8.237,2.571) (8.213,2.571)};
\draw[green!55!black,line width=0.8pt] plot coordinates{(8.033,2.571) (8.018,2.571) (7.800,2.571) (7.582,2.571) (7.363,2.571) (7.137,2.571) (6.919,2.571) (6.708,2.571) (6.630,2.544) (6.630,2.500) (6.630,2.457) (6.630,2.413) (6.591,2.376) (6.513,2.349) (6.404,2.326) (6.256,2.312) (6.068,2.305) (5.850,2.305) (5.632,2.305) (5.413,2.305) (5.195,2.305) (4.976,2.305) (4.758,2.305) (4.540,2.305) (4.321,2.305) (4.103,2.305) (4.009,2.307) (3.838,2.315) (3.705,2.333) (3.604,2.355) (3.533,2.386) (3.510,2.423) (3.549,2.460) (3.627,2.487) (3.736,2.510) (3.877,2.524) (3.978,2.528)};
\draw[blue!70!black,line width=0.8pt] plot coordinates{(4.155,1.288) (4.228,1.284) (4.345,1.273) (4.446,1.257) (4.524,1.236) (4.571,1.212) (4.602,1.181) (4.602,1.146) (4.602,1.110) (4.602,1.075) (4.602,1.041) (4.703,1.025) (4.883,1.025) (5.062,1.025) (5.242,1.025) (5.413,1.025) (5.593,1.025) (5.772,1.025) (5.951,1.025) (6.131,1.025) (6.310,1.025) (6.482,1.025) (6.661,1.025) (6.841,1.025) (7.020,1.025) (7.199,1.025) (7.379,1.025) (7.535,1.028) (7.566,1.057) (7.566,1.093) (7.566,1.128) (7.566,1.163) (7.574,1.197) (7.621,1.225) (7.683,1.247) (7.769,1.265) (7.878,1.280) (8.011,1.289) (8.033,1.289)};
\draw[blue!70!black,line width=0.8pt] plot coordinates{(8.212,1.288) (8.315,1.283) (8.424,1.270) (8.518,1.252) (8.588,1.231) (8.642,1.205) (8.658,1.173) (8.658,1.138) (8.658,1.102) (8.658,1.067) (8.658,1.031) (8.658,0.996) (8.658,0.960) (8.658,0.925) (8.650,0.893) (8.603,0.865) (8.541,0.843) (8.455,0.825) (8.346,0.810) (8.213,0.801) (8.050,0.799) (7.870,0.799) (7.691,0.799) (7.511,0.799) (7.332,0.799) (7.160,0.799) (6.981,0.799) (6.802,0.799) (6.622,0.799) (6.443,0.799) (6.263,0.799) (6.092,0.799) (5.912,0.799) (5.733,0.799) (5.554,0.799) (5.374,0.799) (5.195,0.799) (5.015,0.799) (4.844,0.799) (4.664,0.799) (4.485,0.799) (4.306,0.799) (4.126,0.799) (4.033,0.799) (3.884,0.806) (3.767,0.817) (3.666,0.833) (3.596,0.852) (3.541,0.878) (3.510,0.907) (3.510,0.943) (3.510,0.978) (3.510,1.014) (3.510,1.049) (3.510,1.085) (3.510,1.120) (3.510,1.155) (3.518,1.191) (3.549,1.220) (3.611,1.242) (3.689,1.262) (3.799,1.276) (3.923,1.288) (3.977,1.289)};
\node[teal!80!black,left] at (1.432,5.952) {$d$};
\node[teal!80!black,right] at (6.680,5.952) {$-1$};
\node[orange!85!black,left] at (1.432,4.836) {$d^*$};
\node[orange!85!black,right] at (6.680,4.836) {$-1$};
\node[violet,left] at (1.432,3.720) {$v'$};
\node[violet,right] at (10.720,3.720) {$-1$};
\node[green!55!black,left] at (3.460,2.418) {$w'$};
\node[green!55!black,right] at (10.736,2.418) {$-1$};
\node[blue!70!black,left] at (3.460,1.178) {$z'$};
\node[blue!70!black,right] at (8.708,1.178) {$-1$};
\draw[black,line width=0.8pt] (2.028,0.434) -- (2.028,3.241);
\draw[black,line width=0.8pt] (2.028,3.441) -- (2.028,4.357);
\draw[black,line width=0.8pt] (2.028,4.557) -- (2.028,5.473);
\draw[black,line width=0.8pt] (2.028,5.673) -- (2.028,6.944);
\fill (2.028,6.594) circle (2.2pt);
\node[above] at (2.028,7.006) {$\delta_{1}$};
\draw[black,line width=0.8pt] (4.056,0.434) -- (4.056,0.699);
\draw[black,line width=0.8pt] (4.056,0.899) -- (4.056,2.206);
\draw[black,line width=0.8pt] (4.056,2.406) -- (4.056,3.241);
\draw[black,line width=0.8pt] (4.056,3.441) -- (4.056,3.467);
\draw[black,line width=0.8pt] (4.056,3.667) -- (4.056,4.357);
\draw[black,line width=0.8pt] (4.056,4.557) -- (4.056,4.583);
\draw[black,line width=0.8pt] (4.056,4.783) -- (4.056,5.473);
\draw[black,line width=0.8pt] (4.056,5.673) -- (4.056,5.699);
\draw[black,line width=0.8pt] (4.056,5.899) -- (4.056,6.944);
\fill (4.056,6.594) circle (2.2pt);
\node[above] at (4.056,7.006) {$\delta_{2}$};
\draw[black,line width=0.8pt] (6.084,0.434) -- (6.084,0.699);
\draw[black,line width=0.8pt] (6.084,0.899) -- (6.084,0.925);
\draw[black,line width=0.8pt] (6.084,1.125) -- (6.084,2.206);
\draw[black,line width=0.8pt] (6.084,2.406) -- (6.084,3.241);
\draw[black,line width=0.8pt] (6.084,3.441) -- (6.084,3.467);
\draw[black,line width=0.8pt] (6.084,3.667) -- (6.084,4.358);
\draw[black,line width=0.8pt] (6.084,4.558) -- (6.084,5.474);
\draw[black,line width=0.8pt] (6.084,5.674) -- (6.084,6.944);
\fill (6.084,6.594) circle (2.2pt);
\node[above] at (6.084,7.006) {$\delta_{3}$};
\draw[black,line width=0.8pt] (8.112,0.434) -- (8.112,0.700);
\draw[black,line width=0.8pt] (8.112,0.900) -- (8.112,3.243);
\draw[black,line width=0.8pt] (8.112,3.443) -- (8.112,6.944);
\fill (8.112,6.594) circle (2.2pt);
\node[above] at (8.112,7.006) {$\delta_{4}$};
\draw[black,line width=0.8pt] (10.140,0.434) -- (10.140,2.206);
\draw[black,line width=0.8pt] (10.140,2.406) -- (10.140,3.509);
\draw[black,line width=0.8pt] (10.140,3.709) -- (10.140,6.944);
\fill (10.140,6.594) circle (2.2pt);
\node[above] at (10.140,7.006) {$\delta_{5}$};
\end{tikzpicture}
\caption{The Kirby diagram of the Lefschetz fibration $W_0$ of
\cref{prop:m0}.  The five vertical arcs are the dotted circles of the
$1$--handles dual to $\delta_1,\ldots,\delta_5$ (drawn through infinity), and
the $2$--handles are attached along the vanishing cycles $d$, $d^*$, $v'$,
$w'$, $z'$ of \cref{fig:newfact}, where $d^*$ is a second copy of $d$, on successive pages from top to bottom, with
framing $-1$ relative to the page.}
\label{fig:W0-lefschetz}
\end{figure}
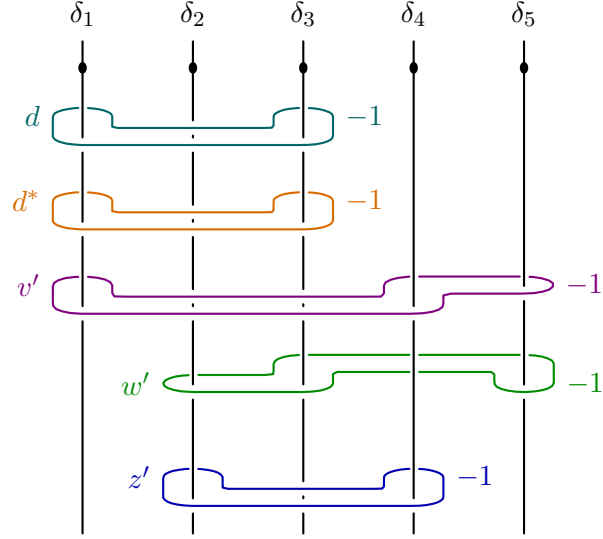

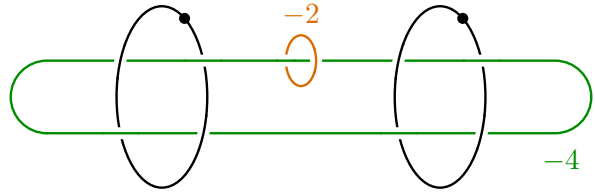
\begin{figure}[htbp]
\centering
\begin{tikzpicture}[scale=0.8, line cap=round, line join=round]
\def\cc{2.3}
\draw[green!55!black,line width=0.9pt]
  (-4.2,-0.6) -- (4.2,-0.6) arc[start angle=-90,end angle=90,radius=0.6]
  -- (-4.2,0.6) arc[start angle=90,end angle=270,radius=0.6];
\draw[black,line width=0.9pt] (\cc,0) ellipse[x radius=0.75,y radius=1.5];
\draw[black,line width=0.9pt] (-\cc,0) ellipse[x radius=0.75,y radius=1.5];
\fill ({\cc+0.75*cos(60)},{1.5*sin(60)}) circle (2.6pt);
\fill ({-\cc+0.75*cos(60)},{1.5*sin(60)}) circle (2.6pt);
\foreach \x/\y in {-2.987/-0.6, -1.613/0.6, 1.613/-0.6, 2.987/0.6}{
  \draw[white,line width=4.5pt,line cap=butt] (\x-0.28,\y) -- (\x+0.28,\y);
  \draw[green!55!black,line width=0.9pt] (\x-0.3,\y) -- (\x+0.3,\y);}
\foreach \cx/\t in {-2.3/156.42, -2.3/-23.58, 2.3/156.42, 2.3/-23.58}{
  \draw[white,line width=4.5pt,line cap=butt,domain=\t-11:\t+11,samples=12]
    plot ({\cx+0.75*cos(\x)},{1.5*sin(\x)});
  \draw[black,line width=0.9pt,domain=\t-12:\t+12,samples=12]
    plot ({\cx+0.75*cos(\x)},{1.5*sin(\x)});}
\draw[orange!85!black,line width=0.9pt] (0,0.6) ellipse[x radius=0.25,y radius=0.42];
\draw[white,line width=4.5pt,line cap=butt] (-0.37,0.6) -- (-0.13,0.6);
\draw[green!55!black,line width=0.9pt] (-0.4,0.6) -- (-0.1,0.6);
\draw[white,line width=4.5pt,line cap=butt,domain=-25:25,samples=10]
  plot ({0.25*cos(\x)},{0.6+0.42*sin(\x)});
\draw[orange!85!black,line width=0.9pt,domain=-30:30,samples=10]
  plot ({0.25*cos(\x)},{0.6+0.42*sin(\x)});
\node[green!55!black] at (4.3,-1.05) {$-4$};
\node[orange!85!black] at (0,1.35) {$-2$};
\end{tikzpicture}
\caption{A Kirby diagram of the filling $W_0$, obtained from the Lefschetz
fibration of \cref{prop:m0} by the handle moves in the proof of
\cref{prop:W0}, which is also the diagram of $DT^*K$ of \cref{fig:klein-kirby} together
with a $2$--handle of framing $-2$ attached along a meridian of its
$2$--handle, i.e., along a circle fiber of $ST^*K\to K$.}
\label{fig:W0-kirby}
\end{figure}

\begin{proposition}\label[proposition]{prop:W0}
The filling $W_0$ is diffeomorphic to $DT^*K\cup h$, where the $2$--handle
$h$ is attached along a fiber $F$ of the circle bundle $ST^*K\to K$ with
framing $-2$ relative to the framing of $F$ induced by the bundle (see
\cref{fig:W0-kirby}).  Consequently $Y$ is orientation-preservingly
diffeomorphic to the $3$--manifold obtained from $ST^*K$ by Dehn surgery along
a regular fiber, with coefficient $-2$ relative to the bundle framing, and
$H_2(W_0;\mathbb Z)\cong\mathbb Z$ is generated by an embedded sphere of
self-intersection $-2$, the union of the core of $h$ and a fiber disk of
$DT^*K$.
\end{proposition}

\begin{proof}
Note that \Cref{fig:W0-lefschetz} is the Kirby diagram of the Lefschetz fibration
$W_0$ given by the factorization $\phi=t_d\,t_d\,t_{v'}t_{w'}t_{z'}$: five
$1$--handles dual to the inner holes $\delta_1,\ldots,\delta_5$ and five
$2$--handles $h_d,h_d^*,h_{v'},h_{w'},h_{z'}$ attached along the vanishing
cycles with framing $-1$ relative to the page.  Its boundary is $Y$.

We reduce this diagram by Kirby calculus as in the proof of
\cref{thm:klein-main}, conjugated by $\sigma_1$, with one additional
slide, except that in the last cancellation we use the $1$--handle dual to
$\delta_1$ instead of $\delta_3$.  The curves $v',w'$ are the only ones enclosing $\delta_5$: slide
$h_{w'}$ over $h_{v'}$ to obtain $h_{w'}'$, which no longer runs over the
$1$--handle dual to $\delta_5$, and cancel $h_{v'}$ against that $1$--handle.
Now $h_{w'}'$ and $h_{z'}$ are the only $2$--handles running over the
$1$--handle dual to $\delta_4$: slide $h_{w'}'$ over $h_{z'}$ to obtain
$h_{w'}''$ and cancel $h_{z'}$ against that $1$--handle.

The $2$--handles $h_d$ and $h_d^*$ are attached along parallel copies of $d$
on consecutive pages, so their linking number is the page framing of $d$,
which is $0$.  Slide $h_d^*$ over $h_d$, subtracting, along a short band
joining the two parallel copies.  Since the framing push-off of $d$ differs
from the page push-off by one full twist around $d$, the resulting $2$--handle
is attached, after an isotopy, along a meridian $\mu$ of the attaching circle
of $h_d$.  It runs over no $1$--handle, and its framing is $-1-1-2\cdot0=-2$.

Finally $h_{w'}''$ and $h_d$ are the only $2$--handles running over the
$1$--handle dual to $\delta_1$: slide $h_{w'}''$ over $h_d$ to obtain
$h_{w'}'''$, which then also passes once through $\mu$, and cancel $h_d$
against that $1$--handle.  Before cancelling, isotope $\mu$ off the attaching
circle of $h_d$ across the belt sphere of this $1$--handle, so that after the
cancellation $\mu$ is a meridian of the attaching circle of $h_{w'}'''$.  What
remains is $\natural^2(S^1\times D^3)$, with $1$--handles dual to
$\delta_2,\delta_3$, together with two $2$--handles: $h_{w'}'''$ of framing
$-4$, following the three slides ($-1\to-2\to-3\to-4$), which runs twice in the
same direction over each surviving $1$--handle, and the $2$--handle along its
meridian $\mu$ of framing $-2$.  An isotopy carries the resulting diagram to
\cref{fig:W0-kirby}. The sequence of handle moves, drawn by hand, is
included as an ancillary file with the arXiv version of this paper. 

By the discussion preceding the proposition, $\mu$ is a fiber $F$ of
$ST^*K\to K$, and its framing $-2$ is measured relative to the bundle framing.
This proves the first statement, and the second follows since Kirby moves
preserve orientation and $\partial W_0=Y$.  Finally, the core of $h$ together
with the fiber disk bounded by $F$ is an embedded sphere of
self-intersection $-2$, which generates $H_2(W_0;\mathbb Z)\cong\mathbb Z$,
since $W_0$ is obtained from $DT^*K\simeq K$ by attaching a single $2$--handle
along a null-homologous circle.
\end{proof}

\begin{proof}[Proof of \cref{prop:baykur-seifert}]
The first statement is part of \cref{prop:W0}.  By \cref{thm:klein-main}, $ST^*K$ is the orientable circle bundle over $K$
with Euler number zero.  In other words, it is Seifert fibered over $K$ with no
exceptional fibers and Seifert Euler parameter $b=0$.  By \cref{prop:W0}, $Y$
is obtained from $ST^*K$ by removing a fibered solid torus neighborhood $N$ of
a regular fiber $F$ and regluing a solid torus whose meridian is the curve
$-2\mu_F+\lambda_F$ on $\partial N$.  Here $\mu_F$ is a meridian of $F$ and
$\lambda_F$, the bundle framing, is a regular fiber on $\partial N$.  The
algebraic intersection number of the new meridian with $\lambda_F$ is $\pm2$.
Hence the Seifert fibration of $ST^*K\setminus N$ extends over the new solid
torus, whose core is an exceptional fiber of multiplicity $2$.  We conclude
that $Y$ is Seifert fibered over $K$ with exactly one exceptional fiber, of
multiplicity $2$.  Its normalized invariants are $(b;(2,1))$ with
$b\in\{0,-1\}$, depending on whether the unnormalized invariant of the new
fiber is $(2,1)$ or $(2,-1)$ in the orientation conventions of
\cite[Proposition~4.3]{Wilkes}.  Reversing the orientation replaces
$(b;(2,1))$ by $(-b-1;(2,1))$.  Therefore, in either case, $Y$ is
homeomorphic, as an unoriented manifold, to the Seifert fibered space over the
Klein bottle with one exceptional fiber of type $(2,1)$ and $b=0$.
\end{proof}

\subsection{The homology of all Stein fillings}\label{subsec:allfillings}

The fillings $W_m$ of \cref{thm:newfillings} have first homology $\Z$ for $m$
odd and $\Z\oplus\Z/2$ for $m$ even.  We show that these are the only
possibilities for any Stein filling of $(Y,\xi)$.

\begin{proposition}\label[proposition]{prop:allfillings}
Every Stein filling $X$ of $(Y,\xi)$ has Euler characteristic $1$, first
homology isomorphic to $\Z$ or to $\Z\oplus\Z/2$, $H_2(X;\Z)\cong\Z$, and
$H_k(X;\Z)=0$ for $k\ge3$.
\end{proposition}

By Wendl's theorem \cite{Wendl}, every Stein filling $X$ of the planar contact
$3$--manifold $(Y,\xi)$ is, up to deformation, an allowable Lefschetz fibration
over the disk with fiber $\Sigma=\Sigma_0^6$ and monodromy $\phi$.  In other
words, it arises from a positive factorization $\phi=t_{c_1}\cdots t_{c_N}$
into right-handed Dehn twists along homologically nontrivial simple closed
curves.  Every simple closed curve $c$ on the planar surface $\Sigma$ separates
it, and we write $S_c\subset\{1,\ldots,5\}$ for the set of indices of the holes
$\delta_1,\ldots,\delta_5$ which $c$ separates from the outer boundary
$\delta_0$.  The homology class of $c$ is then $\sum_{i\in S_c}g_i$, where
$g_i\in H_1(\Sigma;\Z)\cong\Z^5$ is the class of a loop around $\delta_i$, and
$S_c$ is nonempty since $c$ is homologically nontrivial.  Hence
\begin{equation}\label{eq:H1all}
H_1(X;\Z)\cong H_1(\Sigma;\Z)/\langle[c_1],\ldots,[c_N]\rangle ,
\end{equation}
the abelian group generated by $g_1,\ldots,g_5$ subject to the relations
$\sum_{i\in S_{c_k}}g_i=0$, $k=1,\ldots,N$.

The sets $S_{c_1},\ldots,S_{c_N}$ are constrained by the following invariants of
$\phi$.  For a nonempty set $I\subset\{1,\ldots,5\}$, capping the holes
$\delta_j$ with $j\notin I$ by disks induces a homomorphism from the mapping
class group of $\Sigma$ to the mapping class group of a disk with $\lvert
I\rvert$ holes, which maps $t_c$ to the Dehn twist along the image of $c$.  This
image bounds a disk if $c$ encloses no hole of $I$, it is parallel to the
boundary component $\delta_i$ if $c$ encloses exactly one hole $\delta_i$ with
$i\in I$, and it is parallel to the outer boundary if $c$ encloses all the holes
of $I$.  For $I=\{i\}$ the capped surface is an annulus, whose mapping class
group is infinite cyclic, generated by the twist along its core.  Hence the
number
\[
m_i=\#\{k : i\in S_{c_k}\}
\]
is determined by $\phi$.  For $I=\{i,j\}$ the capped surface is a pair of pants,
whose mapping class group is free abelian of rank three, generated by the twists
along its three boundary components, and the image of $t_c$ is the twist along
the outer boundary if and only if $c$ encloses both $\delta_i$ and $\delta_j$.
Hence the number
\[
n_{ij}=\#\{k : i,j\in S_{c_k}\}
\]
is determined by $\phi$ as well.  We evaluate these numbers on Baykur's
factorization $\phi=t_dt_ut_vt_wt_z$, for which $S_d=\{1,3\}$, $S_u=\{2,3\}$,
$S_v=\{2,4,5\}$, $S_w=\{1,3,5\}$ and $S_z=\{1,4\}$.  This gives
$(m_1,\ldots,m_5)=(3,2,3,2,2)$, and $n_{12}=n_{34}=0$, $n_{13}=2$, and
$n_{ij}=1$ for the remaining pairs $\{i,j\}$.

\begin{lemma}\label[lemma]{lem:enum}
Let $S_1,\ldots,S_N$ be nonempty subsets of $\{1,\ldots,5\}$ with these numbers
$m_i$ and $n_{ij}$.  Then $N=5$, and up to order the sets $S_1,\ldots,S_5$ are
either
\[
\{1,3\},\ \{1,4\},\ \{2,3\},\ \{1,3,5\},\ \{2,4,5\}\qquad\text{or}\qquad
\{1,3\},\ \{1,3\},\ \{2,4\},\ \{1,4,5\},\ \{2,3,5\}.
\]
\end{lemma}

\begin{proof}
Let $s_k=\lvert S_k\rvert$.  Counting pairs $(i,k)$ with $i\in S_k$, and triples
$(i,j,k)$ with $i<j$ and $i,j\in S_k$, gives $\sum_ks_k=\sum_im_i=12$ and
$\sum_k\binom{s_k}{2}=\sum_{i<j}n_{ij}=9$.  Since $n_{12}=n_{34}=0$, no $S_k$
contains $\{1,2\}$ or $\{3,4\}$.  In particular $s_k\le3$, since a subset of
$\{1,\ldots,5\}$ with at least four elements omits at most one element, and
therefore contains $\{1,2\}$ or $\{3,4\}$.  Let $a_s$ be the number of indices
$k$ with $s_k=s$.  Then $a_1+2a_2+3a_3=12$ and $a_2+3a_3=9$, hence $a_2=9-3a_3$
and $a_1=3a_3-6$, and nonnegativity forces $a_3\in\{2,3\}$.  If $a_3=3$, then
$a_1=3$ and $a_2=0$, and each of the three sets with three elements avoids
$\{1,2\}$ and $\{3,4\}$, so it has the form $\{i,j,5\}$ with $i\in\{1,2\}$ and
$j\in\{3,4\}$.  Then all three contain $5$, which contradicts $m_5=2$.  Hence
$a_3=2$, $a_2=3$ and $a_1=0$, that is, $N=5$, and there are three sets with two
elements and two sets with three elements.

Since $m_4=2$, exactly two of the sets contain $4$.  Neither of them contains
$3$, and they do not both contain $1$ and $2$.  Since $n_{14}=n_{24}=n_{45}=1$,
together they contain $1$ once, $2$ once and $5$ once.  The pair $\{4,5\}$ and
$\{1,2,4\}$ is excluded by $n_{12}=0$, so these two sets are either $\{1,4\}$
and $\{2,4,5\}$, or $\{2,4\}$ and $\{1,4,5\}$.  In the first case the remaining
three sets all contain $3$, since $m_3=3$, and they consist of two sets with two
elements and one set with three elements.  They have to realize the remaining
numbers $m_1=2$, $m_2=1$, $m_5=1$, $n_{15}=1$ and $n_{25}=0$.  Hence the set with
three elements contains $1$ and $5$, so it is $\{1,3,5\}$, and the other two are
$\{1,3\}$ and $\{2,3\}$.  In the second case the remaining numbers are $m_1=2$,
$m_2=1$, $m_5=1$, $n_{15}=0$ and $n_{25}=1$, and they force the sets
$\{2,3,5\}$, $\{1,3\}$ and $\{1,3\}$.  Both lists realize all the numbers $m_i$
and $n_{ij}$.
\end{proof}

\begin{proof}[Proof of \cref{prop:allfillings}]
By \cref{lem:enum}, $X$ has a handle decomposition with one $0$--handle, five
$1$--handles and five $2$--handles, so $\chi(X)=1$.  We compute $H_1(X;\Z)$ from
\eqref{eq:H1all} in the two cases of \cref{lem:enum}.  In the first case the
relations are
\[
g_1+g_3=0,\quad g_1+g_4=0,\quad g_2+g_3=0,\quad g_1+g_3+g_5=0,\quad g_2+g_4+g_5=0 .
\]
The first and the fourth give $g_5=0$, and then $g_3=-g_1$, $g_4=-g_1$ and
$g_2=-g_3=g_1$, after which the last relation holds automatically.  Hence
$H_1(X;\Z)\cong\Z$, generated by $g_1$.  In the second case the relations are
\[
g_1+g_3=0,\quad g_2+g_4=0,\quad g_1+g_4+g_5=0,\quad g_2+g_3+g_5=0 ,
\]
the first one appearing twice.  They give $g_3=-g_1$, $g_4=-g_2$ and
$g_5=g_2-g_1$, and the last relation becomes $2g_2-2g_1=0$.  Hence $H_1(X;\Z)$
is generated by $g_1$ and $g_2$ with the single relation $2(g_2-g_1)=0$, so it
is isomorphic to $\Z\oplus\Z/2$, with $\Z$ generated by $g_1$ and $\Z/2$
generated by $g_2-g_1$.  In both cases $b_1(X)=1$.  Since $X$ has no handles of
index three or four, $H_k(X;\Z)=0$ for $k\ge3$, and $H_2(X;\Z)$ is free abelian
of rank $b_2(X)=\chi(X)-1+b_1(X)=1$.
\end{proof}

By \cref{thm:newfillings}, both cases of \cref{prop:allfillings} occur, for
infinitely many pairwise non-homeomorphic fillings each.

\section{Rational homology ball Stein fillings}\label{sec:vary}

Baykur's Stein fillings $X_n$ of $(Y,\xi)$ come from the positive factorizations
$t_d\,t_\alpha^nBt_\alpha^{-n}$ of its monodromy $t_dB$.  In this section we keep
$B$ and its conjugates $t_\alpha^nBt_\alpha^{-n}$, but replace $d$ by one of the
curves $d_m=t_u^m(c_{12})$ and $d'_m=t_u^m(c_{345})$, $m\ge1$, of the
introduction, where $c_{12}$ and $c_{345}$ are the round curves enclosing the
holes $\delta_1,\delta_2$ and $\delta_3,\delta_4,\delta_5$, and $u$ is the round
curve enclosing $\delta_2,\delta_3$.  The simplest examples $d_1$ and
$d'_1$ are shown in \cref{fig:di}.  Each curve gives a new planar
contact $3$--manifold with Stein fillings $X_{m,n}$ and $X'_{m,n}$, $n\ge0$.  We compute the fundamental groups of these fillings as
in \cref{subsec:pi1new}, which proves \cref{thm:ball}, and we describe the
filling $X_{1,0}$ by explicit Kirby diagrams.

We use the notation of \cref{subsec:pi1new}: $x_1,\ldots,x_5$ are the free generators
of $\pi_1(\Sigma,\ast)$, $Q=F_5/\ncl{u,v,w,z}\cong\pi_1(K)$, and
$x\mapsto\overline{x}$ denotes the quotient map $F_5\to Q$, as in \cref{lem:pi1}.  By \eqref{eq:subst} and
\eqref{eq:Q},
\begin{equation}\label{eq:qmap}
\overline{x_1}=rb,\quad \overline{x_2}=b,\quad \overline{x_3}=b^{-1},\quad \overline{x_4}=r^{-1}b^{-1},\quad
\overline{x_5}=r,
\end{equation}
in $Q=\langle r,b\mid brb^{-1}=r^{-1}\rangle$.  Every element of $Q$ can be
written uniquely as $r^ib^j$ with $i,j\in\Z$, and since $b^jr^k=r^{(-1)^jk}b^j$,
which follows from $brb^{-1}=r^{-1}$, the multiplication in these normal forms is
\begin{equation}\label{eq:moveR}
(r^ib^j)(r^kb^l)=r^{\,i+(-1)^jk}\,b^{\,j+l}.
\end{equation}
In particular $b^2$ is central in $Q$.  We use \eqref{eq:moveR} without further
comment in the computations below.

\subsection{The construction}\label{subsec:general}

Let $c$ be Baykur's curve $d$ or one of the curves $d_m$, $d'_m$.  Since
$t_\alpha$ commutes with $B$ \cite[(8)--(9)]{BaykurPlanar}, for every $n\ge0$,
\[
t_cB=t_c\,t_\alpha^nBt_\alpha^{-n}
=t_c\,t_{t_\alpha^n(u)}\,t_{t_\alpha^n(v)}\,t_{t_\alpha^n(w)}\,t_{t_\alpha^n(z)}
\]
is a positive factorization.  We denote by $(Y_c,\xi_c)$ the contact
$3$--manifold supported by the open book $(\Sigma,t_cB)$, and by
\[
X_n(c)=X(\Sigma;\,c,\,t_\alpha^n(u),\,t_\alpha^n(v),\,t_\alpha^n(w),\,t_\alpha^n(z))
\]
the corresponding Stein fillings.  Thus $(Y_d,\xi_d)=(Y,\xi)$ and $X_n(d)=X_n$.
As in \cref{lem:pi1},
\[
\pi_1(X_n(c))\cong F_5/\ncl{c,\,t_\alpha^n(u),\,t_\alpha^n(v),\,t_\alpha^n(w),\,t_\alpha^n(z)}
\cong Q/\ncl{\overline{t_\alpha^{-n}(c)}},
\]
where the second isomorphism is induced by the automorphism $t_\alpha^{-n}$ of
$F_5$.

\begin{lemma}\label[lemma]{lem:theta}
For every $n\ge0$, the homomorphism $F_5\to Q$, $x\mapsto\overline{t_\alpha^{-n}(x)}$, is given by
\[
x_1\mapsto r^{2n+1}b,\quad x_2\mapsto r^{2n}b,\quad x_3\mapsto b^{-1},\quad
x_4\mapsto r^{-1}b^{-1},\quad x_5\mapsto r .
\]
\end{lemma}

\begin{proof}
The Dehn twist $t_\alpha$ along the round curve $\alpha=x_2x_1$ conjugates $x_1$
and $x_2$ by $\alpha$, fixes $x_3,x_4,x_5$, and fixes $\alpha$.  Hence
$t_\alpha^{-n}(x_i)=\alpha^{-n}x_i\alpha^n$ for $i=1,2$ and $t_\alpha^{-n}(x_i)=x_i$
for $i=3,4,5$, which gives the images of $x_3,x_4,x_5$.  In $Q$ we have
$\overline{\alpha}=b(rb)=r^{-1}b^2$, and since $b^2$ is central,
conjugation by $\overline{\alpha}^{\,-n}$ is conjugation by $r^n$.  For $g=r^ib$
this gives $r^n(r^ib)r^{-n}=r^{i+2n}b$, so $x_1\mapsto r^{2n+1}b$ and
$x_2\mapsto r^{2n}b$.
\end{proof}

For example, for Baykur's curve $d=x_3x_1$, \cref{lem:theta} gives
\[
\overline{t_\alpha^{-n}(d)}=b^{-1}(r^{2n+1}b) ,
\]
so $\overline{t_\alpha^{-n}(d)}=r^{-(2n+1)}$.  By \cref{lem:pi1} this recovers
$\pi_1(X_n)\cong Q/\ncl{r^{2n+1}}\cong\Z/(2n+1)\rtimes_{-1}\Z$.

\subsection{Proof of \texorpdfstring{\cref{thm:ball}}{Theorem 1.1}}\label{subsec:ball}

We first outline the strategy.  By \cref{lem:pi1}, the fundamental group of the
filling $X_n(c)$ is the quotient of $Q=\pi_1(K)=\langle r,b\mid brb^{-1}=r^{-1}\rangle$
by the normal closure of the single element $\overline{t_\alpha^{-n}(c)}$.  So the
whole computation reduces to finding this element for $c=d_m$ and $c=d'_m$.  In
\cref{lem:dm-images} we write $d_m$ and $d'_m$ as words in the generators
$x_1,\ldots,x_5$, using that $t_u^m$ conjugates $x_2$ and $x_3$ by $u^m$, and we
push these words into $Q$ using \cref{lem:theta}.  The result is
$r^{-(4mn+1)}b^2$ for $d_m$ and $r^{4mn}b^{-2}$ for $d'_m$.  Hence both
fundamental groups are of the form
$\langle r,b\mid brb^{-1}=r^{-1},\ b^2=r^\ell\rangle$, with $\ell=4mn+1$ odd in the
first case and $\ell=4mn$ even in the second case.  In
\cref{lem:group-even} we show that such a group has order $4\ell$ for every
$\ell>0$, which gives the order $16mn$ in the second case for $n\ge1$, and in
\cref{lem:group-odd} we use this to identify the group in the first case with
$\Z/(4mn+1)\rtimes_{-1}\Z/4$.  In both
families the orders of the finite groups strictly increase with $n$, for $n\ge0$
in the first family and for $n\ge1$ in the second, while for $n=0$ the group in
the second family is infinite.  Hence the fillings are pairwise
non-homeomorphic.  Finally, the first homology of the fillings is finite, and a
count of handles shows that they are rational homology balls.  The first
homology of their boundaries is computed directly from the planar open book in
\cref{lem:H1-boundaries}.

\begin{figure}[ht]
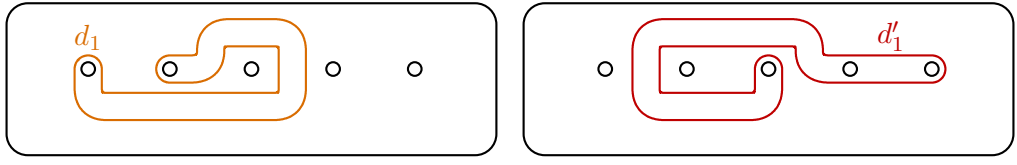

\centering
\makebox[\textwidth][c]{%
}
\caption{The curves $d_1=t_u(c_{12})=\sigma_2^{-1}(d)$ (left) and
$d'_1=t_u(c_{345})=\sigma_2^{-2}(c_{345})$ (right) on the page $\Sigma_0^6$, with
the holes $\delta_1,\ldots,\delta_5$ from left to right, as in
\cref{fig:baykur-curves}.}
\label{fig:di}
\end{figure}

\begin{lemma}\label[lemma]{lem:dm-images}
Let $m\ge1$ and $n\ge0$.  The curves $d_m$ and $d'_m$ are represented by the words
\[
d_m=u^mx_2u^{-m}x_1,\qquad d'_m=x_5x_4\,u^mx_3u^{-m},
\]
they enclose the holes $\delta_1,\delta_2$ and $\delta_3,\delta_4,\delta_5$,
respectively, and their images in $Q$ under $x\mapsto\overline{t_\alpha^{-n}(x)}$ are
\[
\overline{t_\alpha^{-n}(d_m)}=r^{-(4mn+1)}b^2,\qquad
\overline{t_\alpha^{-n}(d'_m)}=r^{4mn}b^{-2}.
\]
\end{lemma}

\begin{proof}
The Dehn twist $t_u$ along the round curve $u=x_3x_2$ conjugates $x_2$ and $x_3$
by $u$, fixes $x_1,x_4,x_5$, and fixes $u$, so $t_u^m(x_i)=u^mx_iu^{-m}$ for
$i=2,3$.  Hence, since $c_{12}=x_2x_1$ and $c_{345}=x_5x_4x_3$,
\[
d_m=u^mx_2u^{-m}x_1,\qquad d'_m=x_5x_4\,u^mx_3u^{-m}.
\]
For $m=1$ the first word is $x_3x_2\,x_2\,x_2^{-1}x_3^{-1}\,x_1=x_3x_2x_3^{-1}x_1$,
which is also $\sigma_2^{-1}(x_3x_1)$, since $\sigma_2^{-1}$ fixes $x_1$ and maps
$x_3$ to $x_3x_2x_3^{-1}$, so $d_1=\sigma_2^{-1}(d)$.  For $m=1$ the second
word is $x_5x_4\,x_3x_2\,x_3\,x_2^{-1}x_3^{-1}$.  The exponent
sums of $x_1,\ldots,x_5$ in the word of $d_m$ are $1,1,0,0,0$, and in the word
of $d'_m$ they are $0,0,1,1,1$, so $d_m$ encloses the holes
$\delta_1,\delta_2$ and $d'_m$ encloses $\delta_3,\delta_4,\delta_5$, as $c_{12}$
and $c_{345}$ do.

We compute the images in $Q$ with \cref{lem:theta} and \eqref{eq:moveR}.
Since $x_3\mapsto b^{-1}$ and $x_2\mapsto r^{2n}b$, the curve $u$ maps to
$b^{-1}r^{2n}b=r^{-2n}$, and $u^m$ maps to $r^{-2mn}$.  Hence
\[
\overline{t_\alpha^{-n}(d_m)}=r^{-2mn}\,(r^{2n}b)\,r^{2mn}\,(r^{2n+1}b)
=r^{-2mn+2n-2mn-(2n+1)}\,b^2=r^{-(4mn+1)}b^2 ,
\]
and, since $x_5x_4\mapsto r(r^{-1}b^{-1})=b^{-1}$,
\[
\overline{t_\alpha^{-n}(d'_m)}=b^{-1}\,r^{-2mn}\,b^{-1}\,r^{2mn}
=r^{2mn+2mn}\,b^{-2}=r^{4mn}b^{-2}.
\]
\end{proof}

\begin{lemma}\label[lemma]{lem:group-even}
Let $\ell\ge0$ be an integer, and let
\[
G=\langle r,b\mid brb^{-1}=r^{-1},\ b^2=r^\ell \rangle .
\]
If $\ell>0$, then $G$ is the dicyclic group $\mathrm{Dic}_\ell=\langle a,x\mid a^{2\ell}=1,\ x^2=a^\ell,\
xax^{-1}=a^{-1}\rangle$.  In particular, $G$ is finite of order $4\ell$, and its
elements are the pairwise distinct elements $r^ib^j$ with $0\le i<2\ell$ and
$j\in\{0,1\}$.  For $\ell =0$ it is the infinite dihedral group
$Q/\ncl{b^2}\cong\Z\rtimes_{-1}\Z/2$.
\end{lemma}

\begin{proof}
Conjugating $b^2=r^\ell$ by $b$ gives $b^2=r^{-\ell}$, hence $r^\ell=r^{-\ell}$ and
$r^{2\ell}=1$ in $G$.  Adding the relation $r^{2\ell}=1$ to the presentation of
$G$ therefore does not change the group, and the resulting presentation is the
standard presentation of $\mathrm{Dic}_\ell$, with $a=r$ and $x=b$.  It is well known
that $\mathrm{Dic}_\ell$ has order $4\ell$.  Since every element of $G$ can be
written as $r^ib^j$ with $0\le i<2\ell$ and $j\in\{0,1\}$, these $4\ell$ elements
are pairwise distinct.  For $\ell =0$ the presentation is
$\langle r,b\mid brb^{-1}=r^{-1},\ b^2=1\rangle=Q/\ncl{b^2}$, the infinite
dihedral group $\Z\rtimes_{-1}\Z/2$, where $\Z=\langle r\rangle$ and
$\Z/2=\langle b\rangle$.
\end{proof}

\begin{lemma}\label[lemma]{lem:group-odd}
Let $k$ be a positive odd integer.  Then the group
\[
G=\langle r,b\mid brb^{-1}=r^{-1},\ b^2=r^k\rangle
\]
is isomorphic to $\Z/k\rtimes_{-1}\Z/4$, a group of order $4k$.
\end{lemma}

\begin{proof}
By \cref{lem:group-even}, $G$ has order $4k$, and its elements $r^ib^j$,
$0\le i<2k$, $j\in\{0,1\}$, are pairwise distinct.  In particular $r$ has order
$2k$, so $s=r^2$ has order $k$, and $b$ has order $4$, since $b^2=r^k\neq1$ and
$b^4=r^{2k}=1$.  The subgroup $\langle s\rangle$ is normal, since
$bsb^{-1}=s^{-1}$.  Moreover $\langle s\rangle\cap\langle b\rangle=1$: the
elements of $\langle b\rangle$ are $1,b,r^k,r^kb$, and neither $r^k$, which is an
odd power of $r$ because $k$ is odd, nor $b$ or $r^kb$ lies in
$\langle s\rangle=\{r^{2i}\}$.  Since $\lvert\langle s\rangle\rvert\cdot
\lvert\langle b\rangle\rvert=4k=\lvert G\rvert$, it follows that $G$ is the
semidirect product of the normal subgroup $\langle s\rangle\cong\Z/k$ and the
subgroup $\langle b\rangle\cong\Z/4$, in which $b$ acts on $s$ by inversion, that
is, $G\cong\Z/k\rtimes_{-1}\Z/4$.
\end{proof}

\begin{lemma}\label[lemma]{lem:H1-boundaries}
For every $m\ge1$,
\[
H_1(Y_m;\Z)\cong\Z/4\oplus\Z/4,\qquad H_1(Y'_m;\Z)\cong\Z/2\oplus\Z/2\oplus\Z/4 .
\]
In particular, $Y_m$ and $Y'_m$ are rational homology spheres, and no $Y_m$ is
homeomorphic to any $Y'_{m'}$.
\end{lemma}

\begin{proof}
We write $c$ for $d_m$ or $d'_m$, and we compute $H_1(Y_c;\Z)$ directly from the
planar open book $(\Sigma_0^6,t_cB)$, as in Etnyre and the author \cite{EtnyreOzbagci}.
Let $g_k\in H_1(\Sigma_0^6;\Z)$ be the class of a loop around the hole $\delta_k$,
so that $H_1(\Sigma_0^6;\Z)$ is freely generated by $g_1,\ldots,g_5$, and let $b_k$
be an arc in the page from $\delta_0$ to $\delta_k$.  The $3$--manifold $Y_c$ is
obtained from the mapping torus of the monodromy $\phi=t_cB$ by attaching a solid
torus along each boundary component, and $H_1(Y_c;\Z)$ is the quotient of
$H_1(\Sigma_0^6;\Z)$ by the classes $[\phi(b_k)-b_k]$, $k=1,\ldots,5$.  The
monodromy is a product of right-handed Dehn twists along the curves $c,u,v,w,z$.
Each of these curves bounds a disk in the plane containing some of the holes,
so its class is $\sum_k A_{kl}\,g_k$, where $A$ is the matrix whose $(k,l)$ entry
is $1$ if the $l$-th curve of $(c,u,v,w,z)$ encloses $\delta_k$ and $0$ otherwise.
In particular these Dehn twists act trivially on $H_1(\Sigma_0^6;\Z)$, and by the
Picard--Lefschetz formula, applied to one twist at a time,
\[
[\phi(b_k)-b_k]=\sum_l \langle b_k,\gamma_l\rangle\,[\gamma_l]
=\pm\sum_l A_{kl}\sum_j A_{jl}\,g_j ,
\]
where $\gamma_l$ runs through the curves $c,u,v,w,z$, oriented as the boundaries of
the disks they bound.  Here we used that the algebraic intersection number
$\langle b_k,\gamma_l\rangle$ is $\pm1$, with a sign that is the same for all $l$, if
$\gamma_l$ encloses $\delta_k$, since then $\gamma_l$ separates $\delta_k$ from
$\delta_0$, and that it is $0$ otherwise.  Hence $H_1(Y_c;\Z)$ is the abelian
group generated by $g_1,\ldots,g_5$, subject to the relations given by the rows of
the symmetric matrix $AA^T$.  Since $A$ only depends on the holes enclosed by $c$,
the group $H_1(Y_c;\Z)$ is the same for all $d_m$, and the same for all $d'_m$.

For $c=d_m$, which encloses $\delta_1,\delta_2$, the rows of $AA^T$ give the
relations
\begin{align*}
&3g_1+g_2+g_3+g_4+g_5=0, && g_1+3g_2+g_3+g_4+g_5=0, && g_1+g_2+2g_3+g_5=0,\\
&g_1+g_2+2g_4+g_5=0, && g_1+g_2+g_3+g_4+2g_5=0 .
\end{align*}
The difference of the first and the fifth relation gives $g_5=2g_1$, and then the
third and the first give $g_2=-3g_1-2g_3$ and $g_4=g_3-2g_1$.  The remaining
relations become $4g_1=0$ and $8g_1+4g_3=0$, so $H_1(Y_m;\Z)$ is generated by
$g_1$ and $g_3$ with the relations $4g_1=0$ and $4g_3=0$, that is,
$H_1(Y_m;\Z)\cong\Z/4\oplus\Z/4$.

For $c=d'_m$, which encloses $\delta_3,\delta_4,\delta_5$, the relations are
\begin{align*}
&2g_1+g_3+g_4+g_5=0, && 2g_2+g_3+g_4+g_5=0, && g_1+g_2+3g_3+g_4+2g_5=0,\\
&g_1+g_2+g_3+3g_4+2g_5=0, && g_1+g_2+2g_3+2g_4+3g_5=0 .
\end{align*}
The first and the third give $g_5=-2g_1-g_3-g_4$ and $g_4=-3g_1+g_2+g_3$, and the
remaining relations become $2h=0$, $4g_1=2h$ and $2h'=0$, where $h=g_2-g_1$ and
$h'=g_3+g_1$.  Hence $H_1(Y'_m;\Z)$ is generated by $g_1,h,h'$ with the relations
$4g_1=0$, $2h=0$ and $2h'=0$, that is, $H_1(Y'_m;\Z)\cong\Z/2\oplus\Z/2\oplus\Z/4$.
\end{proof}

\begin{proof}[Proof of \cref{thm:ball}]
Part (1) is \cref{lem:H1-boundaries}.

For (2), by \cref{lem:pi1,lem:dm-images}, $\pi_1(X_{m,n})$ is the group
$\langle r,b\mid brb^{-1}=r^{-1},\ b^2=r^k\rangle$ with $k=4mn+1$, which is odd and
positive.  By \cref{lem:group-odd} it is isomorphic to $\Z/(4mn+1)\rtimes_{-1}\Z/4$,
of order $4(4mn+1)$, including the case $n=0$, where it is $\Z/4$.  Abelianizing
the relations gives $2r=0$ and $2b=kr=r$, since $k$ is odd.  Hence
$H_1(X_{m,n};\Z)$ is generated by $b$, with the relation $4b=0$, that is,
$H_1(X_{m,n};\Z)\cong\Z/4$.

For (3), by \cref{lem:pi1,lem:dm-images}, $\pi_1(X'_{m,n})$ is the group
$\langle r,b\mid brb^{-1}=r^{-1},\ b^2=r^\ell\rangle$ with $\ell=4mn$.  By
\cref{lem:group-even} it is the dicyclic group $\mathrm{Dic}_{4mn}$, of order
$16mn$, for $n\ge1$, and the infinite dihedral group $\Z\rtimes_{-1}\Z/2$ for
$n=0$.  Abelianizing the relations gives $2r=0$ and $2b=\ell r=0$, since $\ell$ is
even.  Hence $H_1(X'_{m,n};\Z)\cong\Z/2\oplus\Z/2$, also for $n=0$.

In both cases the first homology is finite, so $b_1=0$.  Each of these fillings
has a handle decomposition with one $0$--handle, five $1$--handles and five
$2$--handles, so $\chi=1$, hence $b_2=0$, and it is a rational homology ball.

For (4), the numbers $4(4mn+1)$, $n\ge0$, are pairwise distinct, so the fillings
$X_{m,n}$ have pairwise non-isomorphic fundamental groups.  Similarly, the groups
$\pi_1(X'_{m,n})$, $n\ge1$, have pairwise distinct orders $16mn$, and
$\pi_1(X'_{m,0})$ is infinite.  Hence in both families the fillings are pairwise
non-homeomorphic.
\end{proof}

In \cref{subsec:computer-families} we show by computer that, within each family,
some of these $3$--manifolds are pairwise non-homeomorphic.

\begin{figure}[ht]
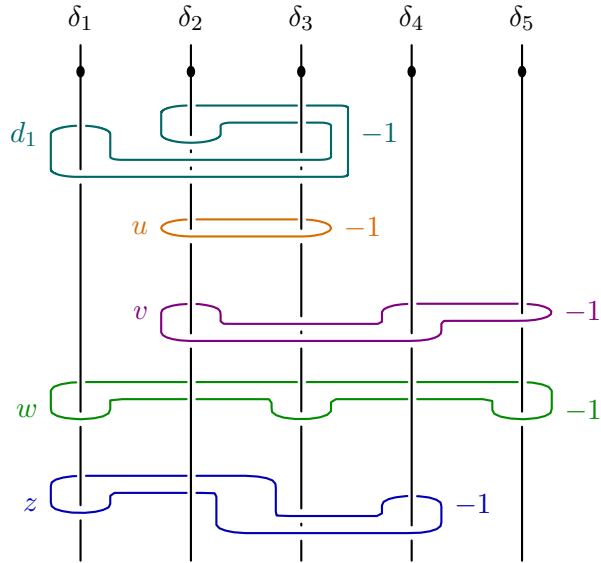

\centering

\caption{The Kirby diagram of the Lefschetz fibration
$X_{1,0}=X(\Sigma_0^6;d_1,u,v,w,z)$: as in \cref{fig:baykur-kirby}, with the
vanishing cycle $d$ replaced by $d_1$.}
\label{fig:X10-lefschetz}
\end{figure}

\begin{figure}[ht]
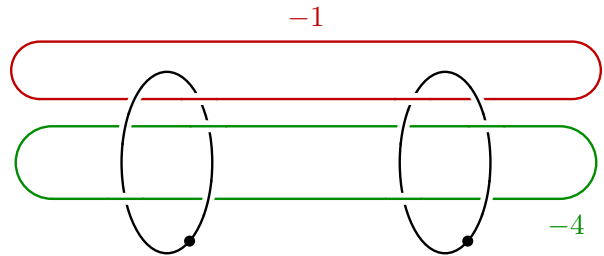

\centering
%
\caption{A Kirby diagram of the Stein filling $X_{1,0}$ of $(Y_1,\xi_1)$:
the diagram of $DT^*K$ of \cref{fig:klein-kirby} together with a $2$--handle of
framing $-1$ which runs once over each $1$--handle, in opposite directions.
Compare \cref{fig:W0-kirby}, where the additional $2$--handle is a meridian of the
$2$--handle of $DT^*K$.}
\label{fig:X0d1}
\end{figure}

\begin{proposition}\label[proposition]{rem:X0d1}
The rational homology ball $X_{1,0}=X(\Sigma;d_1,u,v,w,z)$ has the following
Kirby diagrams.
\begin{enumerate}
\item[(a)] $X_{1,0}$ is diffeomorphic to $DT^*K$ with a single $2$--handle
attached with framing $-1$, as in \cref{fig:X0d1}.  The attaching circle of this
$2$--handle runs algebraically once over each $1$--handle of $DT^*K$, in
opposite directions.
\item[(b)] $X_{1,0}$ has the Kirby diagram of \cref{fig:X0d1-trefoil}, with a
single $1$--handle and a single $2$--handle attached along a left-handed trefoil
with framing $-8$, which runs four times over the $1$--handle in the same
direction.  In particular $\pi_1(X_{1,0})\cong\langle a\mid a^4\rangle\cong\Z/4$,
and $Y_1$ is $(0,-8)$--surgery on the two-component link of
\cref{fig:X0d1-trefoil}, formed by an unknot and a left-handed trefoil with
linking number $4$.
\end{enumerate}
\end{proposition}

\begin{proof}
For (a) we start from the Kirby diagram of the Lefschetz fibration $X_{1,0}$ in
\cref{fig:X10-lefschetz}.  The handle moves are those of the proof of
\cref{thm:klein-main}, together with two slides of $h_{d_1}$ over $h_u$ before
the last cancellation.  They lead to \cref{fig:X0d1}, and the sequence of handle
moves, drawn by hand, is included as an ancillary file.  As a consistency check,
the new $2$--handle kills the free part of $H_1(DT^*K;\Z)\cong\Z\oplus\Z/2$, since
it runs algebraically once over each $1$--handle in opposite directions, and it
creates no second homology, since $H_2(DT^*K;\Z)\cong H_2(K;\Z)=0$.  This is in
accordance with $X_{1,0}$ being a rational homology ball.

For (b) we start from the diagram of \cref{fig:X0d1}.  The $(-4)$--framed
$2$--handle is slid twice over the $(-1)$--framed one, which changes its framing
to $-5$ and then to $-5+(-1)+2\cdot(-1)=-8$, and the $(-1)$--framed $2$--handle
is then cancelled against one of the $1$--handles.  After an isotopy this gives
\cref{fig:X0d1-trefoil}, and the sequence of handle moves, drawn by hand, is
included as an ancillary file.  The presentation of $\pi_1(X_{1,0})$ can be read
off from the diagram, with one generator $a$ for the $1$--handle and the relation
$a^4=1$ given by the $2$--handle.  Replacing the dotted circle by a $0$--framed
unknot gives the surgery description of $Y_1$.
\end{proof}

\begin{figure}[ht]
\centering
\begin{tikzpicture}[line cap=round,line join=round]
\draw[black,line width=1pt] plot[smooth] coordinates{(-1.219,2.662) (-1.209,2.651) (-1.200,2.639) (-1.190,2.628) (-1.180,2.616) (-1.171,2.604) (-1.161,2.592) (-1.151,2.579) (-1.142,2.567) (-1.132,2.554) (-1.123,2.540) (-1.113,2.527) (-1.104,2.513) (-1.095,2.499) (-1.087,2.484) (-1.078,2.470) (-1.070,2.455) (-1.062,2.439) (-1.054,2.424) (-1.046,2.408) (-1.039,2.391) (-1.032,2.375) (-1.025,2.358) (-1.019,2.340) (-1.013,2.323) (-1.007,2.305) (-1.001,2.287) (-0.996,2.268) (-0.991,2.250) (-0.987,2.231) (-0.982,2.211) (-0.979,2.192)};
\draw[black,line width=1pt] plot[smooth] coordinates{(-0.962,2.030) (-0.962,2.009) (-0.963,1.988) (-0.964,1.967) (-0.965,1.947) (-0.967,1.926) (-0.970,1.905) (-0.973,1.885) (-0.977,1.864) (-0.981,1.844) (-0.987,1.824) (-0.993,1.804) (-0.999,1.784) (-1.007,1.764) (-1.015,1.745) (-1.024,1.725) (-1.034,1.707) (-1.045,1.688) (-1.056,1.670) (-1.069,1.652) (-1.082,1.634) (-1.097,1.617) (-1.112,1.600) (-1.128,1.583) (-1.146,1.567) (-1.164,1.551) (-1.183,1.536) (-1.204,1.521) (-1.225,1.506) (-1.247,1.492) (-1.271,1.478) (-1.295,1.464) (-1.320,1.451) (-1.346,1.438) (-1.373,1.425) (-1.401,1.413) (-1.429,1.401) (-1.459,1.389) (-1.489,1.377) (-1.519,1.366) (-1.551,1.355) (-1.582,1.343) (-1.615,1.332) (-1.647,1.321) (-1.680,1.311) (-1.714,1.300) (-1.747,1.289) (-1.781,1.278) (-1.814,1.266) (-1.848,1.255) (-1.882,1.244) (-1.915,1.232) (-1.949,1.220) (-1.982,1.208) (-2.015,1.196) (-2.048,1.183) (-2.080,1.170) (-2.112,1.157) (-2.144,1.144) (-2.175,1.130) (-2.206,1.115) (-2.236,1.101) (-2.266,1.086) (-2.296,1.070) (-2.325,1.055) (-2.353,1.039) (-2.381,1.022) (-2.409,1.005) (-2.435,0.988) (-2.462,0.971) (-2.487,0.953) (-2.513,0.935) (-2.537,0.916) (-2.561,0.898) (-2.584,0.879) (-2.607,0.860) (-2.629,0.840) (-2.650,0.821) (-2.670,0.801) (-2.690,0.781)};
\draw[black,line width=1pt] plot[smooth] coordinates{(-2.791,0.660) (-2.805,0.640) (-2.819,0.619) (-2.831,0.599) (-2.842,0.579) (-2.853,0.559) (-2.862,0.539) (-2.871,0.519) (-2.879,0.500) (-2.886,0.480) (-2.892,0.461) (-2.898,0.442) (-2.902,0.423) (-2.906,0.405) (-2.909,0.386) (-2.911,0.368) (-2.912,0.350) (-2.913,0.332) (-2.913,0.315) (-2.912,0.298) (-2.911,0.281) (-2.909,0.264) (-2.906,0.248) (-2.904,0.231) (-2.900,0.215) (-2.896,0.200) (-2.892,0.184) (-2.887,0.169) (-2.882,0.155) (-2.877,0.140) (-2.872,0.126) (-2.866,0.112) (-2.860,0.098) (-2.854,0.084) (-2.848,0.071) (-2.842,0.058) (-2.835,0.045) (-2.829,0.033) (-2.823,0.021) (-2.816,0.009) (-2.810,-0.003) (-2.803,-0.015) (-2.797,-0.026) (-2.791,-0.037) (-2.784,-0.047) (-2.778,-0.058) (-2.772,-0.068) (-2.765,-0.078) (-2.759,-0.087) (-2.752,-0.096) (-2.746,-0.106) (-2.740,-0.114) (-2.733,-0.123) (-2.726,-0.131) (-2.720,-0.139) (-2.713,-0.147) (-2.706,-0.154) (-2.698,-0.161) (-2.691,-0.168) (-2.684,-0.175) (-2.676,-0.181) (-2.668,-0.187) (-2.660,-0.193) (-2.652,-0.198) (-2.643,-0.204) (-2.635,-0.209) (-2.626,-0.213) (-2.616,-0.217) (-2.607,-0.221) (-2.597,-0.225) (-2.587,-0.228) (-2.577,-0.231) (-2.566,-0.234) (-2.555,-0.236) (-2.544,-0.238) (-2.533,-0.240) (-2.521,-0.241) (-2.509,-0.242) (-2.497,-0.243) (-2.484,-0.243) (-2.471,-0.243) (-2.458,-0.242) (-2.444,-0.241) (-2.430,-0.240) (-2.416,-0.238) (-2.401,-0.236) (-2.386,-0.233) (-2.371,-0.230) (-2.355,-0.227) (-2.339,-0.223) (-2.323,-0.218) (-2.306,-0.213) (-2.289,-0.208) (-2.271,-0.202) (-2.253,-0.196) (-2.234,-0.189) (-2.215,-0.182) (-2.195,-0.175) (-2.176,-0.166) (-2.155,-0.158) (-2.134,-0.149) (-2.113,-0.140) (-2.091,-0.130) (-2.068,-0.119) (-2.045,-0.109) (-2.021,-0.097) (-1.997,-0.086) (-1.973,-0.074) (-1.947,-0.062) (-1.921,-0.049) (-1.895,-0.036) (-1.868,-0.023) (-1.840,-0.009) (-1.812,0.005) (-1.783,0.020) (-1.753,0.034) (-1.723,0.049) (-1.692,0.064) (-1.661,0.079) (-1.629,0.095) (-1.597,0.111) (-1.564,0.127) (-1.531,0.143) (-1.498,0.159) (-1.464,0.175) (-1.429,0.192) (-1.395,0.209) (-1.360,0.225) (-1.325,0.242) (-1.289,0.259) (-1.254,0.276) (-1.218,0.293) (-1.183,0.310) (-1.147,0.327) (-1.112,0.344) (-1.076,0.362) (-1.041,0.379) (-1.006,0.396) (-0.971,0.413) (-0.937,0.431) (-0.903,0.448) (-0.869,0.465) (-0.836,0.482) (-0.803,0.500) (-0.770,0.517) (-0.738,0.534) (-0.707,0.552) (-0.676,0.569) (-0.645,0.586) (-0.616,0.603) (-0.587,0.621)};
\draw[black,line width=1pt] plot[smooth] coordinates{(-0.451,0.707) (-0.426,0.724) (-0.402,0.741) (-0.378,0.758) (-0.355,0.776) (-0.333,0.793) (-0.312,0.810) (-0.291,0.827) (-0.271,0.844) (-0.251,0.861) (-0.232,0.878) (-0.214,0.895) (-0.197,0.912) (-0.180,0.929) (-0.164,0.946) (-0.148,0.963) (-0.133,0.980) (-0.119,0.996) (-0.105,1.013) (-0.092,1.030) (-0.079,1.047) (-0.067,1.064) (-0.056,1.081) (-0.045,1.098) (-0.034,1.115) (-0.024,1.132) (-0.014,1.149) (-0.005,1.166) (0.004,1.183) (0.012,1.201) (0.020,1.218) (0.027,1.236) (0.034,1.254) (0.041,1.272) (0.047,1.290) (0.053,1.309) (0.058,1.328) (0.064,1.347) (0.068,1.366) (0.073,1.386) (0.077,1.405) (0.081,1.426) (0.084,1.446) (0.088,1.467) (0.090,1.487) (0.093,1.508) (0.095,1.530) (0.097,1.551) (0.099,1.573) (0.100,1.595) (0.102,1.617) (0.102,1.639) (0.103,1.661) (0.103,1.684) (0.103,1.706) (0.103,1.728) (0.102,1.751) (0.101,1.773) (0.100,1.796) (0.099,1.818) (0.097,1.840) (0.095,1.863) (0.093,1.885) (0.090,1.907) (0.087,1.929) (0.084,1.950) (0.080,1.972) (0.077,1.994) (0.073,2.015) (0.068,2.036) (0.064,2.057) (0.059,2.078) (0.054,2.099) (0.048,2.120) (0.043,2.140) (0.037,2.161) (0.030,2.181) (0.024,2.201) (0.017,2.221) (0.010,2.240) (0.003,2.260) (-0.005,2.279) (-0.013,2.299) (-0.021,2.318) (-0.029,2.337) (-0.038,2.355) (-0.047,2.374) (-0.056,2.392) (-0.065,2.411) (-0.075,2.429) (-0.085,2.447) (-0.095,2.465) (-0.105,2.482) (-0.116,2.500) (-0.127,2.517) (-0.138,2.535) (-0.149,2.552) (-0.161,2.569) (-0.173,2.586) (-0.185,2.602) (-0.197,2.619) (-0.210,2.635) (-0.223,2.652) (-0.236,2.668) (-0.249,2.684) (-0.263,2.700) (-0.277,2.715) (-0.291,2.731) (-0.305,2.746) (-0.320,2.761) (-0.335,2.776) (-0.350,2.791) (-0.366,2.805) (-0.381,2.819) (-0.397,2.833) (-0.414,2.847) (-0.430,2.861) (-0.447,2.874) (-0.464,2.887) (-0.481,2.899) (-0.498,2.912) (-0.515,2.924) (-0.533,2.936) (-0.551,2.947) (-0.569,2.958) (-0.587,2.968) (-0.605,2.979) (-0.623,2.988) (-0.641,2.998) (-0.660,3.007) (-0.678,3.015) (-0.697,3.023) (-0.715,3.031) (-0.734,3.038) (-0.752,3.044) (-0.771,3.050) (-0.789,3.056) (-0.808,3.061) (-0.826,3.065) (-0.844,3.069) (-0.862,3.073) (-0.880,3.075) (-0.898,3.078) (-0.916,3.080) (-0.933,3.081) (-0.950,3.081) (-0.968,3.082) (-0.984,3.081) (-1.001,3.080) (-1.018,3.079) (-1.034,3.077) (-1.050,3.075) (-1.065,3.072) (-1.081,3.069) (-1.096,3.065) (-1.111,3.061) (-1.125,3.057) (-1.139,3.052) (-1.153,3.046) (-1.166,3.041) (-1.179,3.035) (-1.192,3.028) (-1.204,3.022) (-1.215,3.015) (-1.226,3.008) (-1.237,3.000) (-1.247,2.993) (-1.257,2.985) (-1.266,2.977) (-1.274,2.969) (-1.282,2.960) (-1.289,2.952) (-1.296,2.943) (-1.302,2.934) (-1.307,2.926) (-1.312,2.917) (-1.316,2.908) (-1.320,2.898) (-1.322,2.889) (-1.324,2.880) (-1.326,2.871) (-1.326,2.862) (-1.326,2.852) (-1.326,2.843) (-1.324,2.833) (-1.322,2.824) (-1.320,2.814) (-1.316,2.805) (-1.313,2.795) (-1.308,2.785)};
\draw[green!50!black,line width=1pt] plot[smooth] coordinates{(-2.465,3.025) (-2.490,3.015) (-2.514,3.004) (-2.537,2.993) (-2.560,2.982) (-2.583,2.971) (-2.606,2.959) (-2.628,2.947) (-2.649,2.935) (-2.671,2.923) (-2.691,2.910) (-2.712,2.897) (-2.732,2.884) (-2.751,2.871) (-2.770,2.858) (-2.789,2.844) (-2.807,2.830) (-2.825,2.817) (-2.843,2.803) (-2.860,2.788) (-2.876,2.774) (-2.893,2.760) (-2.908,2.745) (-2.924,2.730) (-2.939,2.715) (-2.954,2.700) (-2.968,2.685) (-2.982,2.670) (-2.996,2.655) (-3.009,2.639) (-3.022,2.623) (-3.034,2.608) (-3.047,2.592) (-3.059,2.576) (-3.070,2.559) (-3.081,2.543) (-3.092,2.526) (-3.103,2.510) (-3.113,2.493) (-3.123,2.476) (-3.133,2.458) (-3.143,2.441) (-3.152,2.423) (-3.160,2.406) (-3.169,2.388) (-3.177,2.369) (-3.185,2.351) (-3.192,2.332) (-3.200,2.314) (-3.207,2.295) (-3.213,2.275) (-3.219,2.256) (-3.225,2.236) (-3.231,2.216) (-3.237,2.196) (-3.242,2.176) (-3.246,2.156) (-3.251,2.135) (-3.255,2.114) (-3.259,2.093) (-3.262,2.071) (-3.266,2.050) (-3.269,2.028) (-3.271,2.006) (-3.274,1.984) (-3.276,1.961) (-3.278,1.939) (-3.279,1.916) (-3.280,1.893) (-3.281,1.870) (-3.282,1.847) (-3.282,1.824) (-3.282,1.800) (-3.282,1.777) (-3.282,1.753) (-3.281,1.730) (-3.280,1.706) (-3.278,1.683) (-3.276,1.659) (-3.274,1.636) (-3.272,1.612) (-3.269,1.589) (-3.267,1.566) (-3.263,1.543) (-3.260,1.520) (-3.256,1.498) (-3.251,1.475) (-3.247,1.453) (-3.242,1.431) (-3.237,1.410) (-3.231,1.388) (-3.225,1.367) (-3.218,1.347) (-3.211,1.326) (-3.204,1.306) (-3.197,1.285) (-3.188,1.265) (-3.180,1.246) (-3.171,1.226) (-3.161,1.207) (-3.151,1.187) (-3.141,1.168) (-3.130,1.149) (-3.119,1.130) (-3.107,1.111) (-3.094,1.093) (-3.081,1.074) (-3.068,1.055) (-3.054,1.036) (-3.039,1.018) (-3.024,0.999) (-3.008,0.980) (-2.992,0.962) (-2.975,0.943) (-2.957,0.924) (-2.939,0.905) (-2.920,0.886) (-2.901,0.867) (-2.881,0.848) (-2.860,0.830) (-2.839,0.811) (-2.817,0.792) (-2.794,0.772) (-2.771,0.753) (-2.747,0.734) (-2.723,0.715) (-2.698,0.696) (-2.672,0.677) (-2.646,0.658) (-2.619,0.639) (-2.592,0.620) (-2.564,0.600) (-2.535,0.581) (-2.506,0.562) (-2.477,0.543) (-2.447,0.524) (-2.417,0.505) (-2.386,0.485) (-2.355,0.466) (-2.324,0.447) (-2.292,0.428) (-2.260,0.409) (-2.228,0.390) (-2.196,0.371) (-2.163,0.352) (-2.130,0.333) (-2.098,0.314) (-2.065,0.296) (-2.032,0.277) (-2.000,0.258) (-1.967,0.240) (-1.935,0.221) (-1.902,0.203) (-1.870,0.185) (-1.838,0.167) (-1.806,0.149) (-1.775,0.131) (-1.744,0.113)};
\draw[green!50!black,line width=1pt] plot[smooth] coordinates{(-1.594,0.028) (-1.565,0.012) (-1.537,-0.004) (-1.509,-0.019) (-1.482,-0.035) (-1.455,-0.050) (-1.428,-0.064) (-1.403,-0.078) (-1.377,-0.092) (-1.352,-0.106) (-1.328,-0.119) (-1.304,-0.131) (-1.281,-0.144) (-1.258,-0.155) (-1.235,-0.167) (-1.213,-0.178) (-1.192,-0.188) (-1.171,-0.198) (-1.150,-0.207) (-1.130,-0.216) (-1.110,-0.225) (-1.091,-0.233) (-1.072,-0.241) (-1.053,-0.248) (-1.035,-0.255) (-1.018,-0.261) (-1.000,-0.266) (-0.983,-0.272) (-0.966,-0.276) (-0.950,-0.281) (-0.934,-0.285) (-0.919,-0.288) (-0.903,-0.291) (-0.888,-0.294) (-0.874,-0.296) (-0.859,-0.298) (-0.846,-0.299) (-0.832,-0.300) (-0.818,-0.301) (-0.805,-0.301) (-0.793,-0.301) (-0.780,-0.300) (-0.768,-0.299) (-0.756,-0.298) (-0.745,-0.296) (-0.733,-0.294) (-0.722,-0.292) (-0.711,-0.289) (-0.701,-0.286) (-0.691,-0.283) (-0.681,-0.279) (-0.671,-0.275) (-0.661,-0.271) (-0.652,-0.266) (-0.643,-0.261) (-0.634,-0.256) (-0.626,-0.250) (-0.618,-0.245) (-0.610,-0.238) (-0.602,-0.232) (-0.594,-0.225) (-0.587,-0.218) (-0.579,-0.211) (-0.572,-0.203) (-0.565,-0.195) (-0.559,-0.187) (-0.552,-0.179) (-0.546,-0.170) (-0.539,-0.161) (-0.533,-0.152) (-0.527,-0.142) (-0.521,-0.132) (-0.515,-0.122) (-0.509,-0.111) (-0.504,-0.101) (-0.498,-0.090) (-0.492,-0.078) (-0.486,-0.067) (-0.481,-0.055) (-0.475,-0.042) (-0.470,-0.030) (-0.464,-0.017) (-0.459,-0.004) (-0.453,0.010) (-0.448,0.023) (-0.442,0.038) (-0.437,0.052) (-0.432,0.067) (-0.427,0.082) (-0.422,0.097) (-0.417,0.112) (-0.413,0.128) (-0.409,0.145) (-0.405,0.161) (-0.401,0.178) (-0.398,0.195) (-0.395,0.213) (-0.393,0.231) (-0.391,0.249) (-0.390,0.267) (-0.389,0.286) (-0.389,0.305) (-0.390,0.324) (-0.391,0.344) (-0.393,0.363) (-0.396,0.383) (-0.399,0.404) (-0.404,0.424) (-0.409,0.445) (-0.415,0.466) (-0.422,0.487) (-0.429,0.508) (-0.438,0.529) (-0.447,0.551) (-0.458,0.573) (-0.469,0.594) (-0.481,0.616) (-0.494,0.638) (-0.507,0.659) (-0.522,0.681) (-0.537,0.703) (-0.553,0.724) (-0.570,0.746) (-0.588,0.767) (-0.606,0.788) (-0.625,0.809) (-0.645,0.830) (-0.665,0.850) (-0.687,0.871) (-0.708,0.891) (-0.731,0.910) (-0.754,0.930) (-0.777,0.949) (-0.801,0.968) (-0.826,0.986) (-0.851,1.004) (-0.877,1.022) (-0.903,1.040) (-0.929,1.057) (-0.956,1.073) (-0.984,1.090) (-1.012,1.106) (-1.040,1.121) (-1.069,1.137) (-1.098,1.151) (-1.128,1.166) (-1.157,1.180) (-1.188,1.194) (-1.218,1.208) (-1.249,1.221) (-1.279,1.235) (-1.310,1.248) (-1.342,1.260) (-1.373,1.273) (-1.404,1.286) (-1.435,1.298) (-1.466,1.310) (-1.497,1.322)};
\draw[green!50!black,line width=1pt] plot[smooth] coordinates{(-1.646,1.383) (-1.674,1.395) (-1.702,1.408) (-1.729,1.420) (-1.756,1.433) (-1.782,1.445) (-1.807,1.458) (-1.831,1.471) (-1.854,1.484) (-1.877,1.497) (-1.898,1.510) (-1.919,1.524) (-1.938,1.537) (-1.957,1.551) (-1.975,1.564) (-1.991,1.578) (-2.007,1.592) (-2.022,1.606) (-2.035,1.620) (-2.048,1.634) (-2.060,1.649) (-2.071,1.663) (-2.081,1.677) (-2.090,1.691) (-2.098,1.706) (-2.106,1.720) (-2.112,1.734) (-2.118,1.748) (-2.123,1.762) (-2.127,1.776) (-2.131,1.790) (-2.134,1.804) (-2.136,1.817) (-2.137,1.831) (-2.138,1.844) (-2.138,1.857) (-2.138,1.870) (-2.137,1.882) (-2.135,1.895) (-2.133,1.907) (-2.131,1.918) (-2.127,1.930) (-2.124,1.941) (-2.120,1.951) (-2.115,1.961) (-2.110,1.971) (-2.105,1.981) (-2.099,1.990) (-2.092,1.998) (-2.085,2.006) (-2.078,2.014) (-2.070,2.021) (-2.062,2.027) (-2.054,2.033) (-2.045,2.039) (-2.036,2.044) (-2.026,2.049) (-2.016,2.053) (-2.005,2.056) (-1.994,2.059) (-1.983,2.062) (-1.971,2.064) (-1.959,2.065) (-1.946,2.066) (-1.933,2.067) (-1.920,2.067)};
\draw[green!50!black,line width=1pt] plot[smooth] coordinates{(-1.762,2.052) (-1.744,2.050) (-1.725,2.047) (-1.706,2.044) (-1.687,2.041) (-1.668,2.039) (-1.648,2.036) (-1.628,2.033) (-1.607,2.031) (-1.587,2.029) (-1.566,2.027) (-1.544,2.025) (-1.523,2.023) (-1.501,2.022) (-1.480,2.020) (-1.458,2.020) (-1.436,2.019) (-1.413,2.019) (-1.391,2.019) (-1.369,2.019) (-1.346,2.020) (-1.324,2.022) (-1.301,2.023) (-1.279,2.025) (-1.256,2.028) (-1.234,2.031) (-1.212,2.034) (-1.190,2.038) (-1.167,2.042) (-1.145,2.047) (-1.124,2.052) (-1.102,2.057) (-1.081,2.063) (-1.059,2.070) (-1.038,2.076) (-1.018,2.084) (-0.997,2.091) (-0.977,2.100) (-0.958,2.108) (-0.938,2.117) (-0.919,2.127) (-0.901,2.137) (-0.883,2.147) (-0.865,2.158) (-0.848,2.170) (-0.831,2.182) (-0.815,2.194) (-0.799,2.207) (-0.784,2.221) (-0.770,2.235) (-0.756,2.249) (-0.743,2.264) (-0.730,2.280) (-0.718,2.296) (-0.706,2.313) (-0.696,2.330) (-0.685,2.348) (-0.676,2.366) (-0.667,2.385) (-0.659,2.404) (-0.651,2.425) (-0.644,2.445) (-0.638,2.466) (-0.632,2.488) (-0.627,2.510) (-0.622,2.533) (-0.618,2.556) (-0.615,2.580) (-0.612,2.604) (-0.610,2.629) (-0.608,2.654) (-0.607,2.680) (-0.607,2.706) (-0.607,2.732) (-0.607,2.759) (-0.608,2.786) (-0.610,2.813) (-0.612,2.841) (-0.614,2.868) (-0.617,2.896)};
\draw[green!50!black,line width=1pt] plot[smooth] coordinates{(-0.642,3.065) (-0.647,3.093) (-0.653,3.121) (-0.659,3.149) (-0.666,3.176) (-0.672,3.204) (-0.679,3.231) (-0.687,3.258) (-0.694,3.284) (-0.702,3.310) (-0.710,3.336) (-0.719,3.362) (-0.728,3.387) (-0.737,3.411) (-0.746,3.435) (-0.755,3.459) (-0.765,3.483) (-0.775,3.505) (-0.786,3.528) (-0.796,3.550) (-0.807,3.571) (-0.818,3.592) (-0.830,3.612) (-0.841,3.632) (-0.853,3.652) (-0.865,3.671) (-0.878,3.689) (-0.891,3.707) (-0.904,3.724) (-0.917,3.741) (-0.931,3.758) (-0.945,3.773) (-0.959,3.789) (-0.974,3.803) (-0.988,3.818) (-1.003,3.831) (-1.019,3.844) (-1.035,3.857) (-1.050,3.869) (-1.067,3.881) (-1.083,3.892) (-1.100,3.902) (-1.117,3.912) (-1.135,3.921) (-1.152,3.930) (-1.170,3.938) (-1.189,3.945) (-1.207,3.952) (-1.226,3.959) (-1.245,3.965) (-1.264,3.970) (-1.284,3.974) (-1.304,3.978) (-1.324,3.982) (-1.344,3.984) (-1.365,3.986) (-1.386,3.988) (-1.407,3.989) (-1.428,3.989) (-1.449,3.989) (-1.471,3.988) (-1.492,3.986) (-1.514,3.984) (-1.536,3.981) (-1.558,3.978) (-1.580,3.974) (-1.601,3.969) (-1.623,3.964) (-1.645,3.958) (-1.667,3.952) (-1.689,3.944) (-1.711,3.937) (-1.732,3.928) (-1.754,3.920) (-1.775,3.910) (-1.796,3.900) (-1.817,3.889) (-1.837,3.878) (-1.858,3.866) (-1.878,3.854) (-1.897,3.841) (-1.917,3.827) (-1.936,3.813) (-1.955,3.798) (-1.974,3.783) (-1.992,3.767) (-2.010,3.750) (-2.027,3.733) (-2.045,3.715) (-2.062,3.697) (-2.078,3.678) (-2.094,3.658) (-2.110,3.638) (-2.126,3.617) (-2.141,3.596) (-2.156,3.573) (-2.171,3.551) (-2.185,3.528) (-2.199,3.504) (-2.212,3.480) (-2.225,3.455) (-2.238,3.430) (-2.251,3.404) (-2.263,3.377) (-2.275,3.351) (-2.287,3.323) (-2.298,3.296) (-2.309,3.268) (-2.320,3.239) (-2.330,3.210) (-2.340,3.181) (-2.349,3.152) (-2.359,3.123) (-2.367,3.093) (-2.376,3.063) (-2.384,3.033) (-2.392,3.003) (-2.399,2.973) (-2.406,2.943) (-2.412,2.913) (-2.418,2.884) (-2.423,2.854) (-2.428,2.824) (-2.433,2.795) (-2.437,2.766) (-2.440,2.738) (-2.443,2.709) (-2.446,2.681) (-2.448,2.654) (-2.449,2.627) (-2.450,2.600) (-2.450,2.574) (-2.450,2.548) (-2.449,2.523) (-2.448,2.499) (-2.445,2.475) (-2.443,2.451) (-2.440,2.428) (-2.436,2.406) (-2.431,2.384) (-2.426,2.363) (-2.420,2.342) (-2.414,2.323) (-2.407,2.303) (-2.400,2.285) (-2.392,2.267) (-2.384,2.249) (-2.375,2.232) (-2.365,2.216) (-2.355,2.200) (-2.345,2.185) (-2.334,2.171) (-2.322,2.157) (-2.310,2.144) (-2.298,2.131) (-2.286,2.119) (-2.273,2.108) (-2.260,2.097) (-2.246,2.087) (-2.233,2.077) (-2.219,2.068) (-2.205,2.059) (-2.190,2.052) (-2.176,2.044) (-2.162,2.038) (-2.147,2.032)};
\draw[green!50!black,line width=1pt] plot[smooth] coordinates{(-2.004,2.005) (-1.990,2.005) (-1.977,2.007) (-1.964,2.009) (-1.951,2.011) (-1.938,2.014) (-1.925,2.018) (-1.913,2.022) (-1.901,2.027) (-1.889,2.032) (-1.877,2.038) (-1.865,2.045) (-1.854,2.051) (-1.843,2.059) (-1.832,2.067) (-1.821,2.075) (-1.810,2.084) (-1.800,2.093) (-1.789,2.102) (-1.779,2.112) (-1.768,2.123) (-1.758,2.133) (-1.748,2.144) (-1.738,2.155) (-1.728,2.166) (-1.718,2.178) (-1.708,2.190) (-1.698,2.202) (-1.688,2.214) (-1.678,2.226) (-1.668,2.238) (-1.658,2.251) (-1.647,2.263) (-1.637,2.276) (-1.627,2.289) (-1.616,2.301) (-1.606,2.314) (-1.595,2.326) (-1.584,2.339) (-1.573,2.352) (-1.562,2.364) (-1.551,2.377) (-1.540,2.389) (-1.529,2.402) (-1.517,2.414) (-1.506,2.427) (-1.494,2.439) (-1.483,2.451) (-1.471,2.463) (-1.460,2.475) (-1.448,2.487) (-1.437,2.499) (-1.426,2.511) (-1.414,2.523) (-1.403,2.535) (-1.392,2.547) (-1.382,2.558) (-1.371,2.570) (-1.361,2.582) (-1.351,2.593) (-1.341,2.605) (-1.332,2.616) (-1.322,2.628) (-1.314,2.639) (-1.305,2.651) (-1.298,2.662) (-1.290,2.673) (-1.283,2.685) (-1.277,2.696) (-1.271,2.707) (-1.266,2.718) (-1.261,2.729) (-1.257,2.741) (-1.254,2.752) (-1.251,2.763) (-1.249,2.774) (-1.248,2.785) (-1.247,2.796) (-1.247,2.807) (-1.248,2.818) (-1.249,2.830) (-1.252,2.841) (-1.255,2.852) (-1.259,2.862) (-1.263,2.873)};
\draw[green!50!black,line width=1pt] plot[smooth] coordinates{(-1.367,2.989) (-1.381,2.999) (-1.396,3.008) (-1.411,3.018) (-1.428,3.027) (-1.445,3.036) (-1.463,3.044) (-1.481,3.053) (-1.501,3.061) (-1.521,3.068) (-1.541,3.076) (-1.562,3.083) (-1.584,3.090) (-1.607,3.096) (-1.630,3.102) (-1.653,3.107) (-1.677,3.112) (-1.701,3.117) (-1.726,3.121) (-1.751,3.125) (-1.776,3.128) (-1.802,3.130) (-1.828,3.132) (-1.855,3.134) (-1.881,3.135) (-1.908,3.135) (-1.935,3.135) (-1.962,3.134) (-1.989,3.133) (-2.016,3.132) (-2.044,3.129) (-2.071,3.126) (-2.098,3.123) (-2.125,3.119) (-2.153,3.115) (-2.180,3.110) (-2.207,3.104) (-2.234,3.098) (-2.260,3.092) (-2.287,3.085)};
\fill (0.103,1.684) circle (2.4pt);
\node[green!50!black,anchor=east] at (-3.328,1.800) {$-8$};
\end{tikzpicture}
\hspace{1.2cm}
\begin{tikzpicture}[line cap=round,line join=round]
\draw[black,line width=1pt] plot[smooth] coordinates{(-1.219,2.662) (-1.209,2.651) (-1.200,2.639) (-1.190,2.628) (-1.180,2.616) (-1.171,2.604) (-1.161,2.592) (-1.151,2.579) (-1.142,2.567) (-1.132,2.554) (-1.123,2.540) (-1.113,2.527) (-1.104,2.513) (-1.095,2.499) (-1.087,2.484) (-1.078,2.470) (-1.070,2.455) (-1.062,2.439) (-1.054,2.424) (-1.046,2.408) (-1.039,2.391) (-1.032,2.375) (-1.025,2.358) (-1.019,2.340) (-1.013,2.323) (-1.007,2.305) (-1.001,2.287) (-0.996,2.268) (-0.991,2.250) (-0.987,2.231) (-0.982,2.211) (-0.979,2.192)};
\draw[black,line width=1pt] plot[smooth] coordinates{(-0.962,2.030) (-0.962,2.009) (-0.963,1.988) (-0.964,1.967) (-0.965,1.947) (-0.967,1.926) (-0.970,1.905) (-0.973,1.885) (-0.977,1.864) (-0.981,1.844) (-0.987,1.824) (-0.993,1.804) (-0.999,1.784) (-1.007,1.764) (-1.015,1.745) (-1.024,1.725) (-1.034,1.707) (-1.045,1.688) (-1.056,1.670) (-1.069,1.652) (-1.082,1.634) (-1.097,1.617) (-1.112,1.600) (-1.128,1.583) (-1.146,1.567) (-1.164,1.551) (-1.183,1.536) (-1.204,1.521) (-1.225,1.506) (-1.247,1.492) (-1.271,1.478) (-1.295,1.464) (-1.320,1.451) (-1.346,1.438) (-1.373,1.425) (-1.401,1.413) (-1.429,1.401) (-1.459,1.389) (-1.489,1.377) (-1.519,1.366) (-1.551,1.355) (-1.582,1.343) (-1.615,1.332) (-1.647,1.321) (-1.680,1.311) (-1.714,1.300) (-1.747,1.289) (-1.781,1.278) (-1.814,1.266) (-1.848,1.255) (-1.882,1.244) (-1.915,1.232) (-1.949,1.220) (-1.982,1.208) (-2.015,1.196) (-2.048,1.183) (-2.080,1.170) (-2.112,1.157) (-2.144,1.144) (-2.175,1.130) (-2.206,1.115) (-2.236,1.101) (-2.266,1.086) (-2.296,1.070) (-2.325,1.055) (-2.353,1.039) (-2.381,1.022) (-2.409,1.005) (-2.435,0.988) (-2.462,0.971) (-2.487,0.953) (-2.513,0.935) (-2.537,0.916) (-2.561,0.898) (-2.584,0.879) (-2.607,0.860) (-2.629,0.840) (-2.650,0.821) (-2.670,0.801) (-2.690,0.781)};
\draw[black,line width=1pt] plot[smooth] coordinates{(-2.791,0.660) (-2.805,0.640) (-2.819,0.619) (-2.831,0.599) (-2.842,0.579) (-2.853,0.559) (-2.862,0.539) (-2.871,0.519) (-2.879,0.500) (-2.886,0.480) (-2.892,0.461) (-2.898,0.442) (-2.902,0.423) (-2.906,0.405) (-2.909,0.386) (-2.911,0.368) (-2.912,0.350) (-2.913,0.332) (-2.913,0.315) (-2.912,0.298) (-2.911,0.281) (-2.909,0.264) (-2.906,0.248) (-2.904,0.231) (-2.900,0.215) (-2.896,0.200) (-2.892,0.184) (-2.887,0.169) (-2.882,0.155) (-2.877,0.140) (-2.872,0.126) (-2.866,0.112) (-2.860,0.098) (-2.854,0.084) (-2.848,0.071) (-2.842,0.058) (-2.835,0.045) (-2.829,0.033) (-2.823,0.021) (-2.816,0.009) (-2.810,-0.003) (-2.803,-0.015) (-2.797,-0.026) (-2.791,-0.037) (-2.784,-0.047) (-2.778,-0.058) (-2.772,-0.068) (-2.765,-0.078) (-2.759,-0.087) (-2.752,-0.096) (-2.746,-0.106) (-2.740,-0.114) (-2.733,-0.123) (-2.726,-0.131) (-2.720,-0.139) (-2.713,-0.147) (-2.706,-0.154) (-2.698,-0.161) (-2.691,-0.168) (-2.684,-0.175) (-2.676,-0.181) (-2.668,-0.187) (-2.660,-0.193) (-2.652,-0.198) (-2.643,-0.204) (-2.635,-0.209) (-2.626,-0.213) (-2.616,-0.217) (-2.607,-0.221) (-2.597,-0.225) (-2.587,-0.228) (-2.577,-0.231) (-2.566,-0.234) (-2.555,-0.236) (-2.544,-0.238) (-2.533,-0.240) (-2.521,-0.241) (-2.509,-0.242) (-2.497,-0.243) (-2.484,-0.243) (-2.471,-0.243) (-2.458,-0.242) (-2.444,-0.241) (-2.430,-0.240) (-2.416,-0.238) (-2.401,-0.236) (-2.386,-0.233) (-2.371,-0.230) (-2.355,-0.227) (-2.339,-0.223) (-2.323,-0.218) (-2.306,-0.213) (-2.289,-0.208) (-2.271,-0.202) (-2.253,-0.196) (-2.234,-0.189) (-2.215,-0.182) (-2.195,-0.175) (-2.176,-0.166) (-2.155,-0.158) (-2.134,-0.149) (-2.113,-0.140) (-2.091,-0.130) (-2.068,-0.119) (-2.045,-0.109) (-2.021,-0.097) (-1.997,-0.086) (-1.973,-0.074) (-1.947,-0.062) (-1.921,-0.049) (-1.895,-0.036) (-1.868,-0.023) (-1.840,-0.009) (-1.812,0.005) (-1.783,0.020) (-1.753,0.034) (-1.723,0.049) (-1.692,0.064) (-1.661,0.079) (-1.629,0.095) (-1.597,0.111) (-1.564,0.127) (-1.531,0.143) (-1.498,0.159) (-1.464,0.175) (-1.429,0.192) (-1.395,0.209) (-1.360,0.225) (-1.325,0.242) (-1.289,0.259) (-1.254,0.276) (-1.218,0.293) (-1.183,0.310) (-1.147,0.327) (-1.112,0.344) (-1.076,0.362) (-1.041,0.379) (-1.006,0.396) (-0.971,0.413) (-0.937,0.431) (-0.903,0.448) (-0.869,0.465) (-0.836,0.482) (-0.803,0.500) (-0.770,0.517) (-0.738,0.534) (-0.707,0.552) (-0.676,0.569) (-0.645,0.586) (-0.616,0.603) (-0.587,0.621)};
\draw[black,line width=1pt] plot[smooth] coordinates{(-0.451,0.707) (-0.426,0.724) (-0.402,0.741) (-0.378,0.758) (-0.355,0.776) (-0.333,0.793) (-0.312,0.810) (-0.291,0.827) (-0.271,0.844) (-0.251,0.861) (-0.232,0.878) (-0.214,0.895) (-0.197,0.912) (-0.180,0.929) (-0.164,0.946) (-0.148,0.963) (-0.133,0.980) (-0.119,0.996) (-0.105,1.013) (-0.092,1.030) (-0.079,1.047) (-0.067,1.064) (-0.056,1.081) (-0.045,1.098) (-0.034,1.115) (-0.024,1.132) (-0.014,1.149) (-0.005,1.166) (0.004,1.183) (0.012,1.201) (0.020,1.218) (0.027,1.236) (0.034,1.254) (0.041,1.272) (0.047,1.290) (0.053,1.309) (0.058,1.328) (0.064,1.347) (0.068,1.366) (0.073,1.386) (0.077,1.405) (0.081,1.426) (0.084,1.446) (0.088,1.467) (0.090,1.487) (0.093,1.508) (0.095,1.530) (0.097,1.551) (0.099,1.573) (0.100,1.595) (0.102,1.617) (0.102,1.639) (0.103,1.661) (0.103,1.684) (0.103,1.706) (0.103,1.728) (0.102,1.751) (0.101,1.773) (0.100,1.796) (0.099,1.818) (0.097,1.840) (0.095,1.863) (0.093,1.885) (0.090,1.907) (0.087,1.929) (0.084,1.950) (0.080,1.972) (0.077,1.994) (0.073,2.015) (0.068,2.036) (0.064,2.057) (0.059,2.078) (0.054,2.099) (0.048,2.120) (0.043,2.140) (0.037,2.161) (0.030,2.181) (0.024,2.201) (0.017,2.221) (0.010,2.240) (0.003,2.260) (-0.005,2.279) (-0.013,2.299) (-0.021,2.318) (-0.029,2.337) (-0.038,2.355) (-0.047,2.374) (-0.056,2.392) (-0.065,2.411) (-0.075,2.429) (-0.085,2.447) (-0.095,2.465) (-0.105,2.482) (-0.116,2.500) (-0.127,2.517) (-0.138,2.535) (-0.149,2.552) (-0.161,2.569) (-0.173,2.586) (-0.185,2.602) (-0.197,2.619) (-0.210,2.635) (-0.223,2.652) (-0.236,2.668) (-0.249,2.684) (-0.263,2.700) (-0.277,2.715) (-0.291,2.731) (-0.305,2.746) (-0.320,2.761) (-0.335,2.776) (-0.350,2.791) (-0.366,2.805) (-0.381,2.819) (-0.397,2.833) (-0.414,2.847) (-0.430,2.861) (-0.447,2.874) (-0.464,2.887) (-0.481,2.899) (-0.498,2.912) (-0.515,2.924) (-0.533,2.936) (-0.551,2.947) (-0.569,2.958) (-0.587,2.968) (-0.605,2.979) (-0.623,2.988) (-0.641,2.998) (-0.660,3.007) (-0.678,3.015) (-0.697,3.023) (-0.715,3.031) (-0.734,3.038) (-0.752,3.044) (-0.771,3.050) (-0.789,3.056) (-0.808,3.061) (-0.826,3.065) (-0.844,3.069) (-0.862,3.073) (-0.880,3.075) (-0.898,3.078) (-0.916,3.080) (-0.933,3.081) (-0.950,3.081) (-0.968,3.082) (-0.984,3.081) (-1.001,3.080) (-1.018,3.079) (-1.034,3.077) (-1.050,3.075) (-1.065,3.072) (-1.081,3.069) (-1.096,3.065) (-1.111,3.061) (-1.125,3.057) (-1.139,3.052) (-1.153,3.046) (-1.166,3.041) (-1.179,3.035) (-1.192,3.028) (-1.204,3.022) (-1.215,3.015) (-1.226,3.008) (-1.237,3.000) (-1.247,2.993) (-1.257,2.985) (-1.266,2.977) (-1.274,2.969) (-1.282,2.960) (-1.289,2.952) (-1.296,2.943) (-1.302,2.934) (-1.307,2.926) (-1.312,2.917) (-1.316,2.908) (-1.320,2.898) (-1.322,2.889) (-1.324,2.880) (-1.326,2.871) (-1.326,2.862) (-1.326,2.852) (-1.326,2.843) (-1.324,2.833) (-1.322,2.824) (-1.320,2.814) (-1.316,2.805) (-1.313,2.795) (-1.308,2.785)};
\draw[green!50!black,line width=1pt] plot[smooth] coordinates{(-2.465,3.025) (-2.490,3.015) (-2.514,3.004) (-2.537,2.993) (-2.560,2.982) (-2.583,2.971) (-2.606,2.959) (-2.628,2.947) (-2.649,2.935) (-2.671,2.923) (-2.691,2.910) (-2.712,2.897) (-2.732,2.884) (-2.751,2.871) (-2.770,2.858) (-2.789,2.844) (-2.807,2.830) (-2.825,2.817) (-2.843,2.803) (-2.860,2.788) (-2.876,2.774) (-2.893,2.760) (-2.908,2.745) (-2.924,2.730) (-2.939,2.715) (-2.954,2.700) (-2.968,2.685) (-2.982,2.670) (-2.996,2.655) (-3.009,2.639) (-3.022,2.623) (-3.034,2.608) (-3.047,2.592) (-3.059,2.576) (-3.070,2.559) (-3.081,2.543) (-3.092,2.526) (-3.103,2.510) (-3.113,2.493) (-3.123,2.476) (-3.133,2.458) (-3.143,2.441) (-3.152,2.423) (-3.160,2.406) (-3.169,2.388) (-3.177,2.369) (-3.185,2.351) (-3.192,2.332) (-3.200,2.314) (-3.207,2.295) (-3.213,2.275) (-3.219,2.256) (-3.225,2.236) (-3.231,2.216) (-3.237,2.196) (-3.242,2.176) (-3.246,2.156) (-3.251,2.135) (-3.255,2.114) (-3.259,2.093) (-3.262,2.071) (-3.266,2.050) (-3.269,2.028) (-3.271,2.006) (-3.274,1.984) (-3.276,1.961) (-3.278,1.939) (-3.279,1.916) (-3.280,1.893) (-3.281,1.870) (-3.282,1.847) (-3.282,1.824) (-3.282,1.800) (-3.282,1.777) (-3.282,1.753) (-3.281,1.730) (-3.280,1.706) (-3.278,1.683) (-3.276,1.659) (-3.274,1.636) (-3.272,1.612) (-3.269,1.589) (-3.267,1.566) (-3.263,1.543) (-3.260,1.520) (-3.256,1.498) (-3.251,1.475) (-3.247,1.453) (-3.242,1.431) (-3.237,1.410) (-3.231,1.388) (-3.225,1.367) (-3.218,1.347) (-3.211,1.326) (-3.204,1.306) (-3.197,1.285) (-3.188,1.265) (-3.180,1.246) (-3.171,1.226) (-3.161,1.207) (-3.151,1.187) (-3.141,1.168) (-3.130,1.149) (-3.119,1.130) (-3.107,1.111) (-3.094,1.093) (-3.081,1.074) (-3.068,1.055) (-3.054,1.036) (-3.039,1.018) (-3.024,0.999) (-3.008,0.980) (-2.992,0.962) (-2.975,0.943) (-2.957,0.924) (-2.939,0.905) (-2.920,0.886) (-2.901,0.867) (-2.881,0.848) (-2.860,0.830) (-2.839,0.811) (-2.817,0.792) (-2.794,0.772) (-2.771,0.753) (-2.747,0.734) (-2.723,0.715) (-2.698,0.696) (-2.672,0.677) (-2.646,0.658) (-2.619,0.639) (-2.592,0.620) (-2.564,0.600) (-2.535,0.581) (-2.506,0.562) (-2.477,0.543) (-2.447,0.524) (-2.417,0.505) (-2.386,0.485) (-2.355,0.466) (-2.324,0.447) (-2.292,0.428) (-2.260,0.409) (-2.228,0.390) (-2.196,0.371) (-2.163,0.352) (-2.130,0.333) (-2.098,0.314) (-2.065,0.296) (-2.032,0.277) (-2.000,0.258) (-1.967,0.240) (-1.935,0.221) (-1.902,0.203) (-1.870,0.185) (-1.838,0.167) (-1.806,0.149) (-1.775,0.131) (-1.744,0.113)};
\draw[green!50!black,line width=1pt] plot[smooth] coordinates{(-1.594,0.028) (-1.565,0.012) (-1.537,-0.004) (-1.509,-0.019) (-1.482,-0.035) (-1.455,-0.050) (-1.428,-0.064) (-1.403,-0.078) (-1.377,-0.092) (-1.352,-0.106) (-1.328,-0.119) (-1.304,-0.131) (-1.281,-0.144) (-1.258,-0.155) (-1.235,-0.167) (-1.213,-0.178) (-1.192,-0.188) (-1.171,-0.198) (-1.150,-0.207) (-1.130,-0.216) (-1.110,-0.225) (-1.091,-0.233) (-1.072,-0.241) (-1.053,-0.248) (-1.035,-0.255) (-1.018,-0.261) (-1.000,-0.266) (-0.983,-0.272) (-0.966,-0.276) (-0.950,-0.281) (-0.934,-0.285) (-0.919,-0.288) (-0.903,-0.291) (-0.888,-0.294) (-0.874,-0.296) (-0.859,-0.298) (-0.846,-0.299) (-0.832,-0.300) (-0.818,-0.301) (-0.805,-0.301) (-0.793,-0.301) (-0.780,-0.300) (-0.768,-0.299) (-0.756,-0.298) (-0.745,-0.296) (-0.733,-0.294) (-0.722,-0.292) (-0.711,-0.289) (-0.701,-0.286) (-0.691,-0.283) (-0.681,-0.279) (-0.671,-0.275) (-0.661,-0.271) (-0.652,-0.266) (-0.643,-0.261) (-0.634,-0.256) (-0.626,-0.250) (-0.618,-0.245) (-0.610,-0.238) (-0.602,-0.232) (-0.594,-0.225) (-0.587,-0.218) (-0.579,-0.211) (-0.572,-0.203) (-0.565,-0.195) (-0.559,-0.187) (-0.552,-0.179) (-0.546,-0.170) (-0.539,-0.161) (-0.533,-0.152) (-0.527,-0.142) (-0.521,-0.132) (-0.515,-0.122) (-0.509,-0.111) (-0.504,-0.101) (-0.498,-0.090) (-0.492,-0.078) (-0.486,-0.067) (-0.481,-0.055) (-0.475,-0.042) (-0.470,-0.030) (-0.464,-0.017) (-0.459,-0.004) (-0.453,0.010) (-0.448,0.023) (-0.442,0.038) (-0.437,0.052) (-0.432,0.067) (-0.427,0.082) (-0.422,0.097) (-0.417,0.112) (-0.413,0.128) (-0.409,0.145) (-0.405,0.161) (-0.401,0.178) (-0.398,0.195) (-0.395,0.213) (-0.393,0.231) (-0.391,0.249) (-0.390,0.267) (-0.389,0.286) (-0.389,0.305) (-0.390,0.324) (-0.391,0.344) (-0.393,0.363) (-0.396,0.383) (-0.399,0.404) (-0.404,0.424) (-0.409,0.445) (-0.415,0.466) (-0.422,0.487) (-0.429,0.508) (-0.438,0.529) (-0.447,0.551) (-0.458,0.573) (-0.469,0.594) (-0.481,0.616) (-0.494,0.638) (-0.507,0.659) (-0.522,0.681) (-0.537,0.703) (-0.553,0.724) (-0.570,0.746) (-0.588,0.767) (-0.606,0.788) (-0.625,0.809) (-0.645,0.830) (-0.665,0.850) (-0.687,0.871) (-0.708,0.891) (-0.731,0.910) (-0.754,0.930) (-0.777,0.949) (-0.801,0.968) (-0.826,0.986) (-0.851,1.004) (-0.877,1.022) (-0.903,1.040) (-0.929,1.057) (-0.956,1.073) (-0.984,1.090) (-1.012,1.106) (-1.040,1.121) (-1.069,1.137) (-1.098,1.151) (-1.128,1.166) (-1.157,1.180) (-1.188,1.194) (-1.218,1.208) (-1.249,1.221) (-1.279,1.235) (-1.310,1.248) (-1.342,1.260) (-1.373,1.273) (-1.404,1.286) (-1.435,1.298) (-1.466,1.310) (-1.497,1.322)};
\draw[green!50!black,line width=1pt] plot[smooth] coordinates{(-1.646,1.383) (-1.674,1.395) (-1.702,1.408) (-1.729,1.420) (-1.756,1.433) (-1.782,1.445) (-1.807,1.458) (-1.831,1.471) (-1.854,1.484) (-1.877,1.497) (-1.898,1.510) (-1.919,1.524) (-1.938,1.537) (-1.957,1.551) (-1.975,1.564) (-1.991,1.578) (-2.007,1.592) (-2.022,1.606) (-2.035,1.620) (-2.048,1.634) (-2.060,1.649) (-2.071,1.663) (-2.081,1.677) (-2.090,1.691) (-2.098,1.706) (-2.106,1.720) (-2.112,1.734) (-2.118,1.748) (-2.123,1.762) (-2.127,1.776) (-2.131,1.790) (-2.134,1.804) (-2.136,1.817) (-2.137,1.831) (-2.138,1.844) (-2.138,1.857) (-2.138,1.870) (-2.137,1.882) (-2.135,1.895) (-2.133,1.907) (-2.131,1.918) (-2.127,1.930) (-2.124,1.941) (-2.120,1.951) (-2.115,1.961) (-2.110,1.971) (-2.105,1.981) (-2.099,1.990) (-2.092,1.998) (-2.085,2.006) (-2.078,2.014) (-2.070,2.021) (-2.062,2.027) (-2.054,2.033) (-2.045,2.039) (-2.036,2.044) (-2.026,2.049) (-2.016,2.053) (-2.005,2.056) (-1.994,2.059) (-1.983,2.062) (-1.971,2.064) (-1.959,2.065) (-1.946,2.066) (-1.933,2.067) (-1.920,2.067)};
\draw[green!50!black,line width=1pt] plot[smooth] coordinates{(-1.762,2.052) (-1.744,2.050) (-1.725,2.047) (-1.706,2.044) (-1.687,2.041) (-1.668,2.039) (-1.648,2.036) (-1.628,2.033) (-1.607,2.031) (-1.587,2.029) (-1.566,2.027) (-1.544,2.025) (-1.523,2.023) (-1.501,2.022) (-1.480,2.020) (-1.458,2.020) (-1.436,2.019) (-1.413,2.019) (-1.391,2.019) (-1.369,2.019) (-1.346,2.020) (-1.324,2.022) (-1.301,2.023) (-1.279,2.025) (-1.256,2.028) (-1.234,2.031) (-1.212,2.034) (-1.190,2.038) (-1.167,2.042) (-1.145,2.047) (-1.124,2.052) (-1.102,2.057) (-1.081,2.063) (-1.059,2.070) (-1.038,2.076) (-1.018,2.084) (-0.997,2.091) (-0.977,2.100) (-0.958,2.108) (-0.938,2.117) (-0.919,2.127) (-0.901,2.137) (-0.883,2.147) (-0.865,2.158) (-0.848,2.170) (-0.831,2.182) (-0.815,2.194) (-0.799,2.207) (-0.784,2.221) (-0.770,2.235) (-0.756,2.249) (-0.743,2.264) (-0.730,2.280) (-0.718,2.296) (-0.706,2.313) (-0.696,2.330) (-0.685,2.348) (-0.676,2.366) (-0.667,2.385) (-0.659,2.404) (-0.651,2.425) (-0.644,2.445) (-0.638,2.466) (-0.632,2.488) (-0.627,2.510) (-0.622,2.533) (-0.618,2.556) (-0.615,2.580) (-0.612,2.604) (-0.610,2.629) (-0.608,2.654) (-0.607,2.680) (-0.607,2.706) (-0.607,2.732) (-0.607,2.759) (-0.608,2.786) (-0.610,2.813) (-0.612,2.841) (-0.614,2.868) (-0.617,2.896)};
\draw[green!50!black,line width=1pt] plot[smooth] coordinates{(-0.642,3.065) (-0.647,3.093) (-0.653,3.121) (-0.659,3.149) (-0.666,3.176) (-0.672,3.204) (-0.679,3.231) (-0.687,3.258) (-0.694,3.284) (-0.702,3.310) (-0.710,3.336) (-0.719,3.362) (-0.728,3.387) (-0.737,3.411) (-0.746,3.435) (-0.755,3.459) (-0.765,3.483) (-0.775,3.505) (-0.786,3.528) (-0.796,3.550) (-0.807,3.571) (-0.818,3.592) (-0.830,3.612) (-0.841,3.632) (-0.853,3.652) (-0.865,3.671) (-0.878,3.689) (-0.891,3.707) (-0.904,3.724) (-0.917,3.741) (-0.931,3.758) (-0.945,3.773) (-0.959,3.789) (-0.974,3.803) (-0.988,3.818) (-1.003,3.831) (-1.019,3.844) (-1.035,3.857) (-1.050,3.869) (-1.067,3.881) (-1.083,3.892) (-1.100,3.902) (-1.117,3.912) (-1.135,3.921) (-1.152,3.930) (-1.170,3.938) (-1.189,3.945) (-1.207,3.952) (-1.226,3.959) (-1.245,3.965) (-1.264,3.970) (-1.284,3.974) (-1.304,3.978) (-1.324,3.982) (-1.344,3.984) (-1.365,3.986) (-1.386,3.988) (-1.407,3.989) (-1.428,3.989) (-1.449,3.989) (-1.471,3.988) (-1.492,3.986) (-1.514,3.984) (-1.536,3.981) (-1.558,3.978) (-1.580,3.974) (-1.601,3.969) (-1.623,3.964) (-1.645,3.958) (-1.667,3.952) (-1.689,3.944) (-1.711,3.937) (-1.732,3.928) (-1.754,3.920) (-1.775,3.910) (-1.796,3.900) (-1.817,3.889) (-1.837,3.878) (-1.858,3.866) (-1.878,3.854) (-1.897,3.841) (-1.917,3.827) (-1.936,3.813) (-1.955,3.798) (-1.974,3.783) (-1.992,3.767) (-2.010,3.750) (-2.027,3.733) (-2.045,3.715) (-2.062,3.697) (-2.078,3.678) (-2.094,3.658) (-2.110,3.638) (-2.126,3.617) (-2.141,3.596) (-2.156,3.573) (-2.171,3.551) (-2.185,3.528) (-2.199,3.504) (-2.212,3.480) (-2.225,3.455) (-2.238,3.430) (-2.251,3.404) (-2.263,3.377) (-2.275,3.351) (-2.287,3.323) (-2.298,3.296) (-2.309,3.268) (-2.320,3.239) (-2.330,3.210) (-2.340,3.181) (-2.349,3.152) (-2.359,3.123) (-2.367,3.093) (-2.376,3.063) (-2.384,3.033) (-2.392,3.003) (-2.399,2.973) (-2.406,2.943) (-2.412,2.913) (-2.418,2.884) (-2.423,2.854) (-2.428,2.824) (-2.433,2.795) (-2.437,2.766) (-2.440,2.738) (-2.443,2.709) (-2.446,2.681) (-2.448,2.654) (-2.449,2.627) (-2.450,2.600) (-2.450,2.574) (-2.450,2.548) (-2.449,2.523) (-2.448,2.499) (-2.445,2.475) (-2.443,2.451) (-2.440,2.428) (-2.436,2.406) (-2.431,2.384) (-2.426,2.363) (-2.420,2.342) (-2.414,2.323) (-2.407,2.303) (-2.400,2.285) (-2.392,2.267) (-2.384,2.249) (-2.375,2.232) (-2.365,2.216) (-2.355,2.200) (-2.345,2.185) (-2.334,2.171) (-2.322,2.157) (-2.310,2.144) (-2.298,2.131) (-2.286,2.119) (-2.273,2.108) (-2.260,2.097) (-2.246,2.087) (-2.233,2.077) (-2.219,2.068) (-2.205,2.059) (-2.190,2.052) (-2.176,2.044) (-2.162,2.038) (-2.147,2.032)};
\draw[green!50!black,line width=1pt] plot[smooth] coordinates{(-2.004,2.005) (-1.990,2.005) (-1.977,2.007) (-1.964,2.009) (-1.951,2.011) (-1.938,2.014) (-1.925,2.018) (-1.913,2.022) (-1.901,2.027) (-1.889,2.032) (-1.877,2.038) (-1.865,2.045) (-1.854,2.051) (-1.843,2.059) (-1.832,2.067) (-1.821,2.075) (-1.810,2.084) (-1.800,2.093) (-1.789,2.102) (-1.779,2.112) (-1.768,2.123) (-1.758,2.133) (-1.748,2.144) (-1.738,2.155) (-1.728,2.166) (-1.718,2.178) (-1.708,2.190) (-1.698,2.202) (-1.688,2.214) (-1.678,2.226) (-1.668,2.238) (-1.658,2.251) (-1.647,2.263) (-1.637,2.276) (-1.627,2.289) (-1.616,2.301) (-1.606,2.314) (-1.595,2.326) (-1.584,2.339) (-1.573,2.352) (-1.562,2.364) (-1.551,2.377) (-1.540,2.389) (-1.529,2.402) (-1.517,2.414) (-1.506,2.427) (-1.494,2.439) (-1.483,2.451) (-1.471,2.463) (-1.460,2.475) (-1.448,2.487) (-1.437,2.499) (-1.426,2.511) (-1.414,2.523) (-1.403,2.535) (-1.392,2.547) (-1.382,2.558) (-1.371,2.570) (-1.361,2.582) (-1.351,2.593) (-1.341,2.605) (-1.332,2.616) (-1.322,2.628) (-1.314,2.639) (-1.305,2.651) (-1.298,2.662) (-1.290,2.673) (-1.283,2.685) (-1.277,2.696) (-1.271,2.707) (-1.266,2.718) (-1.261,2.729) (-1.257,2.741) (-1.254,2.752) (-1.251,2.763) (-1.249,2.774) (-1.248,2.785) (-1.247,2.796) (-1.247,2.807) (-1.248,2.818) (-1.249,2.830) (-1.252,2.841) (-1.255,2.852) (-1.259,2.862) (-1.263,2.873)};
\draw[green!50!black,line width=1pt] plot[smooth] coordinates{(-1.367,2.989) (-1.381,2.999) (-1.396,3.008) (-1.411,3.018) (-1.428,3.027) (-1.445,3.036) (-1.463,3.044) (-1.481,3.053) (-1.501,3.061) (-1.521,3.068) (-1.541,3.076) (-1.562,3.083) (-1.584,3.090) (-1.607,3.096) (-1.630,3.102) (-1.653,3.107) (-1.677,3.112) (-1.701,3.117) (-1.726,3.121) (-1.751,3.125) (-1.776,3.128) (-1.802,3.130) (-1.828,3.132) (-1.855,3.134) (-1.881,3.135) (-1.908,3.135) (-1.935,3.135) (-1.962,3.134) (-1.989,3.133) (-2.016,3.132) (-2.044,3.129) (-2.071,3.126) (-2.098,3.123) (-2.125,3.119) (-2.153,3.115) (-2.180,3.110) (-2.207,3.104) (-2.234,3.098) (-2.260,3.092) (-2.287,3.085)};
\node[black,anchor=west] at (0.149,1.684) {$0$};
\node[green!50!black,anchor=east] at (-3.328,1.800) {$-8$};
\end{tikzpicture}
\caption{Left: A Kirby diagram of $X_{1,0}$: a $1$--handle, and a $2$--handle with framing $-8$ attached along a left-handed
trefoil, which runs four times over the $1$--handle in the same direction.  Right: The corresponding
surgery diagram of $Y_1$, with coefficient $0$ on the unknotted component
and $-8$ on the trefoil component.}
\label{fig:X0d1-trefoil}
\end{figure}
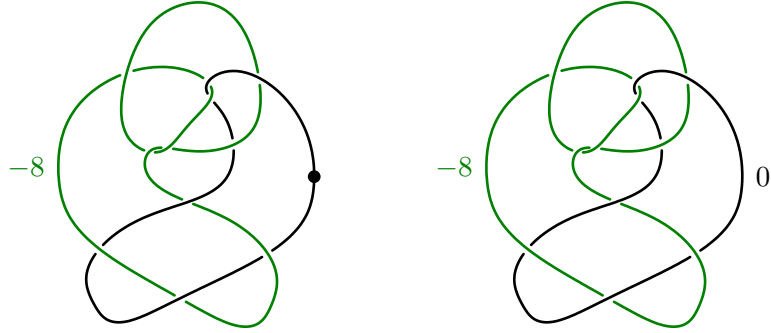

This is the analogue of the description of $Y$ as $(0,-8)$--surgery on $7^2_4$ in
\cref{prop:baykur-surgery}.  The link of \cref{fig:X0d1-trefoil} is L11n201 in the
notation of SnapPy, see \cref{subsec:computer-X10}.

\FloatBarrier
\section{Computer checks}\label{sec:computer}

In this section we collect computations carried out with SnapPy \cite{SnapPy},
Regina \cite{Regina} and the Kirby calculator KLO \cite{KLO}.  None of them is used in the proofs of the results
above.  They confirm some of these results independently, and they give
additional information that we could not obtain by hand.  The scripts, their
input data and their output are included as ancillary files.

\subsection{Baykur's manifold}\label{subsec:computer-Y}
\Cref{prop:baykur-surgery,prop:baykur-seifert} can also be confirmed by
computer.  In SnapPy the link $7^2_4$ is $\mathrm{L7a3}$, and Dehn filling its
exterior with coefficient $0$ on $U$ and $-8$ on $J$ gives a closed manifold
with $H_1\cong\mathbb Z\oplus\mathbb Z/8$.  Regina recognizes the resulting
triangulation, after simplification, as the Seifert fibered space
$\mathrm{SFS}\,[\mathrm{KB}/n2\colon(2,1)]$, that is, Seifert fibered over the
Klein bottle with a single exceptional fiber of type $(2,1)$.  Regina gives the
same answer for the boundary of the Lefschetz fibration $X_0$ itself, computed
from the Kirby diagram of \cref{fig:baykur-kirby}. Since these Seifert invariants determine the manifold up to
homeomorphism, this confirms, independently of the handle calculus, that $Y$ is
$(0,-8)$--surgery on $7^2_4$, as unoriented manifolds.  In the diagram of L7a3
in SnapPy the trefoil component has three negative crossings, so it is the
left-handed trefoil, and the link of \cref{fig:final-kirby} is L7a3 itself rather
than its mirror image.  The orientation is
supplied by the handle calculus.

\subsection{Distinguishing members of the families}\label{subsec:computer-families}
Within each family, some of the $3$--manifolds can be distinguished by computer.
The fundamental group of $Y_c$ has a presentation with generators
$x_1,\ldots,x_5$ and one relation for each hole $\delta_k$, which identifies an arc
from $\ast$ to $\delta_k$ with its image under the monodromy $t_cB$.  Counting the
homomorphisms from these groups to the symmetric group $S_3$ and to the dihedral
group of order ten gives $16$ and $36$ for $Y_m$ with $m$ odd and $22$ and $56$
for $Y_m$ with $m$ even, for $m\le4$, and $52$ and $76$, $28$ and $56$, $70$ and
$56$, $28$ and $76$, and $52$ and $156$ for $Y'_m$ with $m=1,2,3,4,5$,
respectively.  In particular, $Y_1$ and $Y_2$ are not homeomorphic, and
$Y'_1,\ldots,Y'_5$ are pairwise non-homeomorphic.

\subsection{The Kirby diagrams of \texorpdfstring{$X_{1,0}$}{X\_\{1,0\}}}\label{subsec:computer-X10}
The handle moves of \cref{rem:X0d1}(a) were carried out twice independently, and
we checked that the two resulting links are isotopic to the link of
\cref{fig:X0d1}.  For each of them, SnapPy finds an orientation-preserving
isometry between its complement and the complement of the link of
\cref{fig:X0d1}, which matches the components and maps meridians to meridians.
By Dehn filling along the meridians, such an isometry extends to an
orientation-preserving homeomorphism of $S^3$ carrying one link onto the other,
and the framing coefficients, which are measured relative to the Seifert
framings, agree.  Hence the Kirby diagrams are isotopic.  Similarly, the
complement of the link of \cref{fig:X0d1-trefoil} is isometric to that of
L11n201, by an orientation-preserving isometry taking meridians to meridians,
so the two links are isotopic.  As a further check, the boundary of the
diagram of \cref{fig:X0d1-trefoil}, with framing $-8$, has the same first
homology, the same numbers of homomorphisms to $S_3$ and to the dihedral group
of order ten, and the same homology of the connected covers of degree at most
four as $Y_1$, computed from the Lefschetz fibration.

\subsection{Kirby calculus by computer}\label{subsec:computer-KLO}
All the handle moves in this paper, in the proofs of
\cref{prop:baykur-surgery,thm:klein-main,prop:W0,rem:X0d1}, were also carried out with the Kirby
calculator KLO \cite{KLO}, independently of the hand-drawn sequences.  In each
case the computer computation ends with the same framed link as the hand-drawn
sequence, namely the link on the left of \cref{fig:final-kirby}, passing through
the diagram of \cref{fig:X0-DTK}, and the links of
\cref{fig:klein-kirby,fig:W0-kirby,fig:X0d1,fig:X0d1-trefoil}, respectively, the last two
for parts (a) and (b) of \cref{rem:X0d1}.  The diagram of \cref{fig:X0-DTK} was
compared with the corresponding KLO frame as in \cref{subsec:computer-X10}, and
the diagram of \cref{fig:X0d1-trefoil} is drawn from the last KLO frame.  The
sequences of KLO frames, showing the slide bands and the framings, are included
as ancillary files.

\subsection{Software and ancillary files}\label{subsec:software}
The computations were carried out with SnapPy~3.3.2 and Regina~7.4, and the
Kirby calculus with the Kirby calculator KLO \cite{KLO}.  The ancillary files are
organized in one folder for each of
\cref{prop:baykur-surgery,thm:klein-main,prop:W0,rem:X0d1}
and \cref{subsec:computer-Y,subsec:computer-families,subsec:computer-X10},
each with a description of its contents.  The KLO frames of
\cref{subsec:computer-KLO} are in the folders of the corresponding propositions.

\end{document}